\documentclass[11pt, reqno]{amsart}
\usepackage{amssymb,latexsym,amsmath,amsfonts,mathdots,enumitem}
\usepackage{latexsym}
\usepackage[mathscr]{eucal}
\usepackage{soul,xcolor}
\usepackage{colortbl,xcolor}
\usepackage{lmodern}
\usepackage{sansmathaccent}
\usepackage[utf8]{inputenc}
\usepackage{tikz}
\usetikzlibrary{shapes,arrows}
\usetikzlibrary{matrix,calc,shapes,arrows,positioning}

\newcommand{\rvline}{\hspace*{-\arraycolsep}\vline\hspace*{-\arraycolsep}}
\numberwithin{equation}{section}
\theoremstyle{plain}

\newtheorem{thm}{Theorem}[section]

\newtheorem{cor}[thm]{Corollary}
\newtheorem{lem}[thm]{Lemma}
\newtheorem{prop}[thm]{Proposition}
\theoremstyle{definition}
\newtheorem{defn}[thm]{Definition}
\newtheorem{rem}[thm]{Remark}

\numberwithin{equation}{section}
\def\A{{\mathcal A}}

\def\Q{\mathbb{Q}}

\def\a{\alpha}
\def\b{\beta}

\def\Re{\operatorname{Re}}

\def\d{\sum}

\def\beq{\begin{eqnarray}}
\def\eeq{\end{eqnarray}}
\def\beqa{\begin{eqnarray*}}
\def\eeqa{\end{eqnarray*}}

\def\beqn{\begin{equation}}
\def\eeqn{\end{equation}}

\def\mg#1{}

\renewcommand{\epsilon}{\varepsilon}
\renewcommand{\phi}{\varphi}

\begin{document}
\title[Function Theory and necessary conditions for a Schwarz lemma]{Function Theory and necessary conditions for a Schwarz lemma related to $\mu$-Synthesis Domains}
\author{Dinesh Kumar Keshari, Shubhankar Mandal and Avijit Pal}
\address[ D. K. Keshari]{School of Mathematical Sciences, National Institute of Science Education and Research Bhubaneswar, An OCC of Homi Bhabha National Institute, Jatni, Khurda,  Odisha-752050, India}
\email{dinesh@niser.ac.in}
\address[S. Mandal]{Department of Mathematics, IIT Bhilai, 6th Lane Road, Jevra, Chhattisgarh 491002}
\email{S. Mandal:shubhankarm@iitbhilai.ac.in}

\address[A. Pal]{Department of Mathematics, IIT Bhilai, 6th Lane Road, Jevra, Chhattisgarh 491002}
\email{A. Pal:avijit@iitbhilai.ac.in}



\begin{abstract}
  

A subset of $\mathbb{C}^7$ (respectively, of $\mathbb{C}^5$) associated with the structured singular value $\mu_E$, defined on $3 \times 3$ matrices, is denoted by $G_{E(3;3;1,1,1)}$ (respectively, by $G_{E(3;2;1,2)}$). In control engineering, the structured singular value $\mu_E$ plays a crucial role in analyzing the robustness and performance of linear feedback systems.

We characterize the domain $G_{E(3;3;1,1,1)}$ and its closure $\Gamma_{E(3;3;1,1,1)}$, and employ realization formulas to describe both. The domain $G_{E(3;3;1,1,1)}$ and its closure are neither circular nor convex; however, they are simply connected. We provide an alternative proof of the polynomial and linear convexity of  $\Gamma_{E(3;3;1,1,1)}$. Furthermore, we establish necessary conditions for a Schwarz lemma on the domains $G_{E(3;3;1,1,1)}$ and $G_{E(3;2;1,2)}$, and describe the relationships between these two domains as well as between their closed boundaries.
\end{abstract}
\maketitle
\vskip-.5cm

\section{{\bf{Introduction}}}

Let  $\Omega_1,\Omega_2\subset \mathbb C^n$ be domains in $\mathbb C^n.$  Let $\mathcal{O}\left(\Omega_{1}, \Omega_{2}\right)$  be the set of all holomorphic functions from $\Omega_{1} $ to $\Omega_{2}$.
We recall the definition of $\mu_{E}$  from \cite{aj}. 
Let $\mathcal M_{n\times n}(\mathbb{C})$ be the set of all $n\times n$ complex matrices and  $E$ be a linear subspace of $\mathcal M_{n\times n}(\mathbb{C}).$ We define the function $\mu_{E}: \mathcal M_{n\times n}(\mathbb{C}) \to [0,\infty)$ as follows:
\begin{equation}\label{mu}
\mu_{E}(A):=\frac{1}{\inf\{\|X\|: \,\ \det(1-AX)=0,\,\, X\in E\}},\;\; A\in \mathcal M_{n\times n}(\mathbb{C})
	\end{equation}
with the understanding that $\mu_{E}(A):=0$ if $1-AX$ is  nonsingular for all $X\in E.$  Here $\|\cdot\|$ denotes the operator norm.

We recall the definition of generalised tetrablock from \cite{Pawel}. Let us consider positive integers $n\geq 2,s\leq n$ and $r_1,\ldots,r_s$ with $\sum_{i=1}^{s}r_i=n.$ Let
$$A(r_1,\ldots,r_s):=\{0,1,2,\ldots,r_1\}\times\{0,1,2,\ldots,r_2\}\times \ldots \times\{0,1,2,\ldots,r_s\}\setminus\{(0,\ldots,0)\}.$$ On the set $A(r_1,\ldots,r_s)$ we introduce the following order: \\for $\alpha=(\alpha_{1},\dots,\alpha_{s}),\beta=(\beta_{1},\dots\beta_{s})\in A(r_1,\dots,r_s)$ we say that $$\alpha<\beta \,\, \text{if and only if}\,\, \alpha_{j_{o}}<\beta_{j_{o}}, \text{where} \,\, j_{o}=\operatorname{max}\{j:\alpha_{j}\neq\beta_{j}\}.$$ 
 For $\beta=(\beta_1,\ldots,\beta_s)\in A(r_1,\ldots,r_s)$ and $\textbf{z}=(z_1,\ldots,z_s)\in \mathbb C^s$, we set $|\beta|=\sum_{j=1}^{s}\beta_{j}$ and $\textbf{z}^{\beta}:=z_1^{\beta_1}\ldots z_s^{\beta_s}.$ We arrange the elements of $A(r_1,\ldots,r_s)$ in ascending order, namely, $\alpha^{(1)}<\dots<\alpha^{(N)},$ where  $N:=\prod _{j=1}^{s}(r_j+1)-1.$ For ${\textbf{x}=(x_{1},\dots,x_{N})\in \mathbb{C}^{N}}$ and ${\bold z=(z_{1},\dots,z_{s})\in \mathbb{C}^{s}},$ we define 
	  \begin{equation}\label{Rz1}
	  	R^{(n;s;r_1,\ldots,r_s)}_{\textbf{x}}(\bold z):= 1+\sum_{j=1}^{N} (-1) ^{|\alpha^{(j)}|}x_{j}{\textbf z}^{\alpha^{(j)}}
	  \end{equation}
and for $r>0$, $$\mathcal  A_{(n;s;r_1,\ldots,r_s)}^{(r)}:=\{(x_1,\ldots,x_N)\in \mathbb C^N:R^{(n;s;r_1,\ldots,r_s)}_{\textbf{x}}(\bold z)\neq 0~{\rm{for~all}}~\bold z\in r{\bar{\mathbb D}}^s\}$$ and 
$$\mathcal  B_{(n;s;r_1,\ldots,r_s)}^{(r)}:=\{(x_1,\ldots,x_N)\in \mathbb C^N:R^{(n;s;r_1,\ldots,r_s)}_{\textbf{x}}(\bold z)\neq 0~{\rm{for~all}}~\bold z\in r{\mathbb D}^s\}.$$We refer to the set $\mathcal  A_{(n;s;r_1,\ldots,r_s)}^{(1)}$ as generalized tetrablock. 

Let  $E(n;s;r_{1},\dots,r_{s})\subset \mathcal M_{n\times n}(\mathbb{C})$ be the vector subspace  consisting of  block diagonal matrices, namely,
\begin{equation}\label{ls}
    	E=E(n;s;r_{1},\dots,r_{s}):=\{\operatorname{diag}[z_{1}I_{r_{1}},\ldots,z_{s}I_{r_{s}}]\in \mathcal M_{n\times n}(\mathbb{C}): z_{1},\ldots,z_{s}\in \mathbb{C}\},
\end{equation}
 where $\sum_{i=1}^{s}r_i=n.$ Let $$\Omega^{(r)}_{E(n;s;r_{1},\dots,r_{s})}:=\{A\in \mathcal M_{n\times n}(\mathbb{C}):\mu_{E(n;s;r_{1},\dots,r_{s})}(A)<\frac{1}{r}\}.$$ Let $$\mathcal I^{j}:=\{(i_1,\ldots,i_j)\in {\mathbb N}^j:1\leq i_1<\ldots<i_j\leq n, j\leq n\}.$$ For every $\alpha\in A(r_1,\ldots,r_s)$, we define $I^{|\alpha|}_{\alpha}$ as follows:
   \begin{align}
 \mathcal I^{|\alpha|}_{\alpha}:\nonumber&=\{(i_1,\ldots,i_{|\alpha|})\in \mathcal I^{|\alpha|}:r_1+\ldots+r_{j-1}+1\leq i_{(\alpha_1+\ldots+\alpha_{j-1}+1)}\\&<\ldots <i_{(\alpha_1+\ldots+\alpha_{j})}\leq r_1+\ldots+r_{j}, j=1,\ldots,s\},\end{align} where $r_0:=0$ and $\alpha_0=0.$ Note that $\cup_{j=1}^{n}\mathcal I^{j}=\cup_{j=1}^{N}\mathcal I_{\alpha^{(j)}}^{|\alpha^{(j)}|}$ and $\mathcal  I_{\alpha}^{|\alpha|}\cap \mathcal I_{\beta}^{|\beta|}=\emptyset$ for $\alpha \neq \beta.$ For $I\in \mathcal I^{j}$ and $A\in \mathcal M_{n\times n}(\mathbb{C}),$ let $A_{I}$ denotes the $j\times j$ submatrix of $A$ whose rows and columns are indexed by $I.$ We define a polynomial map $\pi_{E(n;s;r_{1},\dots,r_{s})}: \mathcal M_{n\times n}(\mathbb{C}) \to \mathbb C^N$ as follows:
   $$\pi_{E(n;s;r_{1},\dots,r_{s})}(A):=\big(\sum_{I\in \mathcal I_{\alpha^{(1)}}^{|\alpha^{(1)}|}} \det A_I,\ldots,\sum_{I\in \mathcal I_{\alpha^{(N)}}^{|\alpha^{(N)}|}} \det A_I\big).$$ Define 
  $$G_{E{(n;s;r_1,\ldots,r_s)}}^{(r)}:=\{\pi_{E(n;s;r_{1},\dots,r_{s})}(A):\mu_{E(n;s;r_{1},\dots,r_{s})}(A)<\frac{1}{r}\}$$ and 
   $$\Gamma_{E{(n;s;r_1,\ldots,r_s)}}^{(r)}:=\{\pi_{E(n;s;r_{1},\dots,r_{s})}(A):\mu_{E(n;s;r_{1},\dots,r_{s})}(A)\leq\frac{1}{r}\}.$$   
Observe that $G_{E{(n;s;r_1,\ldots,r_s)}}^{(r)}$ is an open set and $\Gamma_{E{(n;s;r_1,\ldots,r_s)}}^{(r)}$ is the closure of $G_{E{(n;s;r_1,\ldots,r_s)}}^{(r)}.$ Here we use the notation  $G_{E{(n;s;r_1,\ldots,r_s)}}^{(1)}=G_{E{(n;s;r_1,\ldots,r_s)}}$ and $\Gamma_{E{(n;s;r_1,\ldots,r_s)}}^{(1)}=\Gamma_{E{(n;s;r_1,\ldots,r_s)}}.$ The above definitions was proposed by P. Zapalowski  \cite{Pawel}. However, if $$E=E(n;1;n)=\{zI_n\in \mathcal M_{n\times n}(\mathbb{C}):z\in \mathbb{C}\},$$ then the domain associated with this $\mu_{E}$-synthesis problem represents  \textit{symmetrized polydisc} \cite{Ccostara}. In particular, for $n=2$ the associated domain is \textit{symmetrized bidisc} \cite{JAY, JANY,JAgler,ay,ay1}. Moreover, note that if $$E=E(n;2;1,n-1)=\{\operatorname{diag}(z_1,z_2I_{n-1})\in \mathcal M_{n\times n}(\mathbb{C}):z_1,z_2\in \mathbb{C}\},$$ then the domain connected with this $\mu_{E}$-synthesis problem is $\mu_{1,n}$-\textit{quotients} \cite{bha}. The domain tetrablock, however, is represented by the $\mu_{1,2}$-\textit{quotients}. If $$E=E(3;3,1,1,1)=\{\operatorname{diag}(z_1,z_2,z_3)\in \mathcal M_{3\times 3}(\mathbb{C}):z_1,z_2,z_3\in \mathbb{C}\},$$ then from Proposition $3.3$ \cite{Pawel} we have
\begin{align}\label {G}G_{E{(3;3;1,1,1)}}^{(r)}:\nonumber &=\{\textbf{w}=(w_1=a_{33},w_2=a_{22},w_3=w_1w_2-a_{23}a_{32},w_4=a_{11}, w_5=w_1w_4-a_{13}a_{31}, \\&w_6=w_2w_4-a_{12}a_{21}, w_7=\det A): A\in \mathcal M_{3\times 3}(\mathbb C) ~{\rm{and }}~\mu_{E(3;3;1,1,1)}(A)<\frac{1}{r}\}.\end{align}

The  following lemma is  simple to ascertain. Consequently, we omit the proof.
\begin{lem}\label{AAB}
			Suppose $A\in \mathcal M_{3\times 3}(\mathbb C).$ Then $$a_{ii}=0~{\rm{for}}~i=1,2,3, a_{11}a_{22}-a_{12}a_{21}=0,a_{11}a_{33}-a_{13}a_{31}=0,\\a_{22}a_{33}-a_{23}a_{32}=0,~{\rm{and}}\operatorname{det}A=0$$ if and only if one of the following condition holds:
			\begin{enumerate}
				\item $a_{ii}=0$ for $i=1,2,3$, $a_{21}=0=a_{13}=a_{23},a_{12}\neq 0,a_{31}\neq 0$ and $a_{32}\neq 0.$
				
				\item $a_{ii}=0$ for $i=1,2,3$, $a_{12}=0=a_{31}=a_{32},a_{21}\neq 0,a_{13}\neq 0$ and $a_{23}\neq 0.$
				
				\item $a_{ii}=0$ for $i=1,2,3$, $a_{21}=0=a_{31}=a_{32},a_{12}\neq 0,a_{13}\neq 0$ and $a_{23}\neq 0.$
				
				\item $a_{ii}=0$ for $i=1,2,3$, $a_{12}=0=a_{13}=a_{23},a_{21}\neq 0,a_{31}\neq 0$ and $a_{32}\neq 0.$
				
				\item $a_{ii}=0$ for $i=1,2,3$, $a_{21}=0=a_{31}=a_{23},a_{12}\neq 0,a_{13}\neq 0$ and $a_{32}\neq 0.$
				
				\item $a_{ii}=0$ for $i=1,2,3$, $a_{12}=0=a_{13}=a_{32},a_{21}\neq 0,a_{31}\neq 0$ and $a_{23}\neq 0.$
				
				\item $a_{ii}=0$ for $i=1,2,3$, $a_{21}=a_{12}=0=a_{31}=a_{32},a_{13}\neq 0$ and $a_{23}\neq 0.$
				
				\item $a_{ii}=0$ for $i=1,2,3$, $a_{21}=a_{12}=0=a_{31}=a_{23},a_{13}\neq 0$ and $a_{32}\neq 0.$
				
				\item  $a_{ii}=0$ for $i=1,2,3$, $a_{21}=a_{12}=0=a_{31}=a_{23}=a_{13}$ and $a_{32}\neq 0.$
				
				\item  $a_{ii}=0$ for $i=1,2,3$, $a_{21}=a_{12}=0=a_{31}=a_{32}=a_{13}$ and $a_{23}\neq 0.$
				
				\item $a_{ij}=0$ for all $i,j=1,2,3.$
			\end{enumerate}
			
		\end{lem}
		\begin{lem}\label{M_3}
			Let $A\in \mathcal M_{3\times 3}(\mathbb C).$ Then 
			
			\begin{equation}\label{ade}
			\begin{aligned}
			a_{ii}=0~{\rm{for}}~i=1,2,3, a_{11}a_{22}-a_{12}a_{21}=0,a_{11}a_{33}-a_{13}a_{31}=0, a_{22}a_{33}-a_{23}a_{32}=0,~{\rm{and}}\operatorname{det}A=0 \end{aligned}\end{equation} if and only if when $\mu_{E(3;3;1,1,1)}(A)=0$.
		\end{lem}
		\begin{proof}

	Suppose that given equation \eqref{ade} holds. By  Lemma \ref{AAB}, we have  $a_{ii}=0$ for $i=1,2,3$, $a_{21}=0=a_{13}=a_{23},a_{12}\neq 0,a_{31}\neq 0$ and $a_{32}\neq 0$ and consequently, we get $$\det(I-AX)=1~~{\rm{for~~ all }}~~X\in E(3;3;1,1,1),$$ which implies $\mu_{E(n;n;1,\dots,1)}(A)=0.$ 		
			
Assume, on the other hand, that $\mu_{E(3;3;1,1,1)}(A)=0.$ This states that there is no $X\in E(3;3;1,1,1)$ that makes $I-AX$ is singular, that is, \begin{equation}\label{detA}\operatorname{det}(I-AX)\neq0\,\, \text{for all}\,\, X\in E(3;3;1,1,1).\end{equation} 
			Note that 
			\small{\begin{align}\label{A33}
					\det(I-AX)\nonumber&=1-a_{11}z_1-a_{22}z_2+(a_{11}a_{22}-a_{12}a_{21})z_1z_2-a_{33}z_3\\&+(a_{11}a_{33}-a_{13}a_{31})z_1z_3+(a_{33}a_{22}-a_{32}a_{23})z_2z_3-\operatorname{det}Az_1z_2z_3.
			\end{align}}

As the polynomial in \eqref{A33} has no zeros in $\mathbb C^3$ and hence it is  a non zero constant. Thus we have  $$a_{ii}=0~{\rm{for}}~i=1,2,3, a_{11}a_{22}-a_{12}a_{21}=0,a_{11}a_{33}-a_{13}a_{31}=0,\\a_{22}a_{33}-a_{23}a_{32}=0~{\rm{and}}\operatorname{det}A=0.$$ This completes the proof.\end{proof}
		
\begin{lem}\label{331}
			Let $A=((a_{ij}))_{i,j=1}^{3}$ and $B=J_1AJ_1,$ where $J_1=\left(\begin{smallmatrix} 0 &0 & 1\\0 & 1 & 0\\ 1& 0& 0\end{smallmatrix}\right).$  Then $$\mu_{E(3;3;1,1,1)}(A)=\mu_{E(3;3;1,1,1)}(B).$$

		\end{lem}
		\begin{proof}
			The proof of the lemma involves two cases:
			
			\noindent$\textbf{Case 1}$: Assume that $\mu_{E(3;3;1,1,1)}(A)=0$. From Lemma \ref{M_3}, it follows that 
			$w_i=0$ for $1\leq i\leq 7,$ where \small{$$w_1=a_{33},w_2=a_{22},w_3=w_1w_2-a_{23}a_{32},w_4=a_{11}, w_5=w_1w_4-a_{13}a_{31}, w_2w_4-a_{12}a_{21}=w_6, w_7=\det A.$$ } Observe that $B= ((b_{ij}))_{i,j=1}^{3}=J_1AJ_1=\left(\begin{smallmatrix} a_{33} & a_{32} & a_{31}\\a_{23} & a_{22} & a_{21}\\ a_{13} & a_{12} & a_{11}\end{smallmatrix}\right).$ Clearly,
			\begin{align*} \label{muEE} 
b_{33}=a_{11}=0, b_{22}=a_{22}=0, b_{33}b_{22}-b_{23}b_{32}=a_{11}a_{22}-a_{12}a_{21}=0, b_{11}=a_{33}=0\\
b_{33}b_{11}-b_{13}b_{31}=a_{11}a_{33}-a_{13}a_{31}=0, b_{11}b_{22}-b_{12}b_{21}=a_{22}a_{33}-a_{23}a_{32}=0, \det B=\det A=0.
			\end{align*} 
			Hence from Lemma \ref{M_3}, we have $\mu_{E(3;3;1,1,1)}(B)=0.$ 
			
			$\textbf{Case 2}$: Assume that $\mu_{E(3;3;1,1,1)}(A)\neq0$.  Then by definition we have
			\begin{align*}
				\frac{1}{\mu_{E(3;3;1,1,1)}(A)}&=\inf\{\|X\|: X\in E(3;3;1,1,1), I-AX~{\rm{is~singular}}\}\\&=\inf\{\|X\|: X\in E(3;3;1,1,1), I-J_1BJ_1X~{\rm{is~singular}}\}\\&=\inf\{\|X\|: X\in E(3;3;1,1,1), I-BJ_1XJ_1~{\rm{is~singular}}\}\\&=\inf\{\|Y\||: Y=J_1XJ_1\in E(3;3;1,1,1), I-BY~{\rm{is~singular}}\}\\&=\frac{1}{\mu_{E(3;3;1,1,1)}(B)}
			\end{align*}
			This completes the proof.
		\end{proof}

The proof of the following lemma is easy to verify. We therefore omit the  proof.  
\begin{lem}\label{meq}
Let $(x_1,x_2,\ldots, x_7)\in \mathbb{C}^7$. Then there exists a $3\times 3$ matrix $A=((a_{ij}))^3_{i,j=1}$ such that
$x_1=a_{11}, x_2=a_{22}, x_3=a_{11}a_{22}-a_{12}a_{21}, x_4=a_{33},$ $x_5=a_{11}a_{33}-a_{13}a_{31},, x_6=a_{22}a_{33}-a_{23}a_{32}$ $~{\rm{and}}~x_7=\det A$. 
\end{lem}

The following proposition is due to P. Zapolowski \cite{Pawel}.

\begin{prop}[Proposition $3.3$, \cite{Pawel}]\label{pawel}
$\mathcal  A_{(3;3;1,1,1)}^{(r)}=G_{E(3;3;1,1,1)}^{(r)}.$

\end{prop}
Let $A=((a_{ij}))_{i,j=1}^{3}$ be a $3\times 3$  matrix.  Set \small{$x_1=a_{11}, x_2=a_{22}, x_3=a_{11}a_{22}-a_{12}a_{21}, x_4=a_{33},$ \\$x_5=a_{11}a_{33}-a_{13}a_{31},, x_6=a_{22}a_{33}-a_{23}a_{32}$} $~{\rm{and}}~x_7=\det A$  and $w_1=a_{33},w_2=a_{22},$\\$w_3=w_1w_2-a_{23}a_{32},w_4=a_{11}, w_5=w_1w_4-a_{13}a_{31}, w_6=w_2w_4-a_{12}a_{21}$ and $w_7=\det A.$ Set  $\textbf{x}:=(x_1,\ldots,x_7)$ and $\textbf{w}:=(w_1,\ldots,w_7).$

\begin{prop} \label{new}${\bf{w}}\in G_{E{(3;3;1,1,1)}}^{(r)}~  ({\rm{respectively }}~{\bf{w}}\in \Gamma_{E{(3;3;1,1,1)}}^{(r)}) $ if and only if $${\bf{x}}\in G_{E{(3;3;1,1,1)}}^{(r)}~ ({\rm{respectively }}~{\bf{x}}\in \Gamma_{E{(3;3;1,1,1)}}^{(r)}).$$
\end{prop}
\begin{proof}
Suppose $\textbf{w}\in G_{E{(3;3;1,1,1)}}^{(r)}$. Then there exists a $3\times 3$ complex matrix  $A=((a_{ij}))_{i,j=1}^{3}$ such that \small{$$w_1=a_{33}, w_2=a_{22}, w_3=\det \left(\begin{smallmatrix} a_{22} & a_{23}\\
					a_{32} & a_{11}
				\end{smallmatrix}\right), w_4=a_{11}, w_5=\det \left(\begin{smallmatrix}
					a_{11} & a_{13}\\
					a_{31} & a_{33}
				\end{smallmatrix}\right), w_6=\det  \left(\begin{smallmatrix}
					a_{22} & a_{21}\\
					a_{12} & a_{11}\end{smallmatrix}\right) ~{\rm{and}}~w_7=\det A$$}

with $\mu_{E(3;3;1,1,1)}(A)<\frac{1}{r}.$ Let $B=J_1AJ_1$. By Lemma \ref{331} we have $$\mu_{E(3;3;1,1,1)}(A)=\mu_{E(3;3;1,1,1)}(B)<\frac{1}{r}$$ and hence we have $\textbf{x}\in G_{E{(3;3;1,1,1)}}^{(r)}.$
 
Similarly, we deduce that $\textbf{x}\in G_{E{(3;3;1,1,1)}}^{(r)}$ implies $\textbf{w}\in G_{E{(3;3;1,1,1)}}^{(r)}$ by applying the aforementioned argument.
This completes the proof.
\end{proof}
By Proposition \ref{new}, we will use the following definitions for $G_{E{(3;3;1,1,1)}}^{(r)}$ and $\Gamma_{E{(3;3;1,1,1)}}^{(r)}$ from now onwards:
\small{\begin{equation}\label {GE}
\begin{aligned}
G_{E{(3;3;1,1,1)}}^{(r)}:=\Big \{ \textbf{x}=(x_1=a_{11}, x_2=a_{22}, x_3=a_{11}a_{22}-a_{12}a_{21}, x_4=a_{33}, x_5=a_{11}a_{33}-a_{13}a_{31},&\\ x_6=a_{22}a_{33}-a_{23}a_{32},x_7=\det A )\in \mathbb C^7: A\in \mathcal M_{3\times 3}(\mathbb C)~{\rm{and}}~\mu_{E(3;3;1,1,1)}(A)<\frac{1}{r} \Big \},
\end{aligned}
\end{equation}}
and 
\begin{equation}\label {GammaE}
\begin{aligned}
\Gamma_{E{(3;3;1,1,1)}}^{(r)}:=\Big \{\textbf{x}=(x_1=a_{11}, x_2=a_{22}, x_3=a_{11}a_{22}-a_{12}a_{21}, x_4=a_{33}, x_5=a_{11}a_{33}-a_{13}a_{31},&\\ x_6=a_{22}a_{33}-a_{23}a_{32},x_7=\det A)\in \mathbb C^7: A\in \mathcal M_{3\times 3}(\mathbb C)~{\rm{and}}~\mu_{E(3;3;1,1,1)}(A)\leq \frac{1}{r}\Big \}.
\end{aligned}
\end{equation}
This paper is organized as follows. Section $2$ deals with the characterization of the domain $G_{E(3;3;1,1,1)}.$ First, we characterise the domain by using the zero set of the polynomial analogous to the symmetrized bi-disc, symmetrized poly-disc, tetrablock and $\mu_{1,n}$-\textit{quotients} domain. We also observe that an element $\textbf{x}\in G_{E(3;3;1,1,1)}$ can be viewed as  an element of tetrablock  $G_{E(2;2;1,1)}$ (See Theorem $2.1.2$). To describe the domain $G_{E(3;3;1,1,1)},$ we consider  $3$ rational functions, namely,  $\Psi^{(1)},\Psi^{(2)}~{\rm{and}}~ \Psi^{(3)}$ as we describe in Section $2.$ It is exceedingly challenging to compute the supremum norm of $\Psi^{(i)},1\leq i\leq 3,$ over $\bar{\mathbb D}^2.$ By realization formula we show that a point  $\mathbf{x}\in G_{E(3;3;1,1,1)}$ is equivalent to the existence of a $3\times 3$ matrix   $A\in \mathcal M_{3\times 3}(\mathbb C)$  such that \small{\begin{equation}\label{rels}x_1=a_{11}, x_2=a_{22}, x_3=\det \left(\begin{smallmatrix} a_{11} & a_{12}\\
					a_{21} & a_{22}
				\end{smallmatrix}\right), x_4=a_{33}, x_5=\det \left(\begin{smallmatrix}
					a_{11} & a_{13}\\
					a_{31} & a_{33}
				\end{smallmatrix}\right), x_6=\det  \left(\begin{smallmatrix}
					a_{22} & a_{23}\\
					a_{32} & a_{33}\end{smallmatrix}\right) ~{\rm{and}}~x_7=\det A,\end{equation}}  \begin{equation}\label{tilgamma133}
					\overline{\tilde{\gamma}_1(z_2,z_3)}(1-|z_2|^2)\tilde{\gamma}_1(z_2,z_3)+\overline{\tilde{\gamma}_2(z_2,z_3)}(1-|z_3|^2)\tilde{\gamma}_2(z_2,z_3)+{\eta(z_2,z_3)}^*(I_{3}-A^*A)\eta(z_2,z_3)>0 
				\end{equation}
				and $\det\left(I_{2}-\left(\begin{smallmatrix}a_{22}&a_{23}\\a_{32} & a_{33}
				\end{smallmatrix}\right)\left(\begin{smallmatrix}z_2 & 0 \\0 &z_{3}
				\end{smallmatrix}\right)\right)\neq 0$ for all $z_2,z_3\in \bar{\mathbb D}$ (see Proposition \ref{matrix AA}).
It follows from Theorem \ref{matix A} that if \\$A\in \mathcal M_{3\times 3}(\mathbb C)$ with $\|A\|<1,$ then $\textbf{x}=(x_1,\ldots,x_7)$ given by \eqref{rels} belongs to $ G_{E(3;3;1,1,1)}$. It is important to know whether the converse of the Theorem  \ref{matix A} is true. The expression in  \eqref{tilgamma133} might be helpful in deciding whether a point $\textbf{x}=(x_1,\ldots,x_7)\in G_{E(3;3;1,1,1)}$ is equivalent to the existence of a $3\times 3$ matrix $A\in \mathcal M_{3\times 3}(\mathbb C)$ with $\|A\|<1$ such that \small{$$x_1=a_{11}, x_2=a_{22}, x_3=\det \left(\begin{smallmatrix} a_{11} & a_{12}\\
					a_{21} & a_{22}
				\end{smallmatrix}\right), x_4=a_{33}, x_5=\det \left(\begin{smallmatrix}
					a_{11} & a_{13}\\
					a_{31} & a_{33}
				\end{smallmatrix}\right), x_6=\det  \left(\begin{smallmatrix}
					a_{22} & a_{23}\\
					a_{32} & a_{33}\end{smallmatrix}\right) ~{\rm{and}}~x_7=\det A.$$} 
 Let $\mathcal E$ and $\mathcal F$ be the Hilbert spaces and let $A=((A_{ij}))_{i=1}^{2}$ be an operator from $\mathcal{E} \oplus \mathcal{H}$ to $\mathcal{F} \oplus \mathcal{K}$ and $X$ be an operator from $\mathcal{K}$ to $\mathcal{H}$ such that $I-A_{22}X$ is invertible.  The matricial Mobius transformation $ \mathcal M_{A}:\mathcal B(\mathcal K,\mathcal H)\to \mathcal B(\mathcal E,\mathcal F)$ is defined by \begin{equation}\label{rf}
	\mathcal M_{A}(X):=A_{11}+A_{12}X\left(I-A_{22}X\right)^{-1}A_{21}. 
	\end{equation}
 Let $$\mathbb D^d=\{(z_1,\ldots,z_d)\in \mathbb C^d:|z_k|<1~{\rm{for}}~k=1,\ldots,d \}.$$ We recall the defintion of Schur and Schur Agler class from \cite{BDK}. The following is the definition of $d$-variable Schur class: $$\mathcal S_{d}(\mathcal E, \mathcal F)=\{S:\mathbb D^d\to \mathcal B(\mathcal E,\mathcal F):S ~{\rm{is~analytic~on}} ~\mathbb D^d ~{\rm{and}}~S(\textbf{z}):\mathcal E \to \mathcal F ~{\rm{is~a~contraction, ~for ~all}\;\;\textbf{z}\in \mathbb{D}^d}\}.$$ Let $\textbf{T}=(T_1,\ldots, T_d)$ be a commuting tuples of strict contractions defined on a Hilbert space $\mathcal K.$ Set $\textbf{T}^n={T_1}^{n_1}\cdots {T_d}^{n_d},$ where $n=n_1+\cdots +n_d.$ The Schur-Agler class $\mathcal S\mathcal A_{d}(\mathcal E, \mathcal F)$ is the set of functions $S(\textbf{z})=\sum_{n\in \mathbb Z_{+}^{d}}S_n\textbf{z}^n$ that are holomorphic on $\mathbb D^d$ and satisfy the condition  $\|S(\textbf{T})\|\leq 1$, where {\small{$S(\textbf{T})=\sum_{n\in \mathbb Z_{+}^{d}}S_n\otimes \textbf{T}^n\in \mathcal B(\mathcal E\otimes \mathcal K, \mathcal F\otimes \mathcal K)~{\rm{and}}~S_n\in \mathcal B(\mathcal E, \mathcal F)$} (convergence~in~the~strong~operator~topology)\cite{BDK}.
 It is worth mentioning that the Schur class and Schur-Agler class are the same for $d=1$ and $d=2.$ 

\begin{thm}[Theorem $1.2$, \cite{BDK}]\label{schuragler}
Let  $S: \mathbb{D}^d \rightarrow \mathcal{B}(\mathcal{E}, \mathcal{F})$ be an analytic function. Then the following are equivalent
	\begin{enumerate}
		\item $S \in \mathcal{S}\mathcal A_{d}(\mathcal {E},\mathcal {F})$;
		\item There exists Hilbert spaces $\mathcal K_{1},\ldots, \mathcal K_{d}$ and unitary colligation $A$ of the following form 
			$$
		\mathrm{A}=\left[\begin{array}{ll}
			A_{11} & A_{12} \\
			A_{21} & A_{22}
		\end{array}\right]:\left[\begin{array}{l}
			
			\mathcal{E}\\
\mathcal{K}_1 \\\vdots\\\mathcal K_{d}		\end{array}\right] \rightarrow\left[\begin{array}{l}
			\mathcal{F}\\\mathcal{K}_1 \\\vdots\\\mathcal K_{d}
			
		\end{array}\right], 
		$$
		so that
		\begin{equation}\label{rf1}
		S(\mathbf z)=A_{11}+A_{12}P(\mathbf z) (I- A_{22}P(\mathbf z))^{-1} A_{21} ,\,\,\text{for}\,\,\mathbf{z} \in \mathbb{D}^d	, 
		\end{equation}
		where
		$$
		P(\mathbf z)=\begin{bmatrix}z_{1} I_{\mathcal K_1}& & \\ & \ddots & \\ & & z_{d}I_{\mathcal K_d}\end{bmatrix}.$$
		
		 \end{enumerate}
	\end{thm}
Above theorem is also valid for $d$-variable Schur class $\mathcal S_{d}(\mathcal E, \mathcal F)$ [\cite{BallF}, Theorem $1.3$]. Observe that $S(\textbf{z})$ is the specific form of the matricial Mobius transformation  $\mathcal M_{A}$, if we consider $\mathcal H=\mathcal K_1\oplus \ldots\oplus K_d=\mathcal K$ and $P(\textbf{z})=X,$ as determined by Theorem \ref{schuragler}. This formula is crucial for understanding the domain $G_{E(2;2;1,1)}.$ 
 Let $A=((a_{ij}))_{i,j=1}^{2}\in \mathcal M_{2\times2}(\mathbb C),\mathcal E=\mathbb C=\mathcal F=\mathcal K_1$ and $X=z$ for all $z\in \mathbb C$ with $1-a_{22}z\neq 0.$  Then from \eqref{rf}, we have 
\begin{align}\label{MAZ}
\mathcal M_{A}(z)&\nonumber=a_{11}+a_{12}z\left(1-a_{22}z\right)^{-1}a_{21}\\&\nonumber=\frac{a_{11}-z(a_{22}a_{11}-a_{12}a_{21})}{1-a_{22}z}\\&=\nonumber \frac{x_1-zx_3}{1-x_2z}\\&=\Psi(z,(x_1,x_2,x_3)),
\end{align}
 where $x_1=a_{11},x_2=a_{22}~{\rm{and}}~x_3= \det A.$ This shows that the matricial Mobius transformation is basically the function $\Psi$ which is defined in \cite{Abouhajar}. By computing the norm of $\Psi$, Abouhajar, White and Young \cite{Abouhajar} showed that a point $(x_1,x_2,x_3)\in G_{E(2;2;1,1)}$ if and  only if there exists a $2\times 2$ matrix $A=((a_{ij}))_{i=1}^{2}$ with $\|A\|<1$ such that $x_ 1=a_{11},x_2=a_{22}~{\rm{and}}~x_3=\det A.$   
 For $A=((a_{ij}))_{i,j=1}^{2}\in \mathcal M_{2\times 2}(\mathbb C)$, we define $\gamma$ and $\eta$ by
	\begin{align}\label{gaaa}\gamma(z)=\left(1-a_{22} z\right)^{-1}a_{21}, ~{\rm{and}}~\eta(z)=\left(\begin{smallmatrix}1\\ z\gamma(z)		\end{smallmatrix}\right), 
	\end{align} 
for all $z\in\mathbb C$ with $(1-a_{22}z)\neq 0.$ Note that 
\begin{align}\label{AAAA norm12}
			1-\overline{\mathcal M_{A}(z)}\mathcal M_{A}(z)&=\overline{\gamma}(z)(1-|z|^2)\gamma (z)+\eta(z)^*(I_{2}-A^*A)\eta(z).
\end{align}
It yields from \eqref{MAZ} and \eqref{AAAA norm12} that a point $(x_1,x_2,x_3)\in G_{E(2;2;1,1)}$ if and only if there exist a $2\times 2$ matrix $A\in \mathcal M_{2\times 2}(\mathbb C)$ such that $x_{1}=a_{11},x_{2}=a_{22},x_{3}=\operatorname{det}A$,
\begin{equation}\label{tetragamma}
\overline{\gamma}(z)(1-|z|^2)\gamma (z)+\eta(z)^*(I_{2}-A^*A)\eta(z)>0~~{\rm{and}}~(1-a_{22}z)\neq 0 ~{\rm{for ~all}}~z\in\bar{\mathbb D}.
\end{equation}
Note that the function $z\to |\Psi(z,(x_1,x_2,x_3))|$ is a  continuous function on $\mathbb T.$ Since $\mathbb T$ is compact,  there exists a point $z_0\in \mathbb T$ such that \begin{equation}\label{Psi112} |\mathcal M_{A}(z_0)|=|\Psi(z_0,(x_1,x_2,x_3))|=\sup_{z\in \mathbb T}|\Psi(z,(x_1,x_2,x_3))|.\end{equation}  From \eqref{tetragamma} and \eqref{Psi112}, we notice that a point $(x_1,x_2,x_3)\in G_{E(2;2;1,1)}$ if and only if $$\eta(z_0)^*(I_{2}-A^*A)\eta(z_0)>0.$$ Furthermore, based on the characterization of tetrablock, we conclude that 
$\frac{\|A\eta(z_0)\|}{\|\eta(z_0)\|}=\|A\|.$  Thus, we think that the realization formula may  play a crucial role in determining a point $\textbf{x}=(x_1,\ldots,x_7)\in G_{E(3;3;1,1,1)}$ is equivalent to the existence of a $3\times 3$ matrix $A\in \mathcal M_{3\times 3}(\mathbb C)$ with $\|A\|<1$ such that \small{$$x_1=a_{11}, x_2=a_{22}, x_3=\det \left(\begin{smallmatrix} a_{11} & a_{12}\\
					a_{21} & a_{22}
				\end{smallmatrix}\right), x_4=a_{33}, x_5=\det \left(\begin{smallmatrix}
					a_{11} & a_{13}\\
					a_{31} & a_{33}
				\end{smallmatrix}\right), x_6=\det  \left(\begin{smallmatrix}
					a_{22} & a_{23}\\
					a_{32} & a_{33}\end{smallmatrix}\right) ~{\rm{and}}~x_7=\det A.$$} 
We also explore the relationship between the domains $G_{E(3;2;1,2)}$ and $G_{E(3;3;1,1,1)}$. We  observe that  if $A\in \mathcal M_{3\times 3}(\mathbb C)$ with $\|A\|<1,$ then $\tilde{\textbf{x}}=(x_1,x_2,x_3,y_1,y_2)$ given by \eqref{mu13qu} belongs to $ G_{E(3;2;1,2)}$. The domain $G_{E(3;3;1,1,1)}$ and its closure are neither circular nor convex; however, they are simply connected. We provide an alternative proof of the polynomial and linear convexity of  $\Gamma_{E(3;3;1,1,1)}$. We describe the relationships between these $G_{E(3;3;1,1,1)}$ and $G_{E(3;2;1,2)}$ domains as well as between their closed boundaries.

 The standard Schwarz lemma gives a necessary and sufficient condition for solving a two-point interpolation problem for analytic functions from the open unit disc $\mathbb D$ into itself. We discuss the necessary conditions for a Schwarz lemma for the domains $G_{E(3;3;1,1,1)}$ and $G_{E(3;2;1,2)}$ in Section $4$. 
  
\section{Characterization of the domain $G_{E(3;3;1,1,1)}$}

In this section we use the zero set of a polynomial to characterise the domains $G_{E(3;3;1,1,1)}$ and $\Gamma_{E(3;3;1,1,1)}$ by applying the idea of the structured singular value. This polynomial allow us to create rational functions $\Psi^{(i)}$ that map $\mathbb{\bar{D}}^2$ to $\mathbb{D}$ for $i=1,2,3.$ The supremum of these rational functions $\Psi^{(i)}$, $1\leq i\leq 3$, is attained on $\mathbb T^2$ and is bounded by $1$. For the purpose of characterising the domains $G_{E(3;3;1,1,1)}$ and $\Gamma_{E(3;3;1,1,1)}$, these rational functions are crucial. Since $\Psi^{(i)} \in \mathcal{S}_{2}(\mathbb C,\mathbb C)$, $1\leq i\leq 3$, it is expressed by the form in \eqref{rf1}. The realization formula is very much useful to express each rational functions $\Psi^{(i)}$ for $i=1,2,3$. We demonstrate that one of these three rational functions can be expressed in terms of the other two. We also discuss two version of realization formulas to describe the domains $G_{E(3;3,1,1,1)}$ and $\Gamma_{E(3;3,1,1,1)}$.
 
\subsection{Characterization by the zero set of a polynomial}
	
	 For $\beta=(\beta_1,\ldots,\beta_n)\in A(1,\ldots,1)$ and $\textbf{z}=(z_1,\ldots,z_n)\in \mathbb C^n$, we set $|\beta|=\sum_{j=1}^{n}\beta_{j}$ and $\textbf{z}^{\beta}:=z_1^{\beta_1}\ldots z_n^{\beta_n}.$  We can arrange the elements of $A(1,\ldots,1)$ as $\alpha^{(1)}<\dots<\alpha^{(N)},$ where  $N=2^{n}-1.$ For ${\textbf{x}=(x_{1},\dots,x_{N})\in \mathbb{C}^{N}}$ and ${\bold z=(z_{1},\dots,z_{n})\in \mathbb{C}^{n}},$ from \eqref{Rz1} we have 
	  \begin{equation}\label{Rz}
	  	R^{(n;n;1,\ldots,1)}_{\textbf{x}}(\bold z):= 1+\sum_{j=1}^{N} (-1) ^{|\alpha^{(j)}|}x_{j}{\textbf z}^{\alpha^{(j)}}.
	  \end{equation}

For $n=2$, we have $$R^{(2;2,1,1)}_{\textbf{x}}(\textbf{z})=1-x_1z_1-x_2z_2+x_3z_1z_2,$$ where $\textbf{x}=(x_1,x_2,x_3)$ and $\mathbf {z}=(z_1,z_2)\in\mathbb C^2.$ In \cite{Abouhajar}, it is shown that 
	  $\mathcal A^{(1)}_{(2;2,1,1)}=G_{E(2;2;1,1)}$ and $\mathcal B^{(1)}_{(2;2;1,1)}=\Gamma_{E(2;2;1,1)}.$  For $n\geq 4,$ Zapolowski \cite{Pawel} proved that $G_{E(n;n;1,\dots,1)}\subset \mathcal A^{(1)}_{(n;n;1,\dots,1)}$ and $\mathcal A^{(1)}_{(3;3;1,1,1)}=G_{E(3;3;1,1,1)}.$
Here we give an explicit proof of $$\mathcal A^{(r)}_{(3;3;1,1,1)}=G^{(r)}_{E(3;3;1,1,1)}~{\rm{and }}~\mathcal B^{(r)}_{(3;3;1,1,1)}=\Gamma^{(r)}_{E(3;3;1,1,1)},$$ for $r>0$, using the same method as in \cite{aaabou}.  	
We use the structured singular valued functional to prove the following theorem, implying that these domains evolved from the $\mu$-synthesis process.	
		\begin{thm}\label{R_3}
			For $r>0,$  $\Gamma^{(r)}_{E(3;3;1,1,1)}=\mathcal B^{(r)}_{(3;3;1,1,1)}.$ In particular, $\Gamma_{E(3;3;1,1,1)}=\mathcal B^{(1)}_{(3;3;1,1,1)}.$
		\end{thm}
		
		\begin{proof}
			
			We first show that 
			\begin{equation}\label{de1} \Gamma^{(r)}_{E(3;3;1,1,1)}\subseteq \mathcal B^{(r)}_{(3;3;1,1,1)}.\end{equation} 
			
			Let $\textbf{x}=(x_1,x_2,x_3,x_4,x_5,x_6,x_7)\in \Gamma^{(r)}_{E(3;3;1,1,1)}$, so by  \eqref{GammaE} there exits a matrix $A=((a_{ij}))_{i,j=1}^{3}$ such that \small{\begin{equation*}x_1=a_{11}, x_2=a_{22}, x_3=\det \left(\begin{smallmatrix} a_{11} & a_{12}\\
					a_{21} & a_{22}
				\end{smallmatrix}\right), x_4=a_{33}, x_5=\det \left(\begin{smallmatrix}
					a_{11} & a_{13}\\
					a_{31} & a_{33}
				\end{smallmatrix}\right), x_6=\det  \left(\begin{smallmatrix}
					a_{22} & a_{23}\\
					a_{32} & a_{33}\end{smallmatrix}\right) ~{\rm{and}}~x_7=\det A\end{equation*}} and $\mu_{E(3;3;1,1,1)}(A)\leq\frac{1}{r}$. We consider two cases. 
			
			$\textbf{Case 1}$: 
			Assume that $\mu_{E(3;3;1,1,1)}(A)\neq0$. As
			
			$$\mu_{E(3;3;1,1,1)}(A)=\frac{1}{\inf\{\|X\|: X\in E(3;3;1,1,1), (I-AX)\,\,\mbox{is singular}\}} \leq\frac{1}{r}$$ is equivalent to $$\inf\{\|X\|: X\in E(3;3;1,1,1), (I-AX)\,\,\mbox{is singular}\} \geq r.$$ Thus if $X=\begin{pmatrix}
				z_1 & 0 & 0\\
				0 & z_2 & 0\\
				0 & 0 & z_3
			\end{pmatrix}\in  E(3;3;1,1,1)$ such that $\det (I-AX)=0$ then we have $$\|X\|=\max\{|z_1|,|z_2|,|z_3|\}\geq r.$$ In other words, if  $X\in E(3;3;1,1,1)$ and $\|X\|=\max\{|z_1|,|z_2|,|z_3|\} < r$ then $\det (I-AX)\neq 0$. For $z_1,z_2,z_3 \in r\mathbb{D}$ we have 
			\begin{eqnarray*}
				0&\neq& \det (I-AX)\\
				&=& 1-a_{11}z_1-a_{22}z_2+\det \left(\begin{smallmatrix} a_{11} & a_{12}\\
					a_{21} & a_{22}
				\end{smallmatrix}\right)\, z_1z_2-a_{33}z_3\\&+&\det \left(\begin{smallmatrix} a_{11} & a_{13}\\
					a_{31} & a_{33}
				\end{smallmatrix}\right)\, z_1z_3+\det \left(\begin{smallmatrix} a_{22} & a_{23}\\
					a_{32} & a_{33}
				\end{smallmatrix}\right)\, z_2z_3-\det A\, z_1z_2z_3\\
				&=& 1-x_1z_1-x_2z_2+x_3z_1z_2-x_4z_3+x_5z_1z_3+x_6z_2z_3-x_7z_1z_2z_3\\
				&=& R_{\bf{x}}^{(3;3;1,1,1)}(\textbf{z}).
			\end{eqnarray*}
			Hence $\textbf{x}\in \mathcal{B}^{(r)}_{(3;3;1,1,1)}$.
			
			$\textbf{Case 2}$: Assume that $\mu_{E(3;3;1,1,1)}(A)=0$. From Lemma \ref{M_3}, it follows that 
			$$a_{ii}=0~{\rm{for}}~i=1,2,3, a_{11}a_{22}-a_{12}a_{21}=0,a_{11}a_{33}-a_{13}a_{31}=0,\\a_{22}a_{33}-a_{23}a_{32}=0~{\rm{and}}\operatorname{det}A=0.$$
			For $\textbf{z}=(z_1,z_2,z_3)\in \mathbb{C}^3$ we have 
			\begin{eqnarray*}
				R_{\textbf{x}}^{(3;3;1,1,1)}(\textbf{z})&=& 1-x_1z_1-x_2z_2+x_3z_1z_2-x_4z_3+x_5z_1z_3+x_6z_2z_3-x_7z_1z_2z_3\\
				&=& 1-a_{11}z_1-a_{22}z_2+\det \left(\begin{smallmatrix} a_{11} & a_{12}\\
					a_{21} & a_{22}
				\end{smallmatrix}\right)\, z_1z_2-a_{33}z_3\\&+&\det \left(\begin{smallmatrix} a_{11} & a_{13}\\
					a_{31} & a_{33}
				\end{smallmatrix}\right)\, z_1z_3+\det \left(\begin{smallmatrix} a_{22} & a_{23}\\
					a_{32} & a_{33}
				\end{smallmatrix}\right)\, z_2z_3-\det A\, z_1z_2z_3\\				&=&1\\
				&\neq& 0.
			\end{eqnarray*}
			In particular, $R_{\textbf{x}}^{(3;3;1,1,1)}(\textbf{z})\neq 0$ for all $\textbf{z}\in r\mathbb{D}^3$. Hence $\textbf{x}\in \mathcal{B}^{(r)}_{(3;3;1,1,1)}$. This completes the proof of (\ref{de1}). 
			
			Now we show that 
			\begin{equation}\label{dem2} \mathcal B^{(r)}_{(3;3;1,1,1)} \subseteq  \Gamma^{(r)}_{E(3;3;1,1,1)}.\end{equation} 
			
			Let $\textbf{x}=(x_1,x_2,x_3,x_4,x_5,x_6,x_7) \in \mathcal B^{(r)}_{(3;3;1,1,1)}$, so for $\textbf{z}=(z_1,z_2,z_3)\in r\mathbb{D}^3$ we have
			\small{\begin{eqnarray}\label{de2}
				0 &\neq& R_{\bf{x}}^{(3;3;1,1,1)}(\textbf{z}) \nonumber\\
				&=& 1-x_1z_1-x_2z_2+x_3z_1z_2-x_4z_3+x_5z_1z_3+x_6z_2z_3-x_7z_1z_2z_3.
			\end{eqnarray}}
			By Lemma \ref{meq}, there exists a matrix $A=((a_{ij}))_{i,j=1}^{3}$ such that 
			\small{\begin{eqnarray}\label{de3}
				x_1=a_{11}, x_2=a_{22}, x_3=\det\left( \begin{smallmatrix} a_{11} & a_{12}\\
					a_{21} & a_{22}
				\end{smallmatrix}\right) x_4=a_{33}, x_5=\det \left(\begin{smallmatrix}
					a_{11} & a_{13}\\
					a_{31} & a_{33}
				\end{smallmatrix}\right),\nonumber
				\\
				x_6=\det \left( \begin{smallmatrix}
					a_{22} & a_{23}\\
					a_{32} & a_{33}\end{smallmatrix}\right)~ \mbox{and} ~x_7=\det A .
			\end{eqnarray}}
			
			Thus from \eqref{de2} and \eqref{de3} we have 
			\begin{equation}\label{de4}
				\det(I-AX)=R_{\bf{x}}^{(3;3;1,1,1)}(\textbf{z})\neq 0\;\;\;\mbox{for all}\;\; \textbf{z}\in r\mathbb{D}^3.
			\end{equation}
			From \eqref{de4} follows that if $\det (I-AX)=0$, then $\max\{|z_1|,|z_2|,|z_3|\}\geq r$, that is, $$\|X\|=\|\begin{pmatrix}
				z_1 & 0 & 0\\
				0 & z_2 & 0\\
				0 & 0 & z_3
			\end{pmatrix}\|\geq r.$$ Consequently, we get $\inf\{\|X\|: X\in E(3;3;1,1,1), (I-AX)\,\,\mbox{is singular}\} \geq r$ which implies that $\mu_{E(3;3;1,1,1)}(A)\leq\frac{1}{r}$.
			Thus we have $\textbf{x}\in \Gamma^{(r)}_{E(3;3;1,1,1)}$. This completes the proof of \eqref{dem2}. 
		\end{proof}	
		\begin{cor}\label{R_33}
			$G_{E(3;3;1,1,1)}=\mathcal A^{(1)}_{(3;3;1,1,1)}.$
		\end{cor}
		\begin{proof}
			Let $\textbf{x}=(x_1,x_2,x_3,x_4,x_5,x_6,x_7)\in G_{E(3;3;1,1,1)}$. By \eqref{GE} there exits a matrix $A=((a_{ij}))_{i,j=1}^{3}$ such that \small{$$x_1=a_{11}, x_2=a_{22}, x_3=\det \left(\begin{smallmatrix} a_{11} & a_{12}\\
					a_{21} & a_{22}
				\end{smallmatrix}\right), x_4=a_{33}, x_5=\det \left(\begin{smallmatrix}
					a_{11} & a_{13}\\
					a_{31} & a_{33}
				\end{smallmatrix}\right), x_6=\det  \left(\begin{smallmatrix}
					a_{22} & a_{23}\\
					a_{32} & a_{33}\end{smallmatrix}\right) ~{\rm{and}}~x_7=\det A,$$} and $\mu_{E(3;3;1,1,1)}(A)<1$. Let $r_0 > 1$ be such that $\mu_{E(3;3;1,1,1)}(A)\leq \frac{1}{r_0}.$ For all $\textbf{z}\in r_0\mathbb{D}^3$, from Theorem \ref{R_3}, it follows that $R^{(3;3;1,1,1)}_{\textbf{x}}(\textbf{z})\neq 0$. In particular, $R^{(3;3;1,1,1)}_{\textbf{x}}(\textbf{z})\neq 0$ for all $\textbf{z}\in  \bar{\mathbb D}^3$. Thus 
			$\textbf{x}\in \mathcal{A}_{(3;3;1,1,1)}^{(1)}$. Hence $G_{E(3;3;1,1,1)}\subseteq \mathcal A^{(1)}_{(3;3;1,1,1)}$.
			
			Conversely, let $\textbf{x}=(x_1,\ldots,x_7)\in \mathcal A^{(1)}_{(3;3;1,1,1)}.$ From  Lemma \ref{meq}, there exists a matrix $A=((a_{ij}))^3_{ i, j=1}$ such that 
			\small{$$x_1=a_{11}, x_2=a_{22}, x_3=\det \left(\begin{smallmatrix} a_{11} & a_{12}\\
					a_{21} & a_{22}
				\end{smallmatrix}\right), x_4=a_{33}, x_5=\det \left(\begin{smallmatrix}
					a_{11} & a_{13}\\
					a_{31} & a_{33}
				\end{smallmatrix}\right), x_6=\det  \left(\begin{smallmatrix}
					a_{22} & a_{23}\\
					a_{32} & a_{33}\end{smallmatrix}\right) ~{\rm{and}}~x_7=\det A.$$}
			Since $\textbf{x}\in \mathcal{A}_{(3;3;1,1,1)}^{(1)}$, so for all $z\in \bar{\mathbb D}^3$, we have 
			
			\begin{eqnarray*}
				0 &\neq & R_{\bf{x}}^{(3;3;1,1,1)}(\textbf{z})\\
				&=& 1-a_{11}z_1-a_{22}z_2+\det A_{12}\, z_1z_2-a_{33}z_3+\det A_{13}\, z_1z_3+\det A_{23}\, z_2z_3-\det A\, z_1z_2z_3.
			\end{eqnarray*}
			From continuity of $R_{\bf{x}}^{(3;3;1,1,1)}(\textbf{z})$ in the variable $\textbf{z}$ and compactness of $\mathbb T\times \mathbb T \times \mathbb T$, there exists $\epsilon_0>0$ such that $R_{\bf{x}}^{(3;3;1,1,1)}(\textbf{z})\neq 0$ for all $\textbf{z} \in (1+\epsilon_0)\mathbb{D}^3$. Therefore,  $\textbf{x}$ is a member of $\mathcal B^{(1+\epsilon_0)}_3$. By Theorem \ref{R_3},  $\textbf{x}$ belongs to $\Gamma^{(1+\epsilon_0)}_{E(3;3;1,1,1)}$. From the definition of $\Gamma^{(1+\epsilon_0)}_{E(3;3;1,1,1)}$, we obtain $\mu_{E(3;3;1,1,1)}(A)\leq \frac{1}{(1+\epsilon_0)}<1$. Thus $\textbf{x}$ belongs to  $G_{E(3;3;1,1,1)}$. In light of this, $\mathcal A^{(1)}_{(3;3;1,1,1)} \subseteq G_{E(3;3;1,1,1)}$. This completes the proof. 
		\end{proof}

\subsection{Characterization by rational functions}

In the work of Abouhajar, White and Young  \cite{Abouhajar}, rational functions plays essential role for the description of the domain tetrablock. To describe the domain $G_{E(3;3;1,1,1)}$, we consider $3$ rational functions, namely, $\Psi^{(1)},\Psi^{(2)}~{\rm{and}}~~\Psi^{(3)}$. Let $J^{(i)}=\{1,2,3\}\setminus\{i\}, 1\leq i\leq 3$. For $1\leq i \leq 3$ and $j_1,j_2 \in J^{(i)}$ with $J_1< J_2$, set $$\textbf{x}_{J^{(i)}}^{\prime}:=(x_{2^{j_{1}-1}},x_{2^{j_{2}-1}},x_{(2^{j_{1}-1}+2^{j_{2}-1})}
		),\;\; \textbf{x}_{J^{(i)}}^{\prime\prime}:=(x_{2^{i-1}},x_{(2^{j_{1}-1}+2^{i-1})},x_{(2^{j_{2}-1}+2^{i-1})}
		,x_{(2^{j_{1}-1}+2^{j_{2}-1}+2^{i-1})}),$$ and  $\textbf{z}_{J^{(i)}}:=(z_{j_1},z_{j_2})$. We define the polynomials
$$\begin{aligned}
			R^{(i)}(\textbf{z}_{J^{(i)}},\textbf{x}_{J^{(i)}}^{\prime})=&1-\sum_{l\in J^{(i)}}x_{2^{l-1}}z_{l}+x_{(2^{j_{1}-1}+2^{j_{2}-1})}z_{j_{1}}z_{j_{2}},
		\end{aligned}$$ and  $$\begin{aligned}
			P^{(i)}(\textbf{z}_{J^{(i)}},\textbf{x}_{J^{(i)}}^{\prime\prime})=&x_{2^{i-1}}-\sum_{l\in J^{(i)}}x_{(2^{l-1}+2^{i-1})}z_{l}+x_{(2^{j_{1}-1}+2^{j_{2}-1}+2^{i-1})}z_{j_{1}}z_{j_{2}}.
		\end{aligned}$$
 If $\textbf{x}_{J^{(i)}}^{\prime}\in \Gamma_{E(2;2;1,1)}$, then $R^{(i)}(\textbf{z}_{J^{(i)}},\textbf{x}_{J^{(i)}}^{\prime})\neq 0$, $1\leq i \leq 3$. We define $$\Psi^{(i)}(\textbf{z}_{J^{(i)}},\textbf{x})=\frac{P^{(i)}(\textbf{z}_{J^{(i)}},\textbf{x}_{J^{(i)}}^{\prime\prime})}{R^{(i)}(\textbf{z}_{J^{(i)}},\textbf{x}_{J^{(i)}}^{\prime})}, \;\;\; 1\leq i \leq 3.$$	
		
		For example, for $i=1$, we have 
		$J^{(1)}=\{2,3\}, \textbf{z}_{J^{(1)}}=(z_2,z_3).$ From above construction, we observe that
		$$R^{(1)}(\textbf{z}_{J^{(1)}},\textbf{x}_{J^{(1)}}^{\prime})=1-x_2z_2-x_4z_3+x_6z_2z_3,$$ and $$P^{(1)}(\textbf{z}_{J^{(1)}},\textbf{x}_{J^{(1)}}^{\prime\prime})=x_1-x_3z_2-x_5z_3+x_7z_2z_3.$$ Also, from above discussion, we see that 
		$\textbf{x}_{J^{(1)}}^{\prime}=(x_{2},x_{4},x_{6}
		)$, $\textbf{x}_{J^{(1)}}^{\prime\prime}=(x_{1},x_{3},x_{5}
		,x_{7})$  and  $$\Psi^{(1)}(\textbf{z}_{J^{(1)}},\textbf{x})=\frac{x_{7}z_{2}z_{3}-x_{3}z_{2}-x_{5}z_{3}+x_{1}}{x_{6}z_{2}z_{3}-x_{2}z_{2}-x_{4}z_{3}+1}~{\rm{for}}~\textbf{x}_{J^{(1)}}^{\prime}\in \Gamma_{E(2;2;1,1)}.$$
		
		Note that $\Psi^{(1)}(\cdot,\cdot,\textbf{x})$ is a constant function $x_1$ when $x_{7}=x_{6}x_{1},\,\,x_{3}=x_{2}x_{1},\,\,x_{5}=x_{4}x_{1}.$ 
		Similarly to the above constructions, we can also describe the functions $\Psi^{(2)}(\textbf{z}_{J^{(2)}},\textbf{x})$ and $\Psi^{(3)}(\textbf{z}_{J^{(3)}},\textbf{x}).$ Also, note that for  $\textbf{x}=(x_{1},\dots,x_{7})$, we get 
			\begin{equation}\label{eq:3.0}
				\Psi^{(3)}(z,w,x_1,x_4,x_5,x_2,x_3,x_6,x_7)=\Psi^{(2)}(z,w,\textbf{x})
				=\Psi^{(1)}(z,w,x_{2},x_{1},x_{3},x_{4},x_{6},x_{5},x_{7}).
			\end{equation}

			\begin{thm}\label{phi}
		Let $\textbf{z}_{J^{(i)}}=(z_{j_1},z_{j_2})\in \mathbb{C}^{2}$ for $j_1,j_2\in J^{(i)}$, $1\leq i \leq 3$,  with $j_1<j_2.$ Then $\textbf{x}\in G_{E(3;3;1,1,1)}$ if and only if it satisfies one of the following condition:
				\begin{enumerate}
					\item $\textbf{x}_{J^{(1)}}^{\prime}\in G_{E(2;2;1,1)}$ and 
					$\|\Psi^{(1)}(\cdot,\textbf{x})\|_{H^{\infty}(\bar{\mathbb{D}}^{2})}=\|\Psi^{(1)}(\cdot,\textbf{x})\|_{H^{\infty}(\mathbb{T}^{2})}< 1$ and if  we suppose that $x_{7}=x_{6}x_{1},\,\,x_{3}=x_{2}x_{1},\,\,x_{5}=x_{4}x_{1}$ then $|x_1|<1;$
					\item $\textbf{x}_{J^{(2)}}^{\prime}\in G_{E(2;2;1,1)}$ and 
					$\|\Psi^{(2)}(\cdot,\textbf{x})\|_{H^{\infty}(\bar{\mathbb{D}}^{2})}=\|\Psi^{(2)}(\cdot,\textbf{x})\|_{H^{\infty}(\mathbb{T}^{2})}< 1$  and if  we suppose that $x_{7}=x_{5}x_{2},\,\,x_{3}=x_{2}x_{1},\,\,x_{6}=x_{4}x_{2}$ then $|x_2|<1;$
					
					\item $\textbf{x}_{J^{(3)}}^{\prime}\in G_{E(2;2;1,1)}$ and 
					$\|\Psi^{(3)}(\cdot,\textbf{x})\|_{H^{\infty}(\bar{\mathbb{D}}^{2})}=\|\Psi^{(3)}(\cdot,\textbf{x})\|_{H^{\infty}(\mathbb{T}^{2})}< 1$  and if  we suppose that $x_{7}=x_{3}x_{4},\,\,x_{6}=x_{2}x_{4},\,\,x_{5}=x_{4}x_{1}$ then $|x_4|<1.$
					
				\end{enumerate}
			\end{thm} 
			\begin{proof}
				We only prove that $\textbf{x}\in G_{E(3;3;1,1,1)}$ is equivalent to $$\textbf{x}_{J^{(2)}}^{\prime}\in G_{E(2;2;1,1)}~~{\rm{ and }}~~
				\|\Psi^{(2)}(\cdot,\textbf{x})\|_{H^{\infty}(\bar{\mathbb{D}}^{2})}=\|\Psi^{(2)}(\cdot,\textbf{x})\|_{H^{\infty}(\mathbb{T}^{2})}< 1,$$ because the proof of $(1)$ and $(3)$ will follow easily by using the following identities \begin{equation*}
					\Psi^{(3)}(z,w,x_1,x_4,x_5,x_2,x_3,x_6,x_7)=\Psi^{(2)}(z,w,\textbf{x})
					=\Psi^{(1)}(z,w,x_{2},x_{1},x_{3},x_{4},x_{6},x_{5},x_{7}).
				\end{equation*} 
				From Theorem \ref{R_3}, we get $\textbf{x}\in G_{E(3;3;1,1,1)}$ if and only if 
				 $$R^{(3;3;1,1,1)}_{\textbf{x}}(\textbf{z})=R^{(2)}(\textbf{z}_{J^{(2)}},\textbf{x}_{J^{(2)}}^{\prime})-z_{2}P^{(2)}(\textbf{z}_{J^{(2)}},\textbf{x}_{J^{(2)}}^{\prime\prime})\neq 0, \,\,\;\mbox{for}\;\; \textbf{z}_{J^{(2)}}\in \bar{\mathbb{D}}^{2},\,\,z_{2}\in \bar{\mathbb{D}},$$ that is, $\textbf{x}_{J^{(2)}}^{\prime}\in G_{E(2;2;1,1)}$ and $1\neq z_{2}\Psi^{(2)}(\textbf{z}_{J^{(2)}},\textbf{x})$ for all $\textbf{z}_{J^{(2)}}\in \bar{\mathbb{D}}^{2}$. However, $\Psi^{(2)}(\cdot,\textbf{x})$ is holomorphic on $\mathbb{D}^{2}$ and continuous on $\bar{\mathbb{D}}^{2}$. Hence by the Maximum modulus theorem, it follows that $$\max_{\textbf{z}_{J^{(2)}} \in \bar{\mathbb{D}}^{2}}|\Psi^{(2)}\left(\textbf{z}_{J^{(2)}},\textbf{x}\right)|=\max_{\textbf{z}_{J^{(2)}} \in \mathbb{T}^{2}} |\Psi^{(2)}\left(\textbf{z}_{J^{(2)}},\textbf{x}\right)|<1.$$ This completes the proof.
			\end{proof}
			
			\begin{cor}
				Suppose $\textbf{x}=(x_1,\ldots,x_7)\in \mathbb C^{7}.$ Then
				\begin{enumerate}
					\item $\textbf{x}\in G_{E(3;3;1,1,1)}$ if and only if $(x_1,x_4,x_5,x_2,x_3,x_6,x_7)\in G_{E(3;3;1,1,1)}$.
					\item  $\textbf{x}\in G_{E(3;3;1,1,1)}$ if and only if $(x_2,x_1,x_3,x_4,x_6,x_5,x_7)\in G_{E(3;3;1,1,1)}.$
				\end{enumerate}
				
			\end{cor}	 
The following theorem provides a neat criterion for membership of an element $\textbf{x}\in G_{E(3;3;1,1,1)}$ to its preceding level, that is, tetrablock  $G_{E(2;2;1,1)}.$
	\begin{thm}\label{char}
		Suppose $\textbf{x}=(x_1,\ldots,x_7)\in \mathbb C^{7}.$ Then $\textbf{x}\in 
		G_{E(3;3;1,1,1)}$ if and only if
		it satisfies one of the following conditions:
		\begin{enumerate}
			\item $\left(\frac{x_2-z_1x_3}{1-x_1z_1},\frac{x_4-z_1x_5}{1-x_1z_1},\frac{x_6-z_1x_7}{1-x_1z_1}\right)\in G_{E(2;2;1,1)}$ for all $z_1\in \bar{\mathbb D}.$
			
			\item $\left(\frac{x_1-z_2x_3}{1-x_2z_2},\frac{x_4-z_2x_6}{1-x_2z_2},\frac{x_5-z_2x_7}{1-x_2z_2}\right)\in G_{E(2;2;1,1)}$ for all $z_2\in \bar{\mathbb D}.$
			
			\item $\left(\frac{x_1-z_3x_5}{1-x_4z_3},\frac{x_2-z_3x_6}{1-x_4z_3},\frac{x_3-z_3x_7}{1-x_4z_3}\right)\in G_{E(2;2;1,1)}$  for all $z_3\in \bar{\mathbb D}.$
		\end{enumerate}
	\end{thm}
	\begin{proof}
	Theorem $2.2$ in \cite{Abouhajar} describes various characterizations of tetrablock. Out of those, the one that we need for our proof, is described below:
		\begin{equation}\label{x3x2}|x_3|<1~{\rm{and~there~exist}}~\beta_1,\beta_2\in\mathbb C ~{\rm{with}}~|\beta_1|+|\beta_2|<1~{\rm{such~that}}~x_1=\beta_1+\bar{\beta}_2x_3~{\rm{and}}~x_2=\beta_2+\bar{\beta}_1x_3.\end{equation} We prove that  $\textbf{x}\in G_{E(3;3;1,1,1)}$ if and only if $(1)$ is satisfied, as the proof of the other cases is similar. Let $\textbf{x}\in G_{E(3;3;1,1,1)}$. By Theorem \ref{R_3}
and Theorem \ref{phi}, we have  $(x_1,x_2,x_3)\in G_{E(2;2;1,1)}$ and \begin{equation}\label{r33}
			1-x_1z_1-x_2z_2+x_3z_1z_2-x_4z_3+x_5z_1z_3+x_6z_2z_3-x_7z_1z_2z_3\neq 0~~{\rm{for ~~all}}~~z_1,z_2,z_3\in\bar{\mathbb D}.
		\end{equation}
		It is evident from \eqref{x3x2} that $$|x_1|<1,|x_2|<1~{\rm{ and }}~|x_3|<1.$$
Since $|x_1|<1$, we have $1-x_1z_1\neq 0$ for all $z_1 \in\bar{\mathbb D}.$  As a result, from \eqref{r33} we conclude that  $$\left(\frac{x_2-z_1x_3}{1-x_1z_1},\frac{x_4-z_1x_5}{1-x_1z_1},\frac{x_6-z_1x_7}{1-x_1z_1}\right)\in G_{E(2;2;1,1)}$$ for all $z_1\in \bar{\mathbb D}.$
		
		Conversely, suppose that $\left(\frac{x_2-z_1x_3}{1-x_1z_1},\frac{x_4-z_1x_5}{1-x_1z_1},\frac{x_6-z_1x_7}{1-x_1z_1}\right)\in G_{E(2;2;1,1)}$ for all $z_1\in \bar{\mathbb D}.$ It follows from the definition of $G_{E(2;2;1,1)}$ that 
		$$1-\left(\frac{x_2-z_1x_3}{1-x_1z_1}\right)z_2-\left(\frac{x_4-z_1x_5}{1-x_1z_1}\right)z_3+\left(\frac{x_6-z_1x_7}{1-x_1z_1}\right)z_2z_3\neq 0~~{\rm{for ~~all}}~~z_2,z_3\in\bar{\mathbb D},$$ which implies that $R^{(3;3;1,1,1)}_{\textbf {x}}(\textbf{z})\neq 0$ for all $\textbf{z}=(z_1,z_2,z_3)\in \bar{\mathbb {D}}^3.$ Hence, from  Theorem \ref{R_3}, we conclude that $\textbf{x}=(x_1,\ldots,x_7)\in G_{E(3;3;1,1,1)}$. This completes the proof.
	\end{proof}
\begin{rem}
By  the characterization of tetrablock it follows that

  $$\left(\frac{x_2-z_1x_3}{1-x_1z_1},\frac{x_4-z_1x_5}{1-x_1z_1},\frac{x_6-z_1x_7}{1-x_1z_1}\right)\in G_{E(2;2;1,1)} \;\; \mbox{for all}\;\; z_1\in \bar{\mathbb D}$$ if and only if $$\left(\frac{x_4-z_1x_5}{1-x_1z_1},\frac{x_2-z_1x_3}{1-x_1z_1}\frac{x_6-z_1x_7}{1-x_1z_1}\right)\in G_{E(2;2;1,1)} \;\;\mbox{for all}\;\; z_1\in \bar{\mathbb D}.$$

\end{rem}	
	As a consequence of the above theorem, we prove the following corollary.
	\begin{cor}
		Suppose $\textbf{x}=(x_1,\ldots,x_7)\in \mathbb C^{7}.$ Then $\textbf{x}\in 
		G_{E(3;3;1,1,1)}$ if and only if
		it satisfies one of the following condition:
		\begin{enumerate}
			
	\item
			There exists a $2\times 2$ symmetric matrix $B(z_1)$ with $\|B(z_1)\|<1$ for all $z_1\in \mathbb{\bar{\mathbb D}}$ 
			such that $$B_{11}(z_1)=\frac{x_2-z_1x_3}{1-x_1z_1},B_{22}(z_1)=\frac{x_4-z_1x_5}{1-x_1z_1}~{\rm{ and }}~\det(B(z_1))=\frac{x_6-z_1x_7}{1-x_1z_1}~{\rm{ for ~all }}~z_1\in \bar{\mathbb D};$$

			\item
			There exists a $2\times 2$ symmetric matrix $A(z_3)$ with $\|A(z_3)\|<1$ for all $z_3\in \mathbb{\bar{\mathbb D}}$ 
			such that $$A_{11}(z_3)=\frac{x_1-z_3x_5}{1-x_4z_3},A_{22}(z_3)=\frac{x_2-z_3x_6}{1-x_4z_3}~{\rm{ and }}~\det(A(z_3))=\frac{x_3-z_3x_7}{1-x_4z_3}~{\rm{ for ~all }}~z_3\in \bar{\mathbb D};$$

			\item
			There exists a $2\times 2$ symmetric matrix $C(z_2)$ with $\|C(z_2)\|<1$ for all $z_2\in \mathbb{\bar{\mathbb D}}$ 
			such that $$C_{11}(z_2)=\frac{x_4-z_2x_6}{1-x_2z_2},C_{22}(z_2)=\frac{x_1-z_2x_3}{1-x_2z_2}~{\rm{ and }}~\det(C(z_2))=\frac{x_5-z_2x_7}{1-x_2z_2}~{\rm{ for ~all }}~z_2\in \bar{\mathbb D}.$$
		\end{enumerate} 
	\end{cor}
	\begin{proof}
	We only prove $(1)$ because the proof of the other cases are identical. By Theorem \ref{char}, it follows that  $\textbf{x}\in 
		G_{E(3;3;1,1,1)}$ if and only if $$\left(\frac{x_2-z_1x_3}{1-x_1z_1},\frac{x_4-z_1x_5}{1-x_1z_1},\frac{x_6-z_1x_7}{1-x_1z_1}\right)\in G_{E(2;2;1,1)}\;\; \mbox{for all}\;\;z_1\in \bar{\mathbb D}.$$  
		As $\left(\frac{x_2-z_1x_3}{1-x_1z_1},\frac{x_4-z_1x_5}{1-x_1z_1},\frac{x_6-z_1x_7}{1-x_1z_1}\right)\in G_{E(2;2;1,1)}$  for all $z_1\in \bar{\mathbb D},$ then by characterization of tetrablock [Theorem $2.2$, \cite{Abouhajar}], it yields that there exists $2\times 2$ symmetric matrix $B(z_1)$ with $\|B(z_1)\|<1$ for all $z_1\in \mathbb{\bar{\mathbb D}}$ 
		such that $$B_{11}(z_1)=\frac{x_2-z_1x_3}{1-x_1z_1},B_{22}(z_1)=\frac{x_4-z_1x_5}{1-x_1z_1}~{\rm{ and }}~\det(B(z_1))=\frac{x_6-z_1x_7}{1-x_1z_1}~{\rm{for ~all }}~z_1\in \bar{\mathbb D}$$ and vice-versa. This completes the proof.
		\end{proof}

			\subsection{Characterization by  realization formula}
	In this subsection we describe the domain $G_{E(3;3;1,1,1)}$ by using realization formula. We explore two kinds of realization formulas to illustrate this.

\noindent \textbf{First realization formula:} Let $\mathcal H_{i}$ and $\mathcal K_i, 1\leq i \leq 3$, be the Hilbert spaces and $$P=
\begin{pmatrix}
	\begin{matrix}
		P_{11}
	\end{matrix}& \rvline & 
	\begin{matrix}
		P_{12} & P_{13}
	\end{matrix}  \\
	\hline \begin{matrix}
		P_{21} \\
		P_{31}
	\end{matrix} 
	&\rvline& \begin{matrix}
		P_{22} & P_{23} \\
		P_{32} & P_{33}
	\end{matrix}
\end{pmatrix}\left(=\begin{pmatrix}
	A_{11}& \rvline & 
	A_{12}  \\
	\hline A_{21} 
	&\rvline& A_{22}
\end{pmatrix}\right):\oplus_{i=1}^{3}\mathcal H_i\rightarrow \oplus_{i=1}^{3}\mathcal K_i.$$  Let $X$ be an operator from $\mathcal K_2\oplus\mathcal K_3$ to $\mathcal H_2 \oplus\mathcal H_3$ such that  $\left(I_{\mathcal K_2\oplus\mathcal K_3}-\left(\begin{smallmatrix}P_{22}&P_{23}\\P_{32} & P_{33}
	\end{smallmatrix}\right)X\right)$ is invertible. We have the following expression of the \textit{realization formula} for $A=P$, renaming the term $\mathcal M_{A}(X)$ as $\mathcal G_{P}(X)$ and for $\mathcal{E}=\mathcal H_{1}$, $\mathcal{H}=\mathcal H_{2}\oplus\mathcal{H}_{3}$, $\mathcal F=\mathcal K_1$ and $\mathcal K=\mathcal K_{2}\oplus\mathcal K_{2}$ in \eqref{rf}, 
	\begin{equation}\label{GPX}
		\mathcal G_{P}(X)=P_{11}+\left(\begin{smallmatrix}P_{12}&P_{13}
		\end{smallmatrix}\right)X\left(I_{\mathcal K_2\oplus\mathcal K_3}-\left(\begin{smallmatrix}P_{22}&P_{23}\\P_{32} & P_{33}
		\end{smallmatrix}\right)X\right)^{-1}\left(\begin{smallmatrix}P_{21}\\P_{31}
		\end{smallmatrix}\right).
	\end{equation}  Clearly, $\mathcal G_{P}(X)$ is an operator from $\mathcal H_1$ to $\mathcal K_1.$ The following Proposition follows easily from [Proposition $4.1$, \cite{BLY}].
	\begin{prop}\label{contt.}
		Let $\mathcal H_{i}$ and $\mathcal K_i, 1\leq i \leq 3$, be the Hilbert spaces and $P=((P_{ij}))_{i,j=1}^{3}$ and $Q=((Q_{ij}))_{i,j=1}^{3}$ be the operator from $\oplus_{i=1}^{3}\mathcal H_i$ to $\oplus_{i=1}^{3}\mathcal K_i.$ Also,  let $X$ and $Y$ be operators from $\mathcal K_2\oplus\mathcal K_3$ to $\mathcal H_2\oplus\mathcal H_3$ such that $\left(I_{\mathcal K_2\oplus\mathcal K_3}-\left(\begin{smallmatrix}P_{22}&P_{23}\\P_{32} & P_{33}
		\end{smallmatrix}\right)X\right)$ and $\left(I_{\mathcal K_2\oplus\mathcal K_3}-\left(\begin{smallmatrix}Q_{22}&Q_{23}\\Q_{32} & Q_{33}
		\end{smallmatrix}\right)Y\right)$ are invertible. Then 
		\small{\begin{align}\label{FPQQ}
				I_{\mathcal H_1}-\mathcal G_{Q}(Y)^*\mathcal G_{P}(X)&=B_{1}^*(I_{\mathcal K_2\oplus\mathcal K_3}-Y^*X)A_1+\left(\begin{smallmatrix} I_{\mathcal H_1} & B_1^*Y^*\end{smallmatrix}\right)(I_{\oplus_{i=1}^3\mathcal H_i}-Q^*P)\left(\begin{smallmatrix} I_{\mathcal H_1} \\ XA_1
				\end{smallmatrix}\right),
		\end{align}}
		where $A_1=\left(I_{\mathcal K_2\oplus\mathcal K_3}-\left(\begin{smallmatrix}P_{22}&P_{23}\\P_{32} & P_{33}
		\end{smallmatrix}\right)X\right)^{-1}\left(\begin{smallmatrix}P_{21} \\ P_{31}
		\end{smallmatrix}\right)$ and $B_1=\left(I_{\mathcal K_2\oplus\mathcal K_3}-\left(\begin{smallmatrix}Q_{22}&Q_{23}\\Q_{32} & Q_{33}
		\end{smallmatrix}\right)Y\right)^{-1}\left(\begin{smallmatrix}Q_{21} \\ Q_{31}
		\end{smallmatrix}\right).$
	\end{prop}
	\begin{proof}
		Let $\tilde{P}_{11}=P_{11},\tilde{P}_{12}=\left(\begin{smallmatrix}P_{12} & P_{13}
		\end{smallmatrix}\right),\tilde{P}_{21}=\left(\begin{smallmatrix}P_{21} \\ P_{31}
		\end{smallmatrix}\right),\tilde{P}_{22}=\left(\begin{smallmatrix}P_{22} &P_{23} \\ P_{32} &P_{33}
		\end{smallmatrix}\right), \mathcal H= \mathcal H_1, \mathcal U= \mathcal H_2\oplus\mathcal H_3, \mathcal K= \mathcal K_1$ and $ \mathcal Y= \mathcal K_2\oplus\mathcal K_3$. Then the above identity follows easily by  [Proposition $4.1$, \cite{BLY}].
	\end{proof}

	The following proposition is a consequence of the above proposition.
	
	\begin{prop}\label{fpzz}
		Suppose $\mathcal H_{i}$ and $\mathcal K_i, 1\leq i \leq 3$, are the Hilbert spaces. Let $P=((P_{ij}))_{i,j=1}^{3}$ be the operator from $\oplus_{i=1}^{3}\mathcal H_i$ to $\oplus_{i=1}^{3}\mathcal K_i$ and  $X$ be an operator from $\mathcal K_2\oplus\mathcal K_3$ to $\mathcal H_2\oplus\mathcal H_3$ such that $\left(I_{\mathcal K_2\oplus\mathcal K_3}-\left(\begin{smallmatrix}P_{22}&P_{23}\\P_{32} & P_{33}
		\end{smallmatrix}\right)X\right)$ is invertible. If $\|X\|\leq1$ and $\|P\|< 1$, then $\|\mathcal G_{P}(X)\|<1$.
		
	\end{prop}
	\begin{proof}
		From Proposition \ref{contt.}, we have
		\begin{align}\label{GP}
			I_{\mathcal H_1}-\mathcal G_{P}(X)^*\mathcal G_{P}(X)&=A_{1}^*(I_{\mathcal K_2\oplus\mathcal K_3}-X^*X)A_1+\left(\begin{smallmatrix} I_{\mathcal H_1} & A_1^*X^*\end{smallmatrix}\right)(I_{\oplus_{i=1}^3\mathcal H_i}-P^*P)\left(\begin{smallmatrix} I_{\mathcal H_1} \\ XA_1
			\end{smallmatrix}\right).
		\end{align}
		
		Let  
		$A_2=\left(\begin{smallmatrix} I_{\mathcal H_1} \\XA_1
		\end{smallmatrix}\right):\mathcal H_1 \to \mathcal H_1\oplus\mathcal H_2\oplus\mathcal H_3$ be the mapping from $\mathcal H_1$ to $\mathcal H_1\oplus\mathcal H_2\oplus\mathcal H_3.$ 
		Then from \eqref{GP}, we notice that 
		\begin{align}\label{GP1}
			I_{\mathcal H_1}-\mathcal G_{P}(X)^*\mathcal G_{P}(X)=A_1^*(I_{\mathcal K_2\oplus \mathcal K_3}-X^*X)A_1+A_2^*(I_{\oplus_{i=1}^3\mathcal H_i}-P^*P)A_2.
		\end{align}
		Since $\|X\|\leq1 $ and $\|P\|<1,$ from \eqref{GP1} it implies that $I_{\mathcal H_1}-\mathcal G_{P}(X)^*\mathcal G_{P}(X)>0$ and hence we have $\|\mathcal G_{P}(X)\|<1.$ This completes the proof.
	\end{proof}
We use the Proposition \ref{contt.} to  prove the following proposition and the proof of the proposition is similar to the Proposition \ref{fpzz}. Therefore, we skip the proof.
	\begin{prop}\label{fpzz1}
		Suppose $\mathcal H_{i}$ and $\mathcal K_i, 1\leq i \leq 3$, are the Hilbert spaces. Let $P=((P_{ij}))_{i,j=1}^{3}$ be the operator from $\oplus_{i=1}^{3}\mathcal H_i$ to $\oplus_{i=1}^{3}\mathcal K_i$ and  $X$ be an operator from $\mathcal K_2\oplus\mathcal K_3$ to $\mathcal H_2\oplus\mathcal H_3$ such that $\left(I_{\mathcal K_2\oplus\mathcal K_3}-\left(\begin{smallmatrix}P_{22}&P_{23}\\P_{32} & P_{33}
		\end{smallmatrix}\right)X\right)$ is invertible. If $\|X\|\leq1$ and $\|P\|\leq 1$, then $\|\mathcal G_{P}(X)\|\leq1.$ 
	\end{prop}
	
	For each $A=((a_{ij}))_{i,j=1}^{3}\in \mathcal M_{3\times 3}(\mathbb C)$ with $\|A\|<1$, we define $\tilde{\gamma}$ and $\tilde{\eta}$ by
	\begin{align}\label{ga}\tilde{\gamma}(z_2,z_3)\nonumber&=\left(I_{2}-\left(\begin{smallmatrix}a_{22} & a_{23} \\a_{32} &a_{33}
		\end{smallmatrix}\right)\left(\begin{smallmatrix}z_2 & 0 \\0 &z_{3}
		\end{smallmatrix}\right)\right)^{-1}\left(\begin{smallmatrix}a_{21} \\ a_{31}
		\end{smallmatrix}\right)\\&=\left(\begin{smallmatrix}\tilde{\gamma}_1(z_2,z_3) \\ \tilde{\gamma}_2(z_2,z_3)
		\end{smallmatrix}\right),
	\end{align} 
	and
	\begin{align}\label{et}\tilde{\eta}(z_2,z_3)\nonumber&=\left(\begin{smallmatrix}1\\\left(\begin{smallmatrix}z_2 & 0 \\0 &z_{3}
			\end{smallmatrix}\right)\tilde{\gamma}\left(\left(\begin{smallmatrix}z_2 & 0 \\0 &z_{3}
			\end{smallmatrix}\right)\right)
		\end{smallmatrix}\right)\\&=\left(\begin{smallmatrix}1\\z_2\tilde{\gamma}_1(z_2,z_3) \\ z_3\tilde{\gamma}_2(z_2,z_3)
		\end{smallmatrix}\right)
	\end{align}
	for all $z_2,z_3\in\mathbb C$ with $\det(C(z_2,z_3))\neq 0,$ where \small{$$C(z_2,z_3)=\left(I_{2}-\left(\begin{smallmatrix}a_{22} & a_{23} \\a_{32} &a_{33}
		\end{smallmatrix}\right)\left(\begin{smallmatrix}z_2 & 0 \\0 &z_{3}
		\end{smallmatrix}\right)\right),\tilde{\gamma}_1(z_2,z_3)=\frac{(1-a_{33}z_3)a_{21}+z_3a_{23}a_{31}} {\det(C(z_2,z_3))},$$ and $$\tilde{\gamma}_2(z_2,z_3)=\frac{(1-a_{22}z_2)a_{31}+z_2a_{32}a_{21}} {\det(C(z_2,z_3))}.$$}
	As a result of the above proposition, we will prove the following proposition. 
	\begin{prop}\label{matrix AA}
		Let  $A\in \mathcal M_{3\times 3}(\mathbb C).$  Then 
		\begin{align}\label{AAAA norm}
			1-\overline{\mathcal G_{A}\left(\left(\begin{smallmatrix}w_2 & 0 \\0 &w_{3}
				\end{smallmatrix}\right)\right)}\mathcal G_{A}\left(\left(\begin{smallmatrix}z_2 & 0 \\0 &z_{3}
			\end{smallmatrix}\right)\right)\nonumber&=\overline{{\tilde{\gamma_1}(w_2,w_3)}}(1-\bar{w}_2z_2)\tilde{\gamma}_1(z_2,z_3)+\overline{{\tilde{\gamma_2}(w_2,w_3)}}(1-\bar{w}_3z_3)\tilde{\gamma}_2(z_2,z_3)\\&+{\tilde{\eta}(w_2,w_3)}^*(I_{3}-A^*A)\tilde{\eta}(z_2,z_3)
		\end{align}
		for all $z_2,w_2,z_3,w_3\in \mathbb C$ with $\det\left(I_2-\left(\begin{smallmatrix}a_{22} & a_{23} \\a_{32} &a_{33}
		\end{smallmatrix}\right)\left(\begin{smallmatrix}z_2 & 0 \\0 &z_{3}
		\end{smallmatrix}\right)\right)\neq 0$ and $\det\left(I_2-\left(\begin{smallmatrix}a_{22} & a_{23} \\a_{32} &a_{33}
		\end{smallmatrix}\right)\left(\begin{smallmatrix}w_2 & 0 \\0 &w_{3}
		\end{smallmatrix}\right)\right)\neq 0$. Moreover, if $\|A\|<1$, then $|\mathcal G_{A}\left(\left(\begin{smallmatrix}z_2 & 0 \\0 &z_{3}
		\end{smallmatrix}\right)\right)|<1$ for all $z_2,z_3\in \bar{\mathbb D}.$ 		
	\end{prop}
	
	\begin{proof}
		Suppose $\mathcal H_i=\mathbb C=\mathcal K_i,1 \leq i\leq 3,$ $P=Q=A$, $X=\left(\begin{smallmatrix}z_2 & 0 \\0 &z_{3}
		\end{smallmatrix}\right)$ and $Y=\left(\begin{smallmatrix}w_2 & 0 \\0 &w_{3}
		\end{smallmatrix}\right)$ in Proposition \ref{contt.}. Then from \eqref{FPQQ}, we observe that 
		
		\small{\begin{align}\label{gpq112}
				&1-\overline{\mathcal G_{A}\left(\left(\begin{smallmatrix}w_2 & 0 \\0 &w_{3}
					\end{smallmatrix}\right)\right)}\mathcal G_{A}\left(\left(\begin{smallmatrix}z_2 & 0 \\0 &z_{3}
				\end{smallmatrix}\right)\right)\nonumber\\ &=\left(\begin{smallmatrix}\tilde{\gamma}_1(w_2,w_3) \\ \tilde{\gamma}_2(w_2,w_3)
				\end{smallmatrix}\right)^*\left(\begin{smallmatrix}(1-\bar{w}_2z_2) & 0 \\0 &(1-\bar{w}_3z_3)
				\end{smallmatrix}\right)\left(\begin{smallmatrix}\tilde{\gamma}_1(z_2,z_3) \\ \nonumber\tilde{\gamma}_2(z_2,z_3)
				\end{smallmatrix}\right)+\tilde{\eta}(w_2,w_3)^*(I_{3}-A^*A)\tilde{\eta}(z_2,z_3)\\\nonumber&=\overline{\tilde{\gamma}_1(w_2,w_3)}(1-\bar{w}_2z_2)\tilde{\gamma}_1(z_2,z_3)+\overline{\tilde{\gamma}_2(w_2,w_3)}(1-\bar{w}_3z_3)\tilde{\gamma}_2(z_2,z_3)\\&+{\eta(w_2,w_3)}^*(I_{3}-A^*A)\eta(z_2,z_3)
\end{align}}
for all $z_2,w_2,z_3,w_3\in \mathbb C$ with $\det\left(I_2-\left(\begin{smallmatrix}a_{22} & a_{23} \\a_{32} &a_{33}
\end{smallmatrix}\right)\left(\begin{smallmatrix}z_2 & 0 \\0 &z_{3}\end{smallmatrix}\right)\right)\neq 0$ and $\det\left(I_2-\left(\begin{smallmatrix}a_{22} & a_{23} \\a_{32} &a_{33}\end{smallmatrix}\right)\left(\begin{smallmatrix}w_2 & 0 \\0 &w_{3}\end{smallmatrix}\right)\right)\neq 0$. It also follows from Proposition \ref{fpzz} that $|\mathcal G_{A}\left(\left(\begin{smallmatrix}z_2 & 0 \\0 &z_{3}\end{smallmatrix}\right)\right)|<1$  for all $z_2,z_3\in \bar{\mathbb D}.$ This completes the proof.
	\end{proof}
	
\noindent \textbf{Second realization formula:} Let $\mathcal H_{i}$ and $\mathcal K_i,$ $1\leq i \leq 3,$  be the Hilbert spaces $$P=
\begin{pmatrix}
	\begin{matrix}
		P_{11} & P_{12} \\
		P_{21} & P_{22}
	\end{matrix}& \rvline & 
	\begin{matrix}
		P_{13}\\
		P_{23}
	\end{matrix}  \\
	\hline \begin{matrix}
		P_{31} & P_{32}
	\end{matrix} 
	&\rvline& P_{33}
\end{pmatrix}\left(=\begin{pmatrix}
	A_{11}& \rvline & 
	A_{12}  \\
	\hline A_{21} 
	&\rvline& A_{22}
\end{pmatrix}\right)
:\oplus_{i=1}^{3}\mathcal H_i\rightarrow \oplus_{i=1}^{3}\mathcal K_i.$$  Assume that $X$ be an operator from $\mathcal K_3$ to $\mathcal H_3$ such that  $\left(I_{\mathcal K_3}-P_{33}X\right)$ is invertible. We denote  $\mathcal M_{A}(X)$ as $\mathcal F_{P}(X)$ and for $\mathcal{U}_1=\mathcal H_{1}\oplus\mathcal H_{2}$, $\mathcal{U}_2=\mathcal{H}_{3}$, $\mathcal V_1=\mathcal K_1\oplus\mathcal K_{2}$ and $\mathcal V_2=\mathcal K_{3}$ in \eqref{rf}, we have the following expression for the \textit{realization formula} 
	\begin{equation}\label{FPX}
		\mathcal F_{P}(X)=((P_{ij}))_{i,j=1}^{2}+\left(\begin{smallmatrix}P_{13}\\P_{23}
		\end{smallmatrix}\right)X(I_{\mathcal K_3}-P_{33}X)^{-1}\left(\begin{smallmatrix}P_{31}&P_{32}
		\end{smallmatrix}\right).
	\end{equation}  Obviously, $\mathcal F_{P}(X)$ is an operator from $\oplus_{i=1}^{2}\mathcal H_i$ to $\oplus_{i=1}^{2}\mathcal K_i.$ The following standard identity follows from [Proposition $4.1$, \cite{BLY}].
	\begin{prop}\label{cont.}
		Let $\mathcal H_{i}$ and $\mathcal K_i$ for $i=1,2,3$ be the Hilbert spaces and $P=((P_{ij}))_{i,j=1}^{3}$ and $Q=((Q_{ij}))_{i,j=1}^{3}$ be the operator from $\oplus_{i=1}^{3}\mathcal H_i$ to $\oplus_{i=1}^{3}\mathcal K_i.$ Also,  let $X$ and $Y$ be operators from $\mathcal K_3$ to $\mathcal H_3$ such that  $I_{\mathcal K_3}-P_{33}X$ and $I_{\mathcal K_3}-Q_{33}Y$ are invertible. Then 
		\small{\begin{align}\label{FPQ}
				&I_{\mathcal H_1\oplus \mathcal H_2}-\mathcal F_{Q}(Y)^*\mathcal F_{P}(X)\nonumber\\ \nonumber&=\left(\begin{smallmatrix}Q_{31} & Q_{32}
				\end{smallmatrix}\right)^*(I_{\mathcal K_3}-Y^*Q_{33}^*)^{-1}(I_{\mathcal K_3}-Y^*X)(I_{\mathcal K_3}-P_{33}X)^{-1}\left(\begin{smallmatrix}P_{31} & P_{32}
				\end{smallmatrix}\right)\\&+\left(\begin{smallmatrix} I_{\mathcal H_1\oplus \mathcal H_2} ~& ~\left(\begin{smallmatrix}Q_{31} & Q_{32}
					\end{smallmatrix}\right)^*(I_{\mathcal K_3}-Y^*Q_{33}^*)^{-1}Y^*\end{smallmatrix}\right)(I_{\oplus_{i=1}^3\mathcal H_i}-Q^*P)\left(\begin{smallmatrix} I_{\mathcal H_1\oplus \mathcal H_2} \\ X(I_{\mathcal K_3}-P_{33}X)^{-1}\left(\begin{smallmatrix}P_{31}& P_{32}
					\end{smallmatrix}\right)\end{smallmatrix}\right).
		\end{align}}
		
	\end{prop}
	\begin{proof} Let
		$$\tilde{P}_{11}=((P_{ij}))_{i,j=1}^{2},\tilde{P}_{12}=\left(\begin{smallmatrix}P_{13} \\ P_{23}
		\end{smallmatrix}\right),\tilde{P}_{21}=\left(\begin{smallmatrix}P_{31} & P_{32}
		\end{smallmatrix}\right),\tilde{P}_{22}=P_{33}, \mathcal H= \mathcal H_1\oplus \mathcal H_2, \mathcal U=\mathcal H_3, \mathcal K= \mathcal K_1\oplus \mathcal K_2$$ and  $\mathcal V=\mathcal K_3$. It is easy to see that the above identity follows from [Proposition $4.1$, \cite{BLY}].
	\end{proof}
	As a consequence of the above proposition, we prove the following proposition.
	\begin{prop}\label{fpz}
		Let $\mathcal H_{i}$ and $\mathcal K_i, 1\leq i \leq 3$ be the Hilbert spaces. Suppose $P=((P_{ij}))_{i,j=1}^{3}$ is the operator from $\oplus_{i=1}^{3}\mathcal H_i$ to $\oplus_{i=1}^{3}\mathcal K_i$ and  $X$ is an operator from $\mathcal K_3$ to $\mathcal H_3$ such that  $I_{\mathcal K_3}-P_{33}X$ is invertible. If $\|X\|\leq1$ and $\|P\|< 1$, then $\|\mathcal F_{P}(X)\|<1.$
		 
	\end{prop}
	\begin{proof}
		From Proposition \ref{cont.}, we notice that 
		
		\begin{align}\label{FP}
			&I_{\mathcal H_1\oplus \mathcal H_2}-\mathcal F_{P}(X)^*\mathcal F_{P}(X)\nonumber\\ &=\left(\begin{smallmatrix}P_{31} & P_{32}
			\end{smallmatrix}\right)^*(I_{\mathcal K_3}-X^*P_{33}^*)^{-1}(I_{\mathcal K_3}-X^*X)(I_{\mathcal K_3}-P_{33}X)^{-1}\left(\begin{smallmatrix}P_{31} & P_{32}
			\end{smallmatrix}\right) \nonumber \\&+\left(\begin{smallmatrix} I_{\mathcal H_1\oplus \mathcal H_2} ~& ~\left(\begin{smallmatrix}P_{31} & P_{32}
				\end{smallmatrix}\right)^*(I_{\mathcal K_3}-X^*P_{33}^*)^{-1}X^*\end{smallmatrix}\right)(I_{\oplus_{i=1}^3\mathcal H_i}-P^*P)\left(\begin{smallmatrix} I_{\mathcal H_1\oplus \mathcal H_2} \\ X(I_{\mathcal K_3}-P_{33}X)^{-1}\left(\begin{smallmatrix}P_{31}& P_{32}
				\end{smallmatrix}\right)\end{smallmatrix}\right).
		\end{align}
		
		Suppose $A_1=(I_{\mathcal K_3}-P_{33}X)^{-1}\left(\begin{smallmatrix}P_{31} & P_{32}
		\end{smallmatrix}\right):\mathcal H_1\oplus\mathcal H_2 \to \mathcal K_3 $ and 
		$$A_2=\left(\begin{smallmatrix} I_{\mathcal H_1\oplus \mathcal H_2} \\XA_1
		\end{smallmatrix}\right):\mathcal H_1\oplus\mathcal H_2 \to \mathcal H_1\oplus\mathcal H_2\oplus\mathcal H_3. $$
		Then from \eqref{FP}, it follows that 
		\begin{align}\label{FP1}
			I_{\mathcal H_1\oplus \mathcal H_2}-\mathcal F_{P}(X)^*\mathcal F_{P}(X)=A_1^*(I_{\mathcal K_3}-X^*X)A_1+A_2^*(I_{\oplus_{i=1}^3\mathcal H_i}-P^*P)A_2.
		\end{align}
		Since $\|X\|\leq1 $ and $\|P\|<1,$ it follows  from \eqref{FP1} that $I_{\mathcal H_1\oplus \mathcal H_2}-\mathcal F_{P}(X)^*\mathcal F_{P}(X)>0$ and hence $\|\mathcal F_{P}(X)\|<1.$ This completes the proof.
	\end{proof}
	For each $A=((a_{ij}))_{i,j=1}^{3}\in \mathcal M_{3\times 3}(\mathbb C)$ with $\|A\|<1$, we define $\gamma$ and $\eta$ by
	$$\gamma(z_3)=(1-a_{33}z_3)^{-1}\left(\begin{smallmatrix}a_{31} & a_{32}
	\end{smallmatrix}\right),$$ and
	$$\eta(z_3)=\left(\begin{smallmatrix}I_{2}\\z_3(1-a_{33}z_3)^{-1}\left(\begin{smallmatrix}a_{31} & a_{32}
		\end{smallmatrix}\right)
	\end{smallmatrix}\right)=\left(\begin{smallmatrix}I_{2}\\z_3\gamma(z_3)
	\end{smallmatrix}\right)$$
	for all $z_3\in\mathbb C$ such that $1-a_{33}z_3\neq 0.$
	We prove the following proposition as a consequence of Proposition \ref{cont.}. 
	\begin{prop}\label{matrix A}
		Let  $A\in \mathcal M_{3\times 3}(\mathbb C)$. Then 
		\begin{align}\label{A norm}
			I_{2}-\mathcal F_{A}(w_3)^*\mathcal F_{A}(z_3)&={\gamma(w_3)}^*(1-\bar{w}_3z_3)\gamma(z_3)+{\eta(w_3)}^*(I_{3}-A^*A)\eta(z_3)
		\end{align}
		for all $z_3,w_3\in \mathbb C$ with $1-a_{33}w_3\neq 0$ and $1-a_{33}z_3\neq 0.$ Moreover, if $\|A\|<1$, then $\|\mathcal F_{A}(z_3)\|<1$ for all $z_3\in \bar{\mathbb D}.$ 
		
	\end{prop}
	
	\begin{proof}
		Suppose $\mathcal H_i=\mathbb C=\mathcal K_i, 1 \leq i \leq 3$, $P=Q=A$, $X=z_3$ and $Y=w_3$ in Proposition \eqref{cont.}. Then from \eqref{FPQ}, we observe that 
		\begin{align}\label{fpq1}
			I_{2}-\mathcal F_{A}(w_3)^*\mathcal F_{A}(z_3)\nonumber&=\left(\begin{smallmatrix}a_{31} & a_{32}
			\end{smallmatrix}\right)^*(1-\bar{w}_3\bar{a}_{33})^{-1}(1-\bar{w}_3z_3)(1-a_{33}z_3)^{-1}\left(\begin{smallmatrix}a_{31} & a_{32}
			\end{smallmatrix}\right)\\\nonumber&+\left(\begin{smallmatrix} I_{2} ~&~ \left(\begin{smallmatrix}a_{31} & a_{32}
				\end{smallmatrix}\right)^*(1-\bar{w}_3\bar{a}_{33})^{-1}\bar{w}_3\end{smallmatrix}\right)(I_{3}-A^*A)\left(\begin{smallmatrix} I_{2} \\ z_3(1-a_{33}z_3)^{-1}\left(\begin{smallmatrix}a_{31}& a_{32}
				\end{smallmatrix}\right)\end{smallmatrix}\right)\\&={\gamma(w_3)}^*(1-\bar{w}_3z_3)\gamma(z_3)+{\eta(w_3)}^*(I_{3}-A^*A)\eta(z_3)
		\end{align}
		for all $z_3,w_3\in \mathbb C$ with $1-a_{33}w_3\neq 0$ and $1-a_{33}z_3\neq 0.$
		From \eqref{fpq1}, we have $\|\mathcal F_{A}(z_3)\|<1$ for all $z_3\in \bar{\mathbb D}$ and $\|A\|<1.$ This completes the proof.
	\end{proof}
	
We derive the following theorem by using first realization formula.			
			\begin{thm}\label{matix A1}
				For $\textbf{x}=(x_1,\ldots,x_7)\in \mathbb C^7$ the following are equivalent.
				\begin{enumerate}
					\item $\textbf{x}=(x_1,\ldots,x_7)\in G_{E(3;3;1,1,1)}$.
					
					\item There exists a $3\times 3$ matrix $A\in \mathcal M_{3\times 3}(\mathbb C)$  such that  \small{$$x_1=a_{11}, x_2=a_{22}, x_3=\det \left(\begin{smallmatrix} a_{11} & a_{12}\\
					a_{21} & a_{22}
				\end{smallmatrix}\right), x_4=a_{33}, x_5=\det \left(\begin{smallmatrix}
					a_{11} & a_{13}\\
					a_{31} & a_{33}
				\end{smallmatrix}\right), x_6=\det  \left(\begin{smallmatrix}
					a_{22} & a_{23}\\
					a_{32} & a_{33}\end{smallmatrix}\right) ~{\rm{and}}~x_7=\det A,$$}  \begin{equation}\label{tilgamma1}
					\overline{\tilde{\gamma}_1(z_2,z_3)}(1-|z_2|^2)\tilde{\gamma}_1(z_2,z_3)+\overline{\tilde{\gamma}_2(z_2,z_3)}(1-|z_3|^2)\tilde{\gamma}_2(z_2,z_3)+{\eta(z_2,z_3)}^*(I_{3}-A^*A)\eta(z_2,z_3)>0 
				\end{equation}
				and $\det\left(I_{2}-\left(\begin{smallmatrix}a_{22}&a_{23}\\a_{32} & a_{33}
				\end{smallmatrix}\right)\left(\begin{smallmatrix}z_2 & 0 \\0 &z_{3}
				\end{smallmatrix}\right)\right)\neq 0$ for all $z_2,z_3\in \bar{\mathbb D}$.
				\end{enumerate}
				
			\end{thm}
			\begin{proof}
				Suppose $\mathbf{x}=(x_1,\ldots,x_7)\in G_{E(3;3;1,1,1)}.$ By \eqref{GE}, there exits a matrix $A=((a_{ij}))_{i,j=1}^{3}$ such that \small{$$x_1=a_{11}, x_2=a_{22}, x_3=\det \left(\begin{smallmatrix} a_{11} & a_{12}\\
					a_{21} & a_{22}
				\end{smallmatrix}\right), x_4=a_{33}, x_5=\det \left(\begin{smallmatrix}
					a_{11} & a_{13}\\
					a_{31} & a_{33}
				\end{smallmatrix}\right), x_6=\det  \left(\begin{smallmatrix}
					a_{22} & a_{23}\\
					a_{32} & a_{33}\end{smallmatrix}\right) ~{\rm{and}}~x_7=\det A.$$} Note that for all $z_2,z_3\in \mathbb C$ with $\det\left(I_{2}-\left(\begin{smallmatrix}a_{22}&a_{23}\\a_{32} & a_{33}
				\end{smallmatrix}\right)\left(\begin{smallmatrix}z_2 & 0 \\0 &z_{3}
				\end{smallmatrix}\right)\right)\neq 0,$ we have
				\small{\begin{align}\label{mathcalG}
						\mathcal G_{A}\left(\left(\begin{smallmatrix}z_2 & 0 \\0 &z_{3}
						\end{smallmatrix}\right)\right)\nonumber &=a_{11}+\left(\begin{smallmatrix}a_{12}&a_{13}
						\end{smallmatrix}\right)\left(\begin{smallmatrix}z_2 & 0 \\0 &z_{3}
						\end{smallmatrix}\right)\left(I_{2}-\left(\begin{smallmatrix}a_{22}&a_{23}\\a_{32} & a_{33}
						\end{smallmatrix}\right)\left(\begin{smallmatrix}z_2 & 0 \\0 &z_{3}
						\end{smallmatrix}\right)\right)^{-1}\left(\begin{smallmatrix}a_{21}\\a_{31}
						\end{smallmatrix}\right)\\\nonumber&=\frac{a_{11}-z_2(a_{11}a_{22}-a_{12}a_{21})-z_3(a_{33}a_{11}-a_{13}a_{31})+z_2z_3\det(A)}{1-z_2a_{22}-z_3a_{33}+z_2z_3(a_{22}a_{33}-a_{23}a_{32})}\\\nonumber&=\frac{x_1-z_2x_3-z_3x_5+z_2z_3x_7}{1-z_2x_2-z_3x_4+z_2z_3x_6}\\&=\Psi^{(1)}(\textbf{z}_{J^{(1)}},\textbf{x}).
				\end{align}}
From Theorem \ref{phi}, we have $\textbf{x}\in G_{E(3;3;1,1,1)}$ if and only if  
\begin{eqnarray}\label{eqn1}
\textbf{x}_{J^{(1)}}^{\prime}\in G_{E(2;2;1,1)}~{\rm{ and }}~\|\Psi^{(1)}(\cdot,\textbf{x})\|_{H^{\infty}(\bar{\mathbb{D}}^{2})}=\|\Psi^{(1)}(\cdot,\textbf{x})\|_{H^{\infty}(\mathbb{T}^{2})}< 1.\end{eqnarray} Thus, from \eqref{mathcalG} and \eqref{eqn1}, it follows that 
			$\textbf{x}\in G_{E(3;3;1,1,1)}$ if and only if	$\sup_{(z_2,z_3)\in\bar{\mathbb D}^2} |\mathcal G_{A}\left(\left(\begin{smallmatrix}z_2 & 0 \\0 &z_{3}
				\end{smallmatrix}\right)\right)|<1$, and $\det\left(I_{2}-\left(\begin{smallmatrix}a_{22}&a_{23}\\a_{32} & a_{33}
				\end{smallmatrix}\right)\left(\begin{smallmatrix}z_2 & 0 \\0 &z_{3}
				\end{smallmatrix}\right)\right)\neq 0$  for all $(z_2,z_3)\in\bar{\mathbb D}^2$. Thus we have $\textbf{x}\in G_{E(3;3;1,1,1)}$ if and only if
				\begin{equation}\label{GAA}
					1-\overline{\mathcal G_{A}\left(\left(\begin{smallmatrix}z_2 & 0 \\0 &z_{3}
						\end{smallmatrix}\right)\right)}\mathcal G_{A}\left(\left(\begin{smallmatrix}z_2 & 0 \\0 &z_{3}
					\end{smallmatrix}\right)\right)>0 ~{\rm{and}}~\det\left(I_{2}-\left(\begin{smallmatrix}a_{22}&a_{23}\\a_{32} & a_{33}
				\end{smallmatrix}\right)\left(\begin{smallmatrix}z_2 & 0 \\0 &z_{3}
				\end{smallmatrix}\right)\right)\neq 0 ~~{\rm{for~all}}~~(z_2,z_3)\in\bar{\mathbb D}^2.
				\end{equation}
				By \eqref{AAAA norm} and \eqref{GAA}, we conclude that $\textbf{x}\in G_{E(3;3;1,1,1)}$ if and only if 
				\begin{equation}\label{tilgamma}
					\overline{\tilde{\gamma}_1(z_2,z_3)}(1-|z_2|^2)\tilde{\gamma}_1(z_2,z_3)+\overline{\tilde{\gamma}_2(z_2,z_3)}(1-|z_3|^2)\tilde{\gamma}_2(z_2,z_3)+{\eta(z_2,z_3)}^*(I_{3}-A^*A)\eta(z_2,z_3)>0 
				\end{equation}
				and  $\det\left(I_{2}-\left(\begin{smallmatrix}a_{22}&a_{23}\\a_{32} & a_{33}
				\end{smallmatrix}\right)\left(\begin{smallmatrix}z_2 & 0 \\0 &z_{3}
				\end{smallmatrix}\right)\right)\neq 0$ for all $z_2,z_3\in \bar{\mathbb D}$.
This completes the proof.
				
						\end{proof}

Below we prove the another implication of Therem \ref{matix A1} by using the second realization formula. For \\$A\in \mathcal M_{3\times 3}(\mathbb C)$ and $\mathcal F_{A}(z_3)$ as in \eqref{FAA}, we define 
			\begin{equation}\label{GFA}\mathcal G_{\mathcal F_{A}(z_3)}(z_2)=A_{11}(z_3)+A_{12}(z_3)z_2(1-A_{22}(z_3)z_2)^{-1}A_{21}(z_3)\end{equation} for all $z_2,z_3\in \mathbb C$ such that $(1-a_{33}z_3)\neq 0$ and $(1-A_{22}(z_3)z_2)\neq 0,$ that is, $$1-a_{22}z_2-a_{33}z_3+(a_{22}a_{33}-a_{23}a_{32})z_2z_3\neq 0.$$
			For $A\in \mathcal M_{3\times 3}(\mathbb C)$ with $\|A\|<1,$ from [Proposition $4.2$,\cite{BLY}], we have $\|\mathcal F_{A}(z_3)\|<1$ for all $z_3\in \bar{\mathbb D}$  which implies from Proposition \ref{fpzz} that $\|\mathcal G_{\mathcal F_{A}(z_3)}(z_2)\|<1$ for all $z_2,z_3\in \mathbb {\bar{D}}$ such that $$1-a_{22}z_2-a_{33}z_3+(a_{22}a_{33}-a_{23}a_{32})z_2z_3\neq 0.$$
			For $\mathcal F_{A}(z_3)\in \mathcal M_{2\times 2}(\mathbb C)$, we define $\tilde{\gamma}^{(z_3)}$ and $\tilde{\eta}^{(z_3)}$ by
			\begin{align}\label{gaz}\tilde{\gamma}^{(z_3)}(z_2)&=(1-A_{22}(z_3)z_2)^{-1}A_{21}(z_3),
			\end{align} 
			and
			\begin{align}\label{ettz}\tilde{\eta}^{(z_3)}(z_2)&=\left(\begin{smallmatrix}1\\z_2\tilde{\gamma}^{(z_3)}(z_2)
				\end{smallmatrix}\right)
			\end{align}
			for all $z_2,z_3\in\mathbb C$ such that $(1-a_{33}z_3)\neq 0$ and  $1-a_{22}z_2-a_{33}z_3+(a_{22}a_{33}-a_{23}a_{32})z_2z_3\neq 0.$ The proof of the following proposition  follows from [Proposition $4.1$ \cite{BLY}]. 
			\begin{prop}\label{matrix AAZZ}
				Suppose  $\mathcal F_{A}(z_3)$  as in \eqref{FAA} for all $z_3\in \mathbb C$ such that $(1-a_{33}z_3)\neq 0.$  Then 
				\begin{align}\label{AAAA norm1}
					1-\overline{\mathcal G_{\mathcal F_{A}(w_3)}(w_2)}\mathcal G_{\mathcal F_{A}(z_3)}(z_2)\nonumber&=\overline{{\tilde{\gamma_1}^{(w_3)}(w_2)}}(1-\bar{w}_2z_2)\tilde{\gamma}_1^{(z_3)}(z_2)\\&+{\tilde{\eta}^{(w_3)}(w_2)}^*(I_{2}-\mathcal F_{A}(w_3)^*\mathcal F_{A}(z_3))\tilde{\eta}^{(z_3)}(z_2)
				\end{align}
				for all $z_2,w_2,z_3,w_3\in \mathbb C$ such that $(1-a_{33}z_3)\neq 0, 1-a_{22}z_2-a_{33}z_3+(a_{22}a_{33}-a_{23}a_{32})z_2z_3\neq 0, (1-a_{33}w_3)\neq 0,$ and  $1-a_{22}w_2-a_{33}w_3+(a_{22}a_{33}-a_{23}a_{32})w_2w _3\neq 0.$ Moreover, if $\|\mathcal F_{A}(z_3)\|<1$, then $|\mathcal G_{\mathcal F_{A}(z_3)}(z_2)|<1$ for all $z_2,z_3\in \bar{\mathbb D}$.
			\end{prop}
			Note that 
			\begin{align}\label{gafz3}
				\mathcal G_{\mathcal F_{A}(z_3)}(z_2)\nonumber&=A_{11}(z_3)+A_{12}(z_3)z_2(1-A_{22}(z_3)z_2)^{-1}A_{21}(z_3)\\\nonumber &=\frac{x_1-z_3x_5}{1-x_4z_3}+\frac{(a_{12}-z_3(a_{12}a_{33}-a_{13}a_{32}))(a_{21}-z_3(a_{21}a_{33}-a_{23}a_{31}))z_2}{(1-x_4z_3)(1-x_2z_2-x_4z_3+z_2z_3x_6)}\\\nonumber&=\frac{x_1-z_2x_3-z_3x_5+z_2z_3x_7}{1-z_2x_2-z_3x_4+z_2z_3x_6}\\&=\Psi^{(1)}(\textbf{z}_{J^{(1)}},\textbf{x}).
			\end{align}

The proof of the following theorem is similar to Theorem \ref{matix A1}. Therefore, we skip the proof.
\begin{thm}\label{matix A144}
				For $\textbf{x}=(x_1,\ldots,x_7)\in \mathbb C^7$ the following are equivalent.
				\begin{enumerate}
					\item $\textbf{x}=(x_1,\ldots,x_7)\in G_{E(3;3;1,1,1)}$;
					
					\item There exists a $3\times 3$ matrix $A\in \mathcal M_{3\times 3}(\mathbb C)$  such that  \small{$$x_1=a_{11}, x_2=a_{22}, x_3=\det \left(\begin{smallmatrix} a_{11} & a_{12}\\
					a_{21} & a_{22}
				\end{smallmatrix}\right), x_4=a_{33}, x_5=\det \left(\begin{smallmatrix}
					a_{11} & a_{13}\\
					a_{31} & a_{33}
				\end{smallmatrix}\right), x_6=\det  \left(\begin{smallmatrix}
					a_{22} & a_{23}\\
					a_{32} & a_{33}\end{smallmatrix}\right) ~{\rm{and}}~x_7=\det A$$} with 
					\begin{equation}\label{tilgamma12}
					\overline{{\tilde{\gamma_1}^{(w_3)}(w_2)}}(1-\bar{w}_2z_2)\tilde{\gamma}_1^{(z_3)}(z_2)+{\tilde{\eta}^{(w_3)}(w_2)}^*(I_{2}-\mathcal F_{A}(w_3)^*\mathcal F_{A}(z_3))\tilde{\eta}^{(z_3)}(z_2)>0 
				\end{equation}
				for all $z_2,z_3\in \bar{\mathbb D}$ such that $\det\left(I_{2}-\left(\begin{smallmatrix}a_{22}&a_{23}\\a_{32} & a_{33}
				\end{smallmatrix}\right)\left(\begin{smallmatrix}z_2 & 0 \\0 &z_{3}
				\end{smallmatrix}\right)\right)\neq 0.$
				\end{enumerate}
				
			\end{thm}

\begin{lem}\label{12}
			Let $A=((a_{ij}))_{i,j=1}^{3}$ and $\tilde{A}=J_1AJ_2,$ where $J_1=\left(\begin{smallmatrix} 0 &0 & 1\\1 & 0 & 0\\ 0& 1 & 0\end{smallmatrix}\right)$ and $J_2=\left(\begin{smallmatrix} 0 &1 & 0\\0 & 0 & 1\\ 1 & 0 & 0\end{smallmatrix}\right).$ Then $$\mu_{E(3;3;1,1,1)}(A)=\mu_{E(3;3;1,1,1)}(\tilde{A}).$$

		\end{lem}
		\begin{proof}
			The proof of the lemma involves two cases:
			
			\noindent$\textbf{Case 1}$: Assume that $\mu_{E(3;3;1,1,1)}(A)=0$. From Lemma \ref{M_3}, it follows that 
			\small{\begin{equation}\label{asa}a_{ii}=0 ~{\rm{for}}~i=1,2,3, a_{11}a_{22}-a_{12}a_{21}=0, a_{11}a_{33}-a_{13}a_{31}=0, a_{22}a_{33}-a_{23}a_{32}=0~{\rm{and}}~\det A=0\end{equation}} Clearly, $\tilde {A}= ((\tilde{a}_{ij}))_{i,j=1}^{3}=J_1AJ_2=\left(\begin{smallmatrix} a_{33} & a_{31} & a_{32}\\a_{13} & a_{11} & a_{12}\\ a_{23} & a_{21} & a_{22}\end{smallmatrix}\right).$ By \eqref{asa}, we get  
			\begin{equation} \label{muEE} 
\tilde{a}_{ii}=0 ~{\rm{for}}~i=1,2,3,\tilde{a}_{11}\tilde{a}_{22}-\tilde{a}_{12}\tilde{a}_{21}=0, \tilde{a}_{11}\tilde{a}_{33}-\tilde{a}_{13}\tilde{a}_{31}=0,\tilde{a}_{22}\tilde{a}_{33}-\tilde{a}_{23}\tilde{a}_{32}=0~{\rm{ and }}~\operatorname{det}\tilde{A}=\operatorname{det}A=0.
			\end{equation} 
			Hence from Lemma \ref{M_3}, we have $\mu_{E(3;3;1,1,1)}(\tilde{A})=0.$
			
			$\textbf{Case 2}$: Assume that $\mu_{E(3;3;1,1,1)}(A)\neq0$.  Then by definition we have
			\begin{align*}
				\frac{1}{\mu_{E(3;3;1,1,1)}(A)}&=\inf\{\|X\|: X\in E(3;3;1,1,1), I-AX~{\rm{is~singular}}\}\\&=\inf\{\|X\|: X\in E(3;3;1,1,1), I-J_2\tilde{A}J_1X~{\rm{is~singular}}\}\\&=\inf\{\|X\|: X\in E(3;3;1,1,1), I-\tilde{A}J_1XJ_2~{\rm{is~singular}}\}\\&=\inf\{\|Y\||: Y=J_1XJ_2\in E(3;3;1,1,1), I-\tilde{A}Y~{\rm{is~singular}}\}\\&=\frac{1}{\mu_{E(3;3;1,1,1)}(\tilde{A})}.
			\end{align*}
			This completes the proof.
		\end{proof}
		The proof of the following lemma is same as Lemma \ref{12}. Therefore, we skip the proof.
		\begin{lem}\label{13}
			Let $A=((a_{ij}))_{i,j=1}^{3}$ and $\tilde{B}=\tilde{J}_1A\tilde{J}_2,$ where $\tilde{J}_1=\left(\begin{smallmatrix} 0 &1& 0\\0& 0 & 1\\ 1& 0& 0\end{smallmatrix}\right)$ and $\tilde{J}_2=\left(\begin{smallmatrix} 0 &0 & 1\\1& 0 & 0\\ 0 & 1 & 0\end{smallmatrix}\right).$ Then $$\mu_{E(3;3;1,1,1)}(A)=\mu_{E(3;3;1,1,1)}(\tilde{B}).$$
			
		\end{lem}
		As $J_1,J_2,\tilde{J}_1$ and $\tilde{J}_2$ are unitary, then $\|A\|=\|\tilde{A}\|=\|\tilde{B}\|.$ Observe that 
		\small{\begin{align}\label{mathcalG12}
				\mathcal G_{\tilde{A}}\left(\left(\begin{smallmatrix}z_1 & 0 \\0 &z_{2}
				\end{smallmatrix}\right)\right)\nonumber &=a_{33}+\left(\begin{smallmatrix}a_{31}&a_{32}
				\end{smallmatrix}\right)\left(\begin{smallmatrix}z_1 & 0 \\0 &z_{2}
				\end{smallmatrix}\right)\left(I_{2}-\left(\begin{smallmatrix}a_{11}&a_{12}\\a_{21} & a_{22}
				\end{smallmatrix}\right)\left(\begin{smallmatrix}z_1 & 0 \\0 &z_{2}
				\end{smallmatrix}\right)\right)^{-1}\left(\begin{smallmatrix}a_{13}\\a_{23}
				\end{smallmatrix}\right)\\\nonumber&=\frac{a_{33}-z_1(a_{11}a_{33}-a_{13}a_{31})-z_2(a_{22}a_{33}-a_{23}a_{32})+z_1z_2\det(A)}{1-z_1a_{11}-z_2a_{22}+z_1z_2(a_{11}a_{22}-a_{12}a_{21})}\\\nonumber&=\frac{x_4-z_1x_5-z_2x_6+z_1z_2x_7}{1-z_1x_1-z_2x_2+z_2z_3x_3}\\&=\Psi^{(3)}(\textbf{z}_{J^{(3)}},\textbf{x}).
		\end{align}}
		and 
		\small{\begin{align}\label{mathcalG31}
				\mathcal G_{\tilde{B}}\left(\left(\begin{smallmatrix}z_3 & 0 \\0 &z_{1}
				\end{smallmatrix}\right)\right)\nonumber &=a_{22}+\left(\begin{smallmatrix}a_{23}&a_{21}
				\end{smallmatrix}\right)\left(\begin{smallmatrix}z_3 & 0 \\0 &z_{1}
				\end{smallmatrix}\right)\left(I_{2}-\left(\begin{smallmatrix}a_{33}&a_{31}\\a_{13} & a_{11}
				\end{smallmatrix}\right)\left(\begin{smallmatrix}z_3& 0 \\0 &z_{1}
				\end{smallmatrix}\right)\right)^{-1}\left(\begin{smallmatrix}a_{32}\\a_{12}
				\end{smallmatrix}\right)\\\nonumber&=\frac{a_{22}-z_1(a_{11}a_{22}-a_{12}a_{21})-z_3(a_{22}a_{33}-a_{23}a_{32})+z_1z_3\det(A)}{1-z_1a_{11}-z_3a_{33}+z_1z_3(a_{11}a_{33}-a_{13}a_{31})}\\\nonumber&=\frac{x_2-z_1x_3-z_3x_6+z_1z_3x_7}{1-z_1x_1-z_3x_4+z_1z_3x_5}\\&=\Psi^{(2)}(\textbf{z}_{J^{(2)}},\textbf{x}).
		\end{align}}
			
Notice that
			if $\mathcal H_i=\mathbb C=\mathcal K_i$ for $i=1,2,3,$ $\tilde{A}=J_1AJ_2$   and $X=z_2$ for which $1-a_{22}z_2\neq 0$ in \eqref{FPX}, then we get
			\begin{align}\label{FAA2}
				\mathcal F_{\tilde{A}}(z_2)\nonumber&=\left(\begin{smallmatrix} a_{33} & a_{31}\\a_{13} & a_{11}\end{smallmatrix}\right)+\left(\begin{smallmatrix} a_{32} \\ a_{12}\end{smallmatrix}\right)z_2(1-a_{22}z_2)^{-1}\left(\begin{smallmatrix} a_{23} & a_{21}\end{smallmatrix}\right)\\\nonumber&=\begin{pmatrix} \frac{a_{33}-z_3(a_{22}a_{33}-a_{23}a_{32})}{1-a_{22}z_2} & \frac{a_{31}-z_2(a_{31}a_{22}-a_{21}a_{32})}{1-a_{22}z_2} \\\frac{a_{13}-z_2(a_{13}a_{22}-a_{23}a_{12})}{1-a_{22}z_2} & \frac{a_{11}-z_2(a_{22}a_{11}-a_{12}a_{21})}{1-a_{22}z_2}\end{pmatrix}
				\\&=\left(\begin{smallmatrix} \tilde{A}_{11}(z_2) & \tilde{A}_{12}(z_2)\\\tilde{A}_{21}(z_2) & \tilde{A}_{22}(z_2)\end{smallmatrix}\right).\end{align}
			It follows from \eqref{FAA2} that
			$\tilde{A}_{11}(z_2)=\frac{x_4-z_2x_6}{1-x_2z_2},\tilde{A}_{22}(z_2)=\frac{x_1-z_2x_3}{1-x_2z_2}$ and 
			\begin{align*}
				\det(\mathcal F_{\tilde{A}}(z_2))&=\tilde{A}_{11}(z_2)\tilde{A}_{22}(z_2)-\tilde{A}_{12}(z_2)\tilde{A}_{21}(z_2)\\&=\frac{x_4-z_2x_6}{1-x_2z_2}\frac{x_1-z_2x_3}{1-x_2z_2}-\frac{a_{13}-z_2(a_{13}a_{22}-a_{12}a_{23})}{1-x_2z_2}\frac{a_{31}-z_2(a_{31}a_{22}-a_{32}a_{21})}{1-x_2z_2}\\&=\frac{x_5-z_2x_7}{1-x_2z_2}
			\end{align*} 
			for all $z_2\in \bar{\mathbb D}.$ Also, observe that 
			\begin{align}\label{gafz2}
				\mathcal G_{\mathcal F_{\tilde{A}}(z_2)}(z_1)\nonumber&=\tilde{A}_{11}(z_2)+\tilde{A}_{12}(z_2)z_1(1-\tilde{A}_{22}(z_2)z_1)^{-1}\tilde{A}_{21}(z_2)\\\nonumber &=\frac{x_4-z_2x_6}{1-x_2z_2}+\frac{(a_{13}-z_2(a_{13}a_{22}-a_{12}a_{23}))(a_{31}-z_2(a_{31}a_{22}-a_{32}a_{21}))z_1}{(1-x_2z_2)(1-x_1z_1-x_2z_2+z_2z_1x_3)}\\\nonumber&=\frac{x_4-z_1x_5-z_2x_6+z_1z_2x_7}{1-z_1x_1-z_2x_2+z_2z_1x_3}\\&=\Psi^{(3)}(\textbf{z}_{J^{(3)}},\textbf{x}).
			\end{align}
			It is also important to note that if $\mathcal H_i=\mathbb C=\mathcal K_i$ for $i=1,2,3,$ $\tilde{B}=\tilde{J}_1A\tilde{J}_2$   and $X=z_1$ for which $1-a_{11}z_1\neq 0$ in \eqref{FPX}, then we have
			\begin{align}\label{FAA21}
				\mathcal F_{\tilde{B}}(z_1)\nonumber&=\left(\begin{smallmatrix} a_{22} & a_{23}\\a_{32} & a_{33}\end{smallmatrix}\right)+\left(\begin{smallmatrix} a_{21} \\ a_{31}\end{smallmatrix}\right)z_1(1-a_{11}z_1)^{-1}\left(\begin{smallmatrix} a_{12} & a_{13}\end{smallmatrix}\right)\\\nonumber&=\begin{pmatrix} \frac{a_{22}-z_1(a_{22}a_{11}-a_{21}a_{12})}{1-a_{11}z_1} & \frac{a_{23}-z_1(a_{23}a_{11}-a_{21}a_{13})}{1-a_{11}z_1} \\\frac{a_{32}-z_1(a_{32}a_{11}-a_{31}a_{12})}{1-a_{11}z_1} & \frac{a_{33}-z_1(a_{33}a_{11}-a_{31}a_{13})}{1-a_{11}z_1}\end{pmatrix}
				\\&=\left(\begin{smallmatrix} \tilde{B}_{11}(z_1) &\tilde{B}_{12}(z_1)\\\tilde{B}_{21}(z_1) & \tilde{B}_{22}(z_1)\end{smallmatrix}\right).\end{align}
			Similarly, we deduce that 
			\begin{equation}\label{GFAZ1Z2}\tilde{B}_{11}(z_1)=\frac{x_2-z_1x_3}{1-x_1z_1}, \tilde{B}_{22}(z_1)=\frac{x_4-z_1x_5}{1-x_1z_1}, \det(\mathcal F_{\tilde{B}}(z_1))=\frac{x_6-z_1x_7}{1-x_1z_1}.
			\end{equation}
By the same method as in \eqref{gafz2}, we also conclude that 
\begin{equation}\label{GFF}\mathcal G_{\mathcal F_{\tilde{B}}(z_1)}(z_3)=\Psi^{(2)}(\textbf{z}_{J^{(2)}},\textbf{x}).\end{equation}	

		\begin{thm}\label{mainresult}
			For $\textbf{x}=(x_1,\ldots,x_7)\in \mathbb C^7$ the following are equivalent.
			\begin{enumerate}
				\item[$1$.]  $\textbf{x}=(x_1,\ldots,x_7)\in G_{E(3;3;1,1,1)}$.
				
				\item[$2$.]  There exists a $3\times 3$ matrix $A\in \mathcal M_{3\times 3}(\mathbb C)$ such that \small{$x_{1}=a_{11},x_{2}=a_{22},x_{3}=a_{11}a_{22}-a_{12}a_{21},$}\\$x_{4}=a_{33},x_{5}=a_{11}a_{33}-a_{13}a_{31}, x_{6}=a_{33}a_{22}-a_{23}a_{32}~{\rm{and}}~x_{7}=\operatorname{det}A$ with $$\sup_{(z_2,z_3)\in\bar{\mathbb D}^2} |\mathcal G_{\mathcal F_{A}(z_3)}(z_2)|=\sup_{(z_2,z_3)\in\bar{\mathbb D}^2} |\mathcal G_{A}\left(\left(\begin{smallmatrix}z_2 & 0 \\0 &z_{3}
				\end{smallmatrix}\right)\right)|<1$$ for all $(z_2,z_3)\in\bar{\mathbb D}^2$ with $\det\left(I_{2}-\left(\begin{smallmatrix}a_{22}&a_{23}\\a_{32} & a_{33}
				\end{smallmatrix}\right)\left(\begin{smallmatrix}z_2 & 0 \\0 &z_{3}
				\end{smallmatrix}\right)\right)\neq 0.$

				\item[$2^{\prime}$.]  There exists a $3\times 3$ matrix $\tilde{A}\in \mathcal M_{3\times 3}(\mathbb C)$ such that \small{$x_{1}=\tilde{a}_{33},x_{2}=\tilde{a}_{11},x_{3}=\tilde{a}_{11}\tilde{a}_{33}-\tilde{a}_{13}\tilde{a}_{31},$} \\$x_{4}=\tilde{a}_{22},x_{5}=\tilde{a}_{22}\tilde{a}_{33}-\tilde{a}_{23}\tilde{a}_{32}, x_{6}=\tilde{a}_{11}\tilde{a}_{22}-\tilde{a}_{12}\tilde{a}_{21}~{\rm{and}}~{x}_{7}=\operatorname{det} \tilde{A}$ with $$\sup_{(z_1,z_2)\in\bar{\mathbb D}^2} |\mathcal G_{\mathcal F_{\tilde{A}}(z_2)}(z_1)|=\sup_{(z_1,z_2)\in\bar{\mathbb D}^2} |\mathcal G_{\tilde{A}}\left(\left(\begin{smallmatrix}z_1 & 0 \\0 &z_{2}
				\end{smallmatrix}\right)\right)|<1$$ for all $(z_1,z_2)\in\bar{\mathbb D}^2$ with $\det\left(I_{2}-\left(\begin{smallmatrix}\tilde{a}_{11}&\tilde{a}_{12}\\\tilde{a}_{21} & \tilde{a}_{22}
				\end{smallmatrix}\right)\left(\begin{smallmatrix}z_1 & 0 \\0 &z_{2}
				\end{smallmatrix}\right)\right)\neq 0.$
				
				\item[$2^{\prime\prime}$.] There exists a $3\times 3$ matrix $\tilde{B}\in \mathcal M_{3\times 3}(\mathbb C)$  such that \small{$ x_1=\tilde{b}_{22},x_2=\tilde{b}_{33},x_3=\tilde{b}_{22}\tilde{b}_{33}-\tilde{b}_{23}\tilde{b}_{32},$} \\$ x_4=\tilde{b}_{11},x_{5}=\tilde{b}_{11}\tilde{b}_{22}-\tilde{b}_{12}\tilde{b}_{21}, x_{6}=\tilde{b}_{11}\tilde{b}_{33}-\tilde{b}_{13}\tilde{b}_{31}~{\rm{and}}~x_{7}=\operatorname{det} \tilde{B}$ with $$\sup_{(z_1,z_3)\in\bar{\mathbb D}^2} |\mathcal G_{\mathcal F_{\tilde{B}}(z_1)}(z_3)|=\sup_{(z_1,z_3)\in\bar{\mathbb D}^2} |\mathcal G_{\tilde{B}}\left(\left(\begin{smallmatrix}z_1 & 0 \\0 &z_{3}
				\end{smallmatrix}\right)\right)|<1$$  for all $(z_1,z_3)\in\bar{\mathbb D}^2$  with $\det\left(I_{2}-\left(\begin{smallmatrix}\tilde{b}_{11}&\tilde{b}_{13}\\\tilde{b}_{31} & \tilde{b}_{33}
				\end{smallmatrix}\right)\left(\begin{smallmatrix}z_1 & 0 \\0 &z_{3}
				\end{smallmatrix}\right)\right)\neq 0.$
			\end{enumerate}
		\end{thm}
		As a consequence of the above theorem, we have the following corollary.
		\begin{cor}
			Suppose $\textbf{x}=(x_1,\ldots,x_7)\in \mathbb C^{7}.$ Then
			\begin{enumerate}
				\item $\textbf{x}\in G_{E(3;3;1,1,1)}$ if and only if $(x_4,x_1,x_5,x_2,x_6,x_3,x_7)\in G_{E(3;3;1,1,1)}$, and
				\item  $\textbf{x}\in G_{E(3;3;1,1,1)}$ if and only if $(x_2,x_4,x_6,x_1,x_3,x_5,x_7)\in G_{E(3;3;1,1,1)}.$
			\end{enumerate}
		\end{cor}

We enumerate every characterization of the domain $G_{E(3;3;1,1,1)}$ in the subsequent theorem.
			\begin{thm}\label{mainthm}
				For $\textbf{x}=(x_1,\ldots,x_7)\in \mathbb C^7$ the following are equivalent.
				\begin{enumerate}
					\item[1.] $\textbf{x}=(x_1,\ldots,x_7)\in G_{E(3;3;1,1,1)}$;
					
					\item[2.] $R_{\bf{x}}^{(3;3;1,1,1)}(\textbf{z})= 1-x_1z_1-x_2z_2+x_3z_1z_2-x_4z_3+x_5z_1z_3+x_6z_2z_3-x_7z_1z_2z_3\neq 0$, for all, $z_1,z_2,z_3\in \bar{\mathbb{D}}$;
					
					\item[$3.$] $\textbf{x}_{J^{(1)}}^{\prime}\in G_{E(2;2;1,1)}$ and 
					$\|\Psi^{(1)}(\cdot,\textbf{x})\|_{H^{\infty}(\bar{\mathbb{D}}^{2})}=\|\Psi^{(1)}(\cdot,\textbf{x})\|_{H^{\infty}(\mathbb{T}^{2})}< 1$ and if  we suppose that $x_{7}=x_{6}x_{1},\,\,x_{3}=x_{2}x_{1},\,\,x_{5}=x_{4}x_{1}$ then $|x_1|<1;$
					
					\item[$3^{\prime}.$]
					$\textbf{x}_{J^{(2)}}^{\prime}\in G_{E(2;2;1,1)}$ and 
					$\|\Psi^{(2)}(\cdot,\textbf{x})\|_{H^{\infty}(\bar{\mathbb{D}}^{2})}=\|\Psi^{(2)}(\cdot,\textbf{x})\|_{H^{\infty}(\mathbb{T}^{2})}< 1$  and if  we suppose that $x_{7}=x_{5}x_{2},\,\,x_{3}=x_{2}x_{1},\,\,x_{6}=x_{4}x_{2}$ then $|x_2|<1;$
					
					\item[$3^{\prime\prime}.$]
					$\textbf{x}_{J^{(3)}}^{\prime}\in G_{E(2;2;1,1)}$ and 
					$\|\Psi^{(3)}(\cdot,\textbf{x})\|_{H^{\infty}(\bar{\mathbb{D}}^{2})}=\|\Psi^{(3)}(\cdot,\textbf{x})\|_{H^{\infty}(\mathbb{T}^{2})}< 1$  and if  we suppose that $x_{7}=x_{3}x_{4},\,\,x_{6}=x_{2}x_{4},\,\,x_{5}=x_{4}x_{1}$ then $|x_4|<1$ ;
					
					\item[$4.$] $\left(\frac{x_1-z_3x_5}{1-x_4z_3},\frac{x_2-z_3x_6}{1-x_4z_3},\frac{x_3-z_3x_7}{1-x_4z_3}\right)\in G_{E(2;2;1,1)}$ for all $z_3\in \bar{\mathbb D};$

					\item[$4^{\prime}.$] $\left(\frac{x_2-z_1x_3}{1-x_1z_1},\frac{x_4-z_1x_5}{1-x_1z_1},\frac{x_6-z_1x_7}{1-x_1z_1}\right)\in G_{E(2;2;1,1)}$  for all $z_1\in \bar{\mathbb D}$;
					
					\item[$4^{\prime\prime}.$] $\left(\frac{x_4-z_2x_6}{1-x_2z_2},\frac{x_1-z_2x_3}{1-x_2z_2},\frac{x_5-z_2x_7}{1-x_2z_2}\right)\in G_{E(2;2;1,1)}$  for all $z_2\in \bar{\mathbb D}$;

					\item[$5.$]
					There exists a $2\times 2$ symmetric matrix $A(z_3)$ with $\|A(z_3)\|<1$ such that $$A_{11}(z_3)=\frac{x_1-z_3x_5}{1-x_4z_3},A_{22}(z_3)=\frac{x_2-z_3x_6}{1-x_4z_3}~{\rm{ and}}~ \det(A(z_3))=\frac{x_3-z_3x_7}{1-x_4z_3}~{\rm{for~all}}~ z_3\in \bar{\mathbb D};$$

					\item[$5^{\prime}.$]
					There exists a $2\times 2$ symmetric matrix $B(z_1)$ with $\|B(z_1)\|<1$ such that $$B_{11}(z_1)=\frac{x_2-z_1x_3}{1-x_1z_1},B_{22}(z_1)=\frac{x_4-z_1x_5}{1-x_1z_1}~{\rm{ and }}~\det(B(z_1))=\frac{x_6-z_1x_7}{1-x_1z_1}~{\rm{for ~all }}~z_1\in \bar{\mathbb D};$$
					
					\item[$5^{\prime\prime}.$]
					There exists a $2\times 2$ symmetric matrix $C(z_2)$ with $\|C(z_2)\|<1$ 		such that $$C_{11}(z_2)=\frac{x_4-z_2x_6}{1-x_2z_2},C_{22}(z_2)=\frac{x_1-z_2x_3}{1-x_2z_2}~{\rm{and}}~ \det(C(z_2))=\frac{x_5-z_2x_7}{1-x_2z_2}~{\rm{for~all}}~z_2\in \bar{\mathbb D};$$
					
				\item[$6$.]  There exists a $3\times 3$ matrix $A\in \mathcal M_{3\times 3}(\mathbb C)$ such that \small{$x_{1}=a_{11},x_{2}=a_{22},x_{3}=a_{11}a_{22}-a_{12}a_{21},$}\\$x_{4}=a_{33},x_{5}=a_{11}a_{33}-a_{13}a_{31}, x_{6}=a_{33}a_{22}-a_{23}a_{32}~{\rm{and}}~x_{7}=\operatorname{det}A,$ $$\sup_{(z_2,z_3)\in\bar{\mathbb D}^2} |\mathcal G_{\mathcal F_{A}(z_3)}(z_2)|=\sup_{(z_2,z_3)\in\bar{\mathbb D}^2} |\mathcal G_{A}\left(\left(\begin{smallmatrix}z_2 & 0 \\0 &z_{3}
				\end{smallmatrix}\right)\right)|<1,$$ and $\det\left(I_{2}-\left(\begin{smallmatrix}a_{22}&a_{23}\\a_{32} & a_{33}
				\end{smallmatrix}\right)\left(\begin{smallmatrix}z_2 & 0 \\0 &z_{3}
				\end{smallmatrix}\right)\right)\neq 0$ for all $(z_2,z_3)\in\bar{\mathbb D}^2$.

				\item[$6^{\prime}$.]  There exists a $3\times 3$ matrix $\tilde{A}\in \mathcal M_{3\times 3}(\mathbb C)$ such that \small{$x_{1}=\tilde{a}_{33},x_{2}=\tilde{a}_{11},x_{3}=\tilde{a}_{11}\tilde{a}_{33}-\tilde{a}_{13}\tilde{a}_{31},$} \\$x_{4}=\tilde{a}_{22},x_{5}=\tilde{a}_{22}\tilde{a}_{33}-\tilde{a}_{23}\tilde{a}_{32}, x_{6}=\tilde{a}_{11}\tilde{a}_{22}-\tilde{a}_{12}\tilde{a}_{21}~{\rm{and}}~{x}_{7}=\operatorname{det} \tilde{A},$ $$\sup_{(z_1,z_2)\in\bar{\mathbb D}^2} |\mathcal G_{\mathcal F_{\tilde{A}}(z_2)}(z_1)|=\sup_{(z_1,z_2)\in\bar{\mathbb D}^2} |\mathcal G_{\tilde{A}}\left(\left(\begin{smallmatrix}z_1 & 0 \\0 &z_{2}
				\end{smallmatrix}\right)\right)|<1,$$ and  $\det\left(I_{2}-\left(\begin{smallmatrix}\tilde{a}_{11}&\tilde{a}_{12}\\\tilde{a}_{21} & \tilde{a}_{22}
				\end{smallmatrix}\right)\left(\begin{smallmatrix}z_1 & 0 \\0 &z_{2}
				\end{smallmatrix}\right)\right)\neq 0$ for all $(z_1,z_2)\in\bar{\mathbb D}^2$.
				
				\item[$6^{\prime\prime}$.] There exists a $3\times 3$ matrix $\tilde{B}\in \mathcal M_{3\times 3}(\mathbb C)$  such that \small{$ x_1=\tilde{b}_{22},x_2=\tilde{b}_{33},x_3=\tilde{b}_{22}\tilde{b}_{33}-\tilde{b}_{23}\tilde{b}_{32},$} \\$ x_4=\tilde{b}_{11},x_{5}=\tilde{b}_{11}\tilde{b}_{22}-\tilde{b}_{12}\tilde{b}_{21}, x_{6}=\tilde{b}_{11}\tilde{b}_{33}-\tilde{b}_{13}\tilde{b}_{31}~{\rm{and}}~x_{7}=\operatorname{det} \tilde{B},$  $$\sup_{(z_1,z_3)\in\bar{\mathbb D}^2} |\mathcal G_{\mathcal F_{\tilde{B}}(z_1)}(z_3)|=\sup_{(z_1,z_3)\in\bar{\mathbb D}^2} |\mathcal G_{\tilde{B}}\left(\left(\begin{smallmatrix}z_1 & 0 \\0 &z_{3}
				\end{smallmatrix}\right)\right)|<1,$$ and $\det\left(I_{2}-\left(\begin{smallmatrix}\tilde{b}_{11}&\tilde{b}_{13}\\\tilde{b}_{31} & \tilde{b}_{33}
				\end{smallmatrix}\right)\left(\begin{smallmatrix}z_1 & 0 \\0 &z_{3}
				\end{smallmatrix}\right)\right)\neq 0$  for all $(z_1,z_3)\in\bar{\mathbb D}^2$.
					
				\end{enumerate}
			\end{thm}
			Set $\tilde{x}_1(z_1)=\frac{x_2-z_1x_3}{1-x_1z_1},~\tilde{x}_2(z_1)=\frac{x_4-z_1x_5}{1-x_1z_1},~\tilde{x}_3(z_1)=\frac{x_6-z_1x_7}{1-x_1z_1},~\tilde{y}_1(z_2)=\frac{x_1-z_2x_3}{1-x_2z_2},~\tilde{y}_2(z_2)=\frac{x_4-z_2x_6}{1-x_2z_2},~\\\tilde{y}_3(z_2)=\frac{x_5-z_2x_7}{1-x_2z_2}, \tilde{z}_1(z_3)=\frac{x_1-z_3x_5}{1-x_4z_3},\tilde{z}_2(z_3)=\frac{x_2-z_3x_6}{1-x_4z_3}~{\rm{and}}~\tilde{z}_3(z_3)=\frac{x_3-z_3x_7}{1-x_4z_3}.$  The proof of the following corollary follows easily from the characterization of $G_{E(2;2;1,1)}$ [Theorem $2.2$, \cite{Abouhajar}], therefore we skip the proof.
			
			\begin{cor}	
				For $\textbf{x}=(x_1,\ldots,x_7)\in \mathbb C^7$ the following are equivalent.
				
				\begin{enumerate}
					\item $\textbf{x}=(x_1,\ldots,x_7)\in G_{E(3;3;1,1,1)}$;		
					
					\item $\sup_{z_1\in \bar{\mathbb{D}}}\{|\tilde{x}_1(z_1)-\overline{\tilde{x}_2(z_1) }\tilde{x}_3(z_1)|+\left|\tilde{x}_1(z_1)\tilde{x}_2(z_1)-\tilde{x}_3(z_1)\right|\} < 1-\inf_{z_1\in \bar{\mathbb{D}}}\left|\tilde{x}_2(z_1)\right|^2$;
					\item $\sup_{z_1\in \bar{\mathbb{D}}}\{|\tilde{x}_2(z_1)-\overline{\tilde{x}_1(z_1) }\tilde{x}_3(z_1)|+\left|\tilde{x}_1(z_1)\tilde{x}_2(z_1)-\tilde{x}_3(z_1)\right|\} < 1-\inf_{z_1\in \bar{\mathbb{D}}}\left|\tilde{x}_1(z_1)\right|^2$;		
					\item $\sup_{z_1\in \bar{\mathbb{D}}}\left\{\left|\tilde{x}_1(z_1)\right|^2-\left|\tilde{x}_2(z_1)\right|^2+\left|\tilde{x}_3(z_1)\right|^2+2\left|\tilde{x}_2(z_1)-\overline{\tilde{x}_1(z_1) }\tilde{x}_3(z_1)\right|\right\} < 1$ and if we assume  $\tilde{x}_1(z_1)\tilde{x}_2(z_1)=\tilde{x}_3(z_1)$ for all $z_1\in \bar{\mathbb{D}}$, then $\sup_{z_1\in \bar{\mathbb{D}}}|\tilde{x}_2(z_1)|<1$;
					\item $\sup_{z_1\in \bar{\mathbb{D}}}\left \{-\left|\tilde{x}_1(z_1)\right|^2+\left|\tilde{x}_2(z_1)\right|^2+\left|\tilde{x}_3(z_1)\right|^2+2\left|\tilde{x}_1(z_1)-\overline{\tilde{x}_2 (z_1)}\tilde{x}_3(z_1)\right| \right\} < 1$ and if we assume $\tilde{x}_1(z_1)\tilde{x}_2(z_1)=\tilde{x}_3(z_1)$  for all $z_1\in \bar{\mathbb{D}}$, then $\sup_{z_1\in \bar{\mathbb{D}}}|\tilde{x}_1(z_1)|<1$;
					\item $\sup_{z_1\in \bar{\mathbb{D}}}\{\left|\tilde{x}_1(z_1)\right|^2+\left|\tilde{x}_2(z_1)\right|^2-\left|\tilde{x}_3(z_1)\right|^2+2\left|\tilde{x}_1(z_1) \tilde{x}_2(z_1)-\tilde{x}_3(z_1)\right|\} < 1$;
					\item $\sup_{z_1\in \bar{\mathbb{D}}}\left\{\left|\tilde{x}_1(z_1)-\overline{\tilde{x}_2(z_1)} \tilde{x}_3(z_1)\right|+\left|\tilde{x}_2(z_1)-\overline{\tilde{x}_1(z_1)} \tilde{x}_3(z_1)\right| \right\} < 1-\inf_{z_1\in \bar{\mathbb{D}}}\left|\tilde{x}_3(z_1)\right|^2$;
					\item $\sup_{z_2\in \bar{\mathbb{D}}}\{|\tilde{y}_1(z_2)-\overline{\tilde{y}_2(z_2)} \tilde{y}_3(z_2)|+\left|\tilde{y}_1(z_2) \tilde{y}_2(z_2)-\tilde{y}_3(z_2)\right|\} < 1-\inf_{z_2\in \bar{\mathbb{D}}}\left|\tilde{y}_2(z_2)\right|^2$;
					\item $\sup_{z_2\in \bar{\mathbb{D}}}\left\{\left|\tilde{y}_2(z_2)-\overline{\tilde{y}_1(z_2)} \tilde{y}_3(z_2)\right|+\left|\tilde{y}_1(z_2) \tilde{y}_2(z_2)-\tilde{y}_3(z_2)\right|\right\} < 1-\inf_{z_2\in \bar{\mathbb{D}}}\left|\tilde{y}_1(z_2)\right|^2$;
					\item $\sup_{z_2\in \bar{\mathbb{D}}}\left\{\left|\tilde{y}_1(z_2)\right|^2-\left|\tilde{y}_2(z_2)\right|^2+\left|\tilde{y}_3(z_2)\right|^2+2\left|\tilde{y}_2(z_2)-\overline{\tilde{y}_1(z_2)} \tilde{y}_3(z_2)\right|\right\} < 1$, and if we assume $\tilde{y}_1(z_2)\tilde{y}_2(z_2)=\tilde{y}_3(z_2)$  for all $z_2\in \bar{\mathbb{D}}$, then $\sup_{z_2\in \bar{\mathbb{D}}}|\tilde{y}_2(z_2)|<1$;
					\item $\sup_{z_2\in \bar{\mathbb{D}}}\left\{-\left|\tilde{y}_1(z_2)\right|^2+\left|\tilde{y}_2(z_2)\right|^2+\left|\tilde{y}_3(z_2)\right|^2+2\left|\tilde{y}_1(z_2)-\overline{\tilde{y}_2(z_2)} \tilde{y}_3(z_2)\right|\right\} < 1$, and if we assume $\tilde{y}_1(z_2)\tilde{y}_2(z_2)=\tilde{y}_3(z_2)$  for all $z_2\in \bar{\mathbb{D}}$, then $\sup_{z_2\in \bar{\mathbb{D}}}|\tilde{y}_1(z_2)|<1$;
					\item $\sup_{z_2\in \bar{\mathbb{D}}}\{\left|\tilde{y}_1(z_2)\right|^2+\left|\tilde{y}_2(z_2)\right|^2-\left|\tilde{y}_3(z_2)\right|^2+2\left|\tilde{y}_1(z_2) \tilde{y}_2(z_2)-\tilde{y}_3(z_2)\right|\} < 1$;
					\item $\sup_{z_2\in \bar{\mathbb{D}}}\left\{\left|\tilde{y}_1(z_2)-\overline{\tilde{y}_2(z_2)} \tilde{y}_3(z_2)\right|+\left|\tilde{y}_2(z_2)-\overline{\tilde{y}_1(z_2)} \tilde{y}_3(z_2)\right|\right\} < 1-\inf_{z_2\in \bar{\mathbb{D}}}\left|\tilde{y}_3(z_2)\right|^2$;
					\item $\sup_{z_3\in \bar{\mathbb{D}}}\{|\tilde{z}_1(z_3)-\overline{\tilde{z}_2(z_3)} \tilde{z}_3(z_3)|+\left|\tilde{z}_1(z_3) \tilde{z}_2(z_3)-\tilde{z}_3(z_3)\right|\} < 1-\inf_{z_3\in \bar{\mathbb{D}}}\left|\tilde{z}_2(z_3)\right|^2$;
					\item $\sup_{z_3\in \bar{\mathbb{D}}}\left\{\left|\tilde{z}_2(z_3)-\overline{\tilde{z}_1(z_3)} \tilde{z}_3(z_3)\right|+\left|\tilde{z}_1(z_3) \tilde{z}_2(z_3)-\tilde{z}_3(z_3)\right|\right\} < 1-\inf_{z_3\in \bar{\mathbb{D}}}\left|\tilde{z}_1(z_3)\right|^2$;
					\item $\sup_{z_3\in \bar{\mathbb{D}}}\left\{\left|\tilde{z}_1(z_3)\right|^2-\left|\tilde{z}_2(z_3)\right|^2+\left|\tilde{z}_3(z_3)\right|^2+2\left|\tilde{z}_2(z_3)-\overline{\tilde{z}_1(z_3)} \tilde{z}_3(z_3)\right|\right\} < 1$, and if we assume $\tilde{z}_1(z_3)\tilde{z}_2(z_3)=\tilde{z}_3(z_3)$  for all $z_3\in \bar{\mathbb{D}}$, then $\sup_{z_3\in \bar{\mathbb{D}}}|\tilde{z}_2(z_3)|<1$;
					\item $\sup_{z_3\in \bar{\mathbb{D}}}\left\{-\left|\tilde{z}_1(z_3)\right|^2+\left|\tilde{z}_2(z_3)\right|^2+\left|\tilde{z}_3(z_3)\right|^2+2\left|\tilde{z}_1(z_3)-\overline{\tilde{z}_2(z_3)} \tilde{z}_3(z_3)\right|\right\} < 1$, and if we assume $\tilde{z}_1(z_3)\tilde{z}_2(z_3)=\tilde{z}_3(z_3)$  for all $z_3\in \bar{\mathbb{D}}$, then $\sup_{z_3\in \bar{\mathbb{D}}}|\tilde{z}_1(z_3)|<1$;
					\item $\sup_{z_3\in \bar{\mathbb{D}}}\{\left|\tilde{z}_1(z_3)\right|^2+\left|\tilde{z}_2(z_3)\right|^2-\left|\tilde{z}_3(z_3)\right|^2+2\left|\tilde{z}_1(z_3) \tilde{z}_2(z_3)-\tilde{z}_3(z_3)\right|\} < 1$;
					\item $\sup_{z_3\in \bar{\mathbb{D}}}\left\{\left|\tilde{z}_1(z_3)-\overline{\tilde{z}_2(z_3)} \tilde{z}_3(z_3)\right|+\left|\tilde{z}_2(z_3)-\overline{\tilde{z}_1(z_3)} \tilde{z}_3(z_3)\right|\right\} < 1-\inf_{z_3\in \bar{\mathbb{D}}}\left|\tilde{z}_3(z_3)\right|^2$.
					
				\end{enumerate}
			\end{cor}

We derive the following theorem by using first realization formula.			
			\begin{thm}\label{matix A}

Let $A=((a_{ij}))^3_{leq i,j=1}\in \mathcal M_{3\times 3}(\mathbb C)$ with $\|A\|<1.$  Set   \begin{eqnarray}\label{av1}
x_1=a_{11}, x_2=a_{22}, x_3=\det \left(\begin{smallmatrix} a_{11} & a_{12}\\
					a_{21} & a_{22}
				\end{smallmatrix}\right), x_4=a_{33}, x_5=\det \left(\begin{smallmatrix}
					a_{11} & a_{13}\\
					a_{31} & a_{33}
				\end{smallmatrix}\right), x_6=\det  \left(\begin{smallmatrix}
					a_{22} & a_{23}\\
					a_{32} & a_{33}\end{smallmatrix}\right) ~{\rm{and}}~x_7=\det A.\end{eqnarray} Then $\textbf{x}=(x_1,\ldots,x_7)\in G_{E(3;3;1,1,1)}.$				
			\end{thm}

\begin{proof}
Note that for all $z_2,z_3\in \mathbb C$ with $\det\left(I_{2}-\left(\begin{smallmatrix}a_{22}&a_{23}\\a_{32} & a_{33}
				\end{smallmatrix}\right)\left(\begin{smallmatrix}z_2 & 0 \\0 &z_{3}
				\end{smallmatrix}\right)\right)\neq 0,$ we have
				\small{\begin{align}\label{mathcalG11}
						\mathcal G_{A}\left(\left(\begin{smallmatrix}z_2 & 0 \\0 &z_{3}
						\end{smallmatrix}\right)\right)\nonumber &=a_{11}+\left(\begin{smallmatrix}a_{12}&a_{13}
						\end{smallmatrix}\right)\left(\begin{smallmatrix}z_2 & 0 \\0 &z_{3}
						\end{smallmatrix}\right)\left(I_{2}-\left(\begin{smallmatrix}a_{22}&a_{23}\\a_{32} & a_{33}
						\end{smallmatrix}\right)\left(\begin{smallmatrix}z_2 & 0 \\0 &z_{3}
						\end{smallmatrix}\right)\right)^{-1}\left(\begin{smallmatrix}a_{21}\\a_{31}
						\end{smallmatrix}\right)\\\nonumber&=\frac{a_{11}-z_2(a_{11}a_{22}-a_{12}a_{21})-z_3(a_{33}a_{11}-a_{13}a_{31})+z_2z_3\det(A)}{1-z_2a_{22}-z_3a_{33}+z_2z_3(a_{22}a_{33}-a_{23}a_{32})}\\\nonumber&=\frac{x_1-z_2x_3-z_3x_5+z_2z_3x_7}{1-z_2x_2-z_3x_4+z_2z_3x_6}\\&=\Psi^{(1)}(\textbf{z}_{J^{(1)}},\textbf{x}).
				\end{align}} 
For $z_2,z_3 \in \bar{\mathbb{D}}$ we have 	$\det\left(I_{2}-\left(\begin{smallmatrix}a_{22}&a_{23}\\a_{32} & a_{33}
				\end{smallmatrix}\right)\left(\begin{smallmatrix}z_2 & 0 \\0 &z_{3}
				\end{smallmatrix}\right)\right)\neq 0,$ which implies that 	$\textbf{x}_{J^{(1)}}^{\prime}\in G_{E(2;2;1,1)}$. From Proposition \ref{matrix AA}, it yields that $|\mathcal G_{A}\left(\left(\begin{smallmatrix}z_2 & 0 \\0 &z_{3}
				\end{smallmatrix}\right)\right)|<1$ for all $z_2,z_3\in \bar{\mathbb D}$.	 Hence, from Theorem \ref{phi}, we conclude that $\textbf{x}\in G_{E(3;3;1,1,1)}.$ This completes the proof.

\end{proof}			
We provide the another proof  of Theorem \ref{matix A} by applying the second realization formula.		
			\begin{thm}\label{another}
				Let $A\in \mathcal M_{3\times 3}(\mathbb C)$ with $\|A\|<1.$  Set $$x_{1}=a_{11},x_{2}=a_{22},x_{3}=\operatorname{det}A_{12},x_{4}=a_{33},x_{5}=\operatorname{det}A_{13}, x_{6}=\operatorname{det}A_{23}~{\rm{and}}~x_{7}=\operatorname{det}A.$$ Then $\textbf{x}=(x_1,\ldots,x_7)\in G_{E(3;3;1,1,1)}.$
			\end{thm}
			\begin{proof}
				Note that
				if $\mathcal H_i=\mathbb C=\mathcal K_i$ for $i=1,2,3,$ $P=A=((a_{ij}))_{i,j=1}^{3}$   and $X=z_3$ for which $1-a_{33}z_3\neq 0$, in \eqref{FPX}, then we have 
				\begin{align}\label{FAA}
					\mathcal F_{A}(z_3)\nonumber&=\left(\begin{smallmatrix} a_{11} & a_{12}\\a_{21} & a_{22}\end{smallmatrix}\right)+\left(\begin{smallmatrix} a_{13} \\ a_{23}\end{smallmatrix}\right)z_3(1-a_{33}z_3)^{-1}\left(\begin{smallmatrix} a_{31} & a_{32}\end{smallmatrix}\right)\\\nonumber&=\begin{pmatrix} \frac{a_{11}-z_3(a_{11}a_{33}-a_{13}a_{31})}{1-a_{33}z_3} & \frac{a_{12}-z_3(a_{12}a_{33}-a_{13}a_{32})}{1-a_{33}z_3} \\\frac{a_{21}-z_3(a_{21}a_{33}-a_{23}a_{31})}{1-a_{33}z_3} & \frac{a_{22}-z_3(a_{22}a_{33}-a_{23}a_{32})}{1-a_{33}z_3}\end{pmatrix}
					\\&=\left(\begin{smallmatrix} A_{11}(z_3) & A_{12}(z_3)\\A_{21}(z_3) & A_{22}(z_3)\end{smallmatrix}\right).\end{align}
				Here we employ the  characterization of $G_{E(2;2;1,1)}$  which implies that a point $(x_1,x_2,x_3)\in G_{E(2;2;1,1)}$ if and only if there exists a $2\times 2$ matrix $B=((b_{ij}))_{i,j=1}^{2}$ with $\|B\|<1$ such that $x_1=b_{11},x_2=b_{22}$ and $x_3=\det B.$ From \eqref{FAA}, we  have
				$A_{11}(z_3)=\frac{x_1-z_3x_5}{1-x_4z_3},A_{22}(z_3)=\frac{x_2-z_3x_6}{1-x_4z_3}$ and 
				\begin{align*}
					\det(\mathcal F_{A}(z_3))&=A_{11}(z_3)A_{22}(z_3)-A_{12}(z_3)A_{21}(z_3)\\&=\frac{x_1-z_3x_5}{1-x_4z_3}\frac{x_2-z_3x_6}{1-x_4z_3}-\frac{a_{12}-z_3(a_{12}a_{33}-a_{13}a_{32})}{1-x_4z_3}\frac{a_{21}-z_3(a_{21}a_{33}-a_{23}a_{31})}{1-x_4z_3}\\&=\frac{x_3-z_3x_7}{1-x_4z_3}
				\end{align*} 
				for all $z_3\in \bar{\mathbb D}.$ For the purpose of demonstrating $\left(\frac{x_1-z_3x_5}{1-x_4z_3},\frac{x_2-z_3x_6}{1-x_4z_3},\frac{x_3-z_3x_7}{1-x_4z_3}\right)\in G_{E(2;2;1,1)}$ for all $z_3\in \bar{\mathbb D}$, we must show $\|\mathcal F_{A}(z_3)\|<1$ for all $z_3\in \bar{\mathbb D}$ [Theorem $2.2$, \cite{Abouhajar}]. It follows from Proposition \ref{matrix A} that $\|\mathcal F_{A}(z_3)\|<1$ for all $z_3\in \bar{\mathbb D}$ and for all $A\in \mathcal M_{3\times3}(\mathbb C)$ with $\|A\|<1.$ Using the  characterization of $G_{E(2;2;1,1)}$, we infer that $\left(\frac{x_1-z_3x_5}{1-x_4z_3},\frac{x_2-z_3x_6}{1-x_4z_3},\frac{x_3-z_3x_7}{1-x_4z_3}\right)\in G_{E(2;2;1,1)}$ for all $z_3\in \bar{\mathbb D}$ and hence from Theorem \ref{char}, we have $\mathbf{x}=(x_1,\ldots,x_7)\in G_{E(3;3;1,1,1)}.$ This completes the proof.
			\end{proof}	

			\begin{thm}\label{anotherp}
				Let $A$ be a $3\times 3$ matrix.   Set $$x_{1}=a_{11},x_{2}=a_{22},x_{3}=\operatorname{det}A_{12},x_{4}=a_{33},x_{5}=\operatorname{det}A_{13}, x_{6}=\operatorname{det}A_{23}~{\rm{and}}~x_{7}=\operatorname{det}A.$$  Assume that $\|\mathcal F_{A}(z_3)\|<1$ for all $z_3\in \mathbb{\bar{\mathbb D}}.$  Then $\textbf{x}=(x_1,\ldots,x_7)\in G_{E(3;3;1,1,1)}$. 
			\end{thm}
			\begin{proof}
Suppose that $\|\mathcal F_{A}(z_3)\|<1$ for all $z_3\in \bar{\mathbb D}. $ Then $1-A_{22}(z_3)z_2 \neq 0$ for all $z_2\in \bar{\mathbb D}$ which implies that 
				$$1-x_2z_2-x_4z_3+x_6z_2z_3\neq 0$$  for all $z_2,z_3\in \mathbb {\bar{D}}$ and hence by definition of tetrablock \cite{Abouhajar} we get $(x_2,x_4,x_6)\in G_{E(2;2;1,1)}.$ 
Since $\|\mathcal F_{A}(z_3)\|<1$ for all $z_3\in \mathbb{\bar{\mathbb D}}$ and $A_{11}(z_3)=\frac{x_1-z_3x_5}{1-x_4z_3},A_{22}(z_3)=\frac{x_2-z_3x_6}{1-x_4z_3}$ and 
				$\det(\mathcal F_{A}(z_3))=\frac{x_3-z_3x_7}{1-x_4z_3}$ for all $z_3\in \bar{\mathbb D},$ it follows from the characterization of tetrablock [Theorem $2.2$, \cite{Abouhajar}] that $$\left(\frac{x_1-z_3x_5}{1-x_4z_3},\frac{x_2-z_3x_6}{1-x_4z_3},\frac{x_3-z_3x_7}{1-x_4z_3}\right)\in G_{E(2;2;1,1)}~{\rm{for~all~}} z_3\in \bar{\mathbb D}.$$ Thus we conclude from  Theorem \ref{char} that $\mathbf{x}=(x_1,\ldots,x_7)\in G_{E(3;3;1,1,1)}.$ This completes the proof.
			\end{proof}		
\begin{rem}
By Theorem \ref{matix A} we notice that if $A\in \mathcal M_{3\times 3}(\mathbb C)$ with $\|A\|<1,$ then $\textbf{x}=(x_1,\ldots,x_7)$ given by \eqref{av1} belongs to $ G_{E(3;3;1,1,1)}$. It is interesting to know whether the converse of the Theorem  \ref{matix A} is true or not. Computing the sup norm of $\Psi^{(i)}, 1\leq i\leq 3,$ over $\bar{\mathbb D}^2$ is extremely difficult. In our opinion, Theorem \ref{matix A1} might be helpful in deciding whether a point $\textbf{x}=(x_1,\ldots,x_7)\in G_{E(3;3;1,1,1)}$ is equivalent to the existence of a $3\times 3$ matrix $A\in \mathcal M_{3\times 3}(\mathbb C)$ with $\|A\|<1$ such that \small{$$x_1=a_{11}, x_2=a_{22}, x_3=\det \left(\begin{smallmatrix} a_{11} & a_{12}\\
					a_{21} & a_{22}
				\end{smallmatrix}\right), x_4=a_{33}, x_5=\det \left(\begin{smallmatrix}
					a_{11} & a_{13}\\
					a_{31} & a_{33}
				\end{smallmatrix}\right), x_6=\det  \left(\begin{smallmatrix}
					a_{22} & a_{23}\\
					a_{32} & a_{33}\end{smallmatrix}\right) ~{\rm{and}}~x_7=\det A.$$}

\end{rem}			
			
\subsection{Characterization of $ \Gamma_{E(3;3;1,1,1)}$ }
			Now, we characterize the domain  $ \Gamma_{E(3;3;1,1,1)}$ using the rational functions. The proof of the following theorem is similar to the Theorem \ref{phi}. Therefore, we skip the proof.
			\begin{thm}\label{phii1}
				For $i=1,2,3,$	suppose $\textbf{z}_{J^{(i)}}=(z_{j_1},z_{j_2})\in \mathbb{C}^{2}$ for $j_1,j_2\in J^{(i)}$ with $j_1<j_2.$ Then $\textbf{x}\in \Gamma_{E(3;3;1,1,1)}$ if and only if it satisfies one of the following condition:
				\begin{enumerate}
					\item $\textbf{x}_{J^{(1)}}^{\prime}\in \Gamma_{E(2;2;1,1)}$ and 
					$\|\Psi^{(1)}(\cdot,\textbf{x})\|_{H^{\infty}(\mathbb{D}^{2})}\leq 1$ and if   $x_{7}=x_{6}x_{1},\,\,x_{3}=x_{2}x_{1},\,\,x_{5}=x_{4}x_{1}$ then $|x_1|\leq1;$
					
					\item  $\textbf{x}_{J^{(2)}}^{\prime}\in \Gamma_{E(2;2;1,1)}$ and 
					$\|\Psi^{(2)}(\cdot,\textbf{x})\|_{H^{\infty}(\mathbb{D}^{2})}\leq 1$ and if  $x_{7}=x_{5}x_{2},\,\,x_{3}=x_{2}x_{1},\,\,x_{6}=x_{4}x_{2}$ then $|x_2|\leq1;$

					\item  $\textbf{x}_{J^{(3)}}^{\prime}\in \Gamma_{E(2;2;1,1)}$ and 
					$\|\Psi^{(3)}(\cdot,\textbf{x})\|_{H^{\infty}(\mathbb{D}^{2})}\leq 1$ and if  $x_{7}=x_{3}x_{4},\,\,x_{6}=x_{2}x_{4},\,\,x_{5}=x_{4}x_{1}$ then $|x_4|\leq 1.$
					
				\end{enumerate}
			\end{thm} 
			
			The following theorem describes a condition for membership of an element  $\textbf{x}=(x_1,\ldots,x_7)\in \Gamma_{E(3;3;1,1,1)}$ to its preceding level $\Gamma_{E(2;2;1,1)}$. The proof of the following theorem is similar to  Theorem \ref{char}. Therefore, we omit the proof. 
			\begin{thm}\label{charr}
				Suppose $\textbf{x}=(x_1,\ldots,x_7)\in \mathbb C^{7}.$ Then $\textbf{x}\in 
				\Gamma_{E(3;3;1,1,1)}$ if and only if
				it satisfies one of the following conditions:
				\begin{enumerate}
					\item $\left(\frac{x_2-z_1x_3}{1-x_1z_1},\frac{x_4-z_1x_5}{1-x_1z_1},\frac{x_6-z_1x_7}{1-x_1z_1}\right)\in \Gamma_{E(2;2;1,1)}$ for all $z_1\in \mathbb D;$
					
					\item $\left(\frac{x_1-z_2x_3}{1-x_2z_2},\frac{x_4-z_2x_6}{1-x_2z_2},\frac{x_5-z_2x_7}{1-x_2z_2}\right)\in \Gamma_{E(2;2;1,1)}$  for all $z_2\in \mathbb D;$
					
					\item $\left(\frac{x_1-z_3x_5}{1-x_4z_3},\frac{x_2-z_3x_6}{1-x_4z_3},\frac{x_3-z_3x_7}{1-x_4z_3}\right)\in \Gamma_{E(2;2;1,1)}$  for all $z_3\in \mathbb D.$
				\end{enumerate}
			\end{thm}
			We give the characterization of $\Gamma_{E(3;3;1,1,1)}$ analogous to Theorem \ref{matix A}.  To characterize  $\Gamma_{E(3;3;1,1,1)}$, we need the following lemmas and proposition. The proof of the  proposition  below is similar to the Proposition \ref{matrix AA}. Therefore, we skip the proof.
			\begin{prop}\label{matrix AAA}
				Suppose  $A\in \mathcal M_{3\times 3}(\mathbb C)$ with $\|A\|\leq1$. Then 
				\begin{align}\label{AAAAA norm}
					1-\overline{\mathcal G_{A}\left(\left(\begin{smallmatrix}w_2 & 0 \\0 &w_{3}
						\end{smallmatrix}\right)\right)}\mathcal G_{A}\left(\left(\begin{smallmatrix}z_2 & 0 \\0 &z_{3}
					\end{smallmatrix}\right)\right)\nonumber&=\overline{{\tilde{\gamma_1}(w_2,w_3)}}(1-\bar{w}_2z_2)\tilde{\gamma}_1(z_2,z_3)+\overline{{\tilde{\gamma_2}(w_2,w_3)}}(1-\bar{w}_3z_3)\tilde{\gamma}_2(z_2,z_3)\\&+{\tilde{\eta}(w_2,w_3)}^*(I_{3}-A^*A)\tilde{\eta}(z_2,z_3)
				\end{align}
				for all $z_2,w_2,z_3,w_3\in \mathbb C$ with $\det\left(I_2-\left(\begin{smallmatrix}a_{22} & a_{23} \\a_{32} &a_{33}
				\end{smallmatrix}\right)\left(\begin{smallmatrix}z_2 & 0 \\0 &z_{3}
				\end{smallmatrix}\right)\right)\neq 0$ and $\det\left(I_2-\left(\begin{smallmatrix}a_{22} & a_{23} \\a_{32} &a_{33}
				\end{smallmatrix}\right)\left(\begin{smallmatrix}w_2 & 0 \\0 &w_{3}
				\end{smallmatrix}\right)\right)\neq 0$. Moreover, $|\mathcal G_{A}\left(\left(\begin{smallmatrix}z_2 & 0 \\0 &z_{3}
				\end{smallmatrix}\right)\right)|\leq1$ for all $z_2,z_3\in \bar{\mathbb D}$ with $\det\left(I_2-\left(\begin{smallmatrix}a_{22} & a_{23} \\a_{32} &a_{33}
				\end{smallmatrix}\right)\left(\begin{smallmatrix}z_2 & 0 \\0 &z_{3}
				\end{smallmatrix}\right)\right)\neq 0.$
				
			\end{prop}
			
			\begin{lem}\label{pro}
				$G_{E(3;3;1,1,1)} \subsetneq \Gamma_{E(3;3;1,1,1)}.$
			\end{lem}
			\begin{proof}
				By using Theorem \ref{R_3} and Corollary \ref{R_33}, it is easy to see that the element $(1,1,1,1,1,1,1)\in \Gamma_{E(3;3;1,1,1)}$ but $(1,1,1,1,1,1,1)$ is not a member of $G_{E(3;3;1,1,1)}$, which shows that $G_{E(3;3;1,1,1)}\subsetneq \Gamma_{E(3;3;1,1,1)}.$
			\end{proof}
			
			\begin{lem}\label{proo}
				$G_{E(2;2;1,1)} \subsetneq \Gamma_{E(2;2;1,1)}.$
			\end{lem}
			\begin{proof}
				Note that $(1,1,1)\in \Gamma_{E(2;2;1,1)}$ but $(1,1,1)\notin G_{E(2;2;1,1)}$. This shows that $G_{E(2;2;1,1)} \subsetneq \Gamma_{E(2;2;1,1)}.$
			\end{proof}
			\begin{lem}\label{tetraa1} Let ${\bf{x}}\in \Gamma_{E(3;3;1,1,1)}\setminus G_{E(3;3;1,1,1)}.$ Then $R_{\bf{x}}^{(3;3;1,1,1)}(\textbf{z})=0$ for some ${\bf{z}}=(z_1,z_2,z_3) \in \mathbb T^3.$
			\end{lem}
			\begin{proof}
				We  prove it by contradiction. Suppose $R_{\bf{x}}^{(3;3;1,1,1)}(\textbf{z})\neq0$ for all ${\bf{z}}=(z_1,z_2,z_3) \in \mathbb T^3.$  Theorem \ref{R_3} and Corollary \ref{R_33} lead to the conclusion that ${\bf{x}}=(x_1,\ldots,x_7)\in G_{E(3;3;1,1,1)},$ which is a contradiction. This shows that $R_{\bf{x}}^{(3;3;1,1,1)}(\textbf{z})=0$ for some ${\bf{z}}=(z_1,z_2,z_3) \in \mathbb T^3.$
			\end{proof}
			
			The proof of the following lemma is similar to Lemma \ref{tetraa1}, therefore we skip the proof.
			
			\begin{lem}\label{tetraa11} 
				Suppose ${\bf{y}}=(y_1,y_2,y_3)$ belongs to $ \Gamma_{E(2;2;1,1)}\setminus G_{E(2;2;1,1)}.$ Then $R_{\bf{y}}^{(2;2;1,1)}({\bf{z^{\prime}}})=0$ for some ${\bf{z^{\prime}}}=(z_1,z_2) \in \mathbb T^2.$ 
			\end{lem}

			\begin{lem}\label{tetra} Let ${\bf{y}}=(y_1,y_2,y_3)\in \Gamma_{E(2;2;1,1)}.$ Then \begin{equation}\label{tetr}
					1-y_1z_1-y_2z_2+y_3z_1z_2\neq 0
				\end{equation} for some $z_1,z_2 \in \mathbb T.$
			\end{lem}
			\begin{proof}
				We prove it by contradiction. Assume that
				\begin{equation}\label{y1}1-y_1z_1-y_2z_2+y_3z_1z_2= 0
				\end{equation} for all $z_1,z_2 \in \mathbb T.$ Then by putting $z_1=z_2=1$, $z_1=-z_2=1$, $-z_1=z_2=1$ and $z_1=z_2=-1$ in \eqref{y1} respectively, we get the following identities:
				\begin{equation}
					1-y_1-y_2+y_3=0, 1-y_1+y_2-y_3=0, 1+y_1-y_2-y_3=0~{\rm{and}}~1+y_1+y_2+y_3=0
				\end{equation}
				respectively. By solving the above equation, we have $y_1=1$ and $y_1=-1$, which is a contradiction. Hence, we have $1-y_1z_1-y_2z_2+y_3z_1z_2\neq 0$ for some $z_1,z_2 \in \mathbb T.$
			\end{proof}
\begin{rem}
In the previous Lemma, we observe that for ${\bf{y}}=(y_1,y_2,y_3)\in \Gamma_{E(2;2;1,1)}, $ there exists a point $(z_1^{0},z_2^{0})\in \mathbb T^2$ such that $1-y_1z_1^{0}-y_2z_2^{0}+y_3z_1^{0}z_2^{0}\neq 0.$ We now show that if  $1-y_1z_1^{0}-y_2z_2^{0}+y_3z_1^{0}z_2^{0}\neq 0$ for ${\bf{y}}=(y_1,y_2,y_3)\in \Gamma_{E(2;2;1,1)}, $ then $1-y_1z_1^{0}\neq 0.$ To see this, it suffices to show that if $1-y_1z_1^{0}=0,$ then $1-y_1z_1^{0}-y_2z_2^{0}+y_3z_1^{0}z_2^{0}=0$ for  ${\bf{y}}=(y_1,y_2,y_3)\in \Gamma_{E(2;2;1,1)}.$ Note that  $1-y_1z_1^{0}=0,$ implies $y_1=\frac{1}{z_1^{0}}$ and $|y_1|=1.$ As   ${\bf{y}}=(y_1,y_2,y_3)\in \Gamma_{E(2;2;1,1)},$ then by characterization of $\Gamma_{E(2;2;1,1)}$ \cite{Abouhajar}, we have 
\begin{equation}\label{blockc}|y_2 -\bar{y}_1y_3| + |y_1y_2- y_3| \leq 1-|y_1|^2.\end{equation}
Since $|y_1|=1,$ it follows from \eqref{blockc} that $y_1y_2=y_3$ and hence  $1-y_1z_1^{0}-y_2z_2^{0}+y_3z_1^{0}z_2^{0}=0$.
				Thus, $$1-y_1z_1^{0}-y_2z_2^{0}+y_3z_1^{0}z_2^{0}\neq 0$$ for ${\bf{y}}=(y_1,y_2,y_3)\in \Gamma_{E(2;2;1,1)}, $ gives  $1-y_1z_1^{0}\neq 0.$ 
				Using the same argument, we also show that if  $1-y_1z_1^{0}-y_2z_2^{0}+y_3z_1^{0}z_2^{0}\neq 0$ for ${\bf{y}}=(y_1,y_2,y_3)\in \Gamma_{E(2;2;1,1)}, $ then $1-y_2z_2^{0}\neq 0.$ 
			\end{rem}
\begin{lem}\label{tetraa} Let ${\bf{x}}\in \Gamma_{E(3;3;1,1,1)}.$ Then $R_{\bf{x}}^{(3;3;1,1,1)}(\textbf{z})\neq0$ for some ${\bf{z}}=(z_1,z_2,z_3) \in \mathbb T^3.$
			\end{lem}
\begin{proof}
We  prove this lemma by using contradiction. Assume that
\begin{align}\label{gamm1}
R_{\bf{x}}^{(3;3;1,1,1)}(z)\nonumber&=1-x_1z_1-x_2z_2+x_3z_1z_2-x_4z_3+x_5z_1z_3+x_6z_2z_3-x_7z_1z_2z_3\\&=0
\end{align}
for all $z_1,z_2,z_3\in \mathbb{T}.$ Now by putting $(z_1,z_2,z_3)=(1,1,1),(1,1,-1),(1,-1,1),(1,-1,-1),(-1,1,1),\\(-1,1,-1),
(-1,-1,1),(-1,-1,-1)$ in \eqref{gamm1} respectively, we have the following system of equation:
\begin{equation}
\begin{aligned}
1-x_1-x_2+x_3-x_4+x_5+x_6-x_7&=0,\\
1-x_1-x_2+x_3+x_4-x_5-x_6+x_7&=0,\\
1-x_1+x_2-x_3-x_4+x_5-x_6+x_7&=0,\\
1-x_1+ x_2-x_3+x_4-x_5+x_6-x_7&=0,\\
1+x_1-x_2-x_3-x_4-x_5+x_6+x_7&=0,\\
1+x_1-x_2-x_3+x_4+x_5-x_6-x_7&=0,\\
1+x_1+x_2+x_3-x_4-x_5-x_6-x_7&=0,\\
1+x_1+x_2+x_3+x_4+x_5+x_6+x_7&=0.  
\end{aligned}
\end{equation}
Solving the above system, we get $x_3=1$ and $x_3=-1$, which is a contradiction.  This completes the proof.
\end{proof}

	\begin{thm}\label{matix AAAA}
				For $\textbf{x}=(x_1,\ldots,x_7)\in \mathbb C^7$ the following are equivalent.
				\begin{enumerate}
					\item $\textbf{x}=(x_1,\ldots,x_7)\in \Gamma_{E(3;3;1,1,1)}$;
					
					\item There exists $3\times 3$ matrix $A\in \mathcal M_{3\times 3}(\mathbb C)$  such that  \small{$$x_1=a_{11}, x_2=a_{22}, x_3=\det \left(\begin{smallmatrix} a_{11} & a_{12}\\
					a_{21} & a_{22}
				\end{smallmatrix}\right), x_4=a_{33}, x_5=\det \left(\begin{smallmatrix}
					a_{11} & a_{13}\\
					a_{31} & a_{33}
				\end{smallmatrix}\right), x_6=\det  \left(\begin{smallmatrix}
					a_{22} & a_{23}\\
					a_{32} & a_{33}\end{smallmatrix}\right) ~{\rm{and}}~x_7=\det A,$$}  \begin{equation}\label{tilgamma1q}
					\overline{\tilde{\gamma}_1(z_2,z_3)}(1-|z_2|^2)\tilde{\gamma}_1(z_2,z_3)+\overline{\tilde{\gamma}_2(z_2,z_3)}(1-|z_3|^2)\tilde{\gamma}_2(z_2,z_3)+{\eta(z_2,z_3)}^*(I_{3}-A^*A)\eta(z_2,z_3)\geq 0, 
				\end{equation}
				and $\det\left(I_{2}-\left(\begin{smallmatrix}a_{22}&a_{23}\\a_{32} & a_{33}
				\end{smallmatrix}\right)\left(\begin{smallmatrix}z_2 & 0 \\0 &z_{3}
				\end{smallmatrix}\right)\right)\neq 0$ for all $z_2,z_3\in \mathbb D$.
				\end{enumerate}
				
			\end{thm}

The following theorem gives the characterization of $\Gamma_{E(3;3;1,1,1)}$.
							
							\begin{thm}\label{maingamma}
								For $\textbf{x}=(x_1,\ldots,x_7)\in \mathbb C^7$ the following are equivalent.
								\begin{enumerate}
									\item[1.] $\textbf{x}\in \Gamma_{E(3;3;1,1,1)}$;
									
									\item[2.] $R_{\bf{x}}^{(3;3;1,1,1)}(\textbf{z})= 1-x_1z_1-x_2z_2+x_3z_1z_2-x_4z_3+x_5z_1z_3+x_6z_2z_3-x_7z_1z_2z_3\neq 0$, for all, $z_1,z_2,z_3\in \mathbb{D}$;
									
									\item[$3.$] $\textbf{x}_{J^{(1)}}^{\prime}\in \Gamma_{E(2;2;1,1)}$ and 
									$\|\Psi^{(1)}(\cdot,\textbf{x})\|_{H^{\infty}(\mathbb{D}^{2})}\leq 1$ and if   $x_{7}=x_{6}x_{1},\,\,x_{3}=x_{2}x_{1},\,\,x_{5}=x_{4}x_{1}$ then $|x_1|\leq1;$

									\item[$3^{\prime}.$]
									$\textbf{x}_{J^{(2)}}^{\prime}\in \Gamma_{E(2;2;1,1)}$ and 
									$\|\Psi^{(2)}(\cdot,\textbf{x})\|_{H^{\infty}(\mathbb{D}^{2})}\leq 1$ and if  $x_{7}=x_{5}x_{2},\,\,x_{3}=x_{2}x_{1},\,\,x_{6}=x_{4}x_{2}$ then $|x_2|\leq1;$
									
									\item[$3^{\prime\prime}.$]
									$\textbf{x}_{J^{(3)}}^{\prime}\in \Gamma_{E(2;2;1,1)}$ and 
									$\|\Psi^{(3)}(\cdot,\textbf{x})\|_{H^{\infty}(\mathbb{D}^{2})}\leq 1$ and if  $x_{7}=x_{3}x_{4},\,\,x_{6}=x_{2}x_{4},\,\,x_{5}=x_{4}x_{1}$ then $|x_4|\leq 1;$

									\item[$4.$] $\left(\frac{x_1-z_3x_5}{1-x_4z_3},\frac{x_2-z_3x_6}{1-x_4z_3},\frac{x_3-z_3x_7}{1-x_4z_3}\right)\in \Gamma_{E(2;2;1,1)}$  for all $z_3\in \mathbb D;$

									\item[$4^{\prime}.$] $\left(\frac{x_2-z_1x_3}{1-x_1z_1},\frac{x_4-z_1x_5}{1-x_1z_1},\frac{x_6-z_1x_7}{1-x_1z_1}\right)\in \Gamma_{E(2;2;1,1)}$  for all $z_1\in \mathbb D$;
									
									\item[$4^{\prime\prime}.$] $\left(\frac{x_4-z_2x_6}{1-x_2z_2},\frac{x_1-z_2x_3}{1-x_2z_2},\frac{x_5-z_2x_7}{1-x_2z_2}\right)\in \Gamma_{E(2;2;1,1)}$  for all $z_2\in \mathbb D$;

									\item[$5.$]
									There exists a $2\times 2$ symmetric matrix $A(z_3)$ with $\|A(z_3)\|\leq1$ 
									such that $$A_{11}(z_3)=\frac{x_1-z_3x_5}{1-x_4z_3},A_{22}(z_3)=\frac{x_2-z_3x_6}{1-x_4z_3}~{\rm{and }}~\det(A(z_3))=\frac{x_3-z_3x_7}{1-x_4z_3}~{\rm{for ~all }}~z_3\in \mathbb D;$$

									\item[$5^{\prime}.$]
									There exists a $2\times 2$ symmetric matrix $B(z_1)$ with $\|B(z_1)\|\leq1$  such that $$B_{11}(z_1)=\frac{x_2-z_1x_3}{1-x_1z_1},B_{22}(z_1)=\frac{x_4-z_1x_5}{1-x_1z_1}~{\rm{and }}~\det(B(z_1))=\frac{x_6-z_1x_7}{1-x_1z_1}~{\rm{for ~all }}~z_1\in \mathbb D;$$
									
									\item[$5^{\prime\prime}.$]
									There exists a $2\times 2$ symmetric matrix $C(z_2)$ with $\|C(z_2)\|\leq1$  such that $$C_{11}(z_2)=\frac{x_4-z_2x_6}{1-x_2z_2},C_{22}(z_2)=\frac{x_1-z_2x_3}{1-x_2z_2}~{\rm{and}}~\det(C(z_2))=\frac{x_5-z_2x_7}{1-x_2z_2}~{\rm{for ~all }}~z_2\in \mathbb D;$$

									\item[$6$.]  There exists a $3\times 3$ matrix $A\in \mathcal M_{3\times 3}(\mathbb C)$ such that \small{$x_{1}=a_{11},x_{2}=a_{22},x_{3}=a_{11}a_{22}-a_{12}a_{21},$}\\$x_{4}=a_{33},x_{5}=a_{11}a_{33}-a_{13}a_{31}, x_{6}=a_{33}a_{22}-a_{23}a_{32}~{\rm{and}}~x_{7}=\operatorname{det}A$, $$\sup_{(z_2,z_3)\in\mathbb D^2} |\mathcal G_{\mathcal F_{A}(z_3)}(z_2)|=\sup_{(z_2,z_3)\in\mathbb D^2 }|\mathcal G_{A}\left(\left(\begin{smallmatrix}z_2 & 0 \\0 &z_{3}
				\end{smallmatrix}\right)\right)|\leq 1,$$ and $\det\left(I_{2}-\left(\begin{smallmatrix}a_{22}&a_{23}\\a_{32} & a_{33}
				\end{smallmatrix}\right)\left(\begin{smallmatrix}z_2 & 0 \\0 &z_{3}
				\end{smallmatrix}\right)\right)\neq 0$ for all $(z_2,z_3)\in\mathbb D^2$.

				\item[$6^{\prime}$.]  There exists a $3\times 3$ matrix $\tilde{A}\in \mathcal M_{3\times 3}(\mathbb C)$ such that \small{$x_{1}=\tilde{a}_{33},x_{2}=\tilde{a}_{11},x_{3}=\tilde{a}_{11}\tilde{a}_{33}-\tilde{a}_{13}\tilde{a}_{31},$} \\$x_{4}=\tilde{a}_{22},x_{5}=\tilde{a}_{22}\tilde{a}_{33}-\tilde{a}_{23}\tilde{a}_{32}, x_{6}=\tilde{a}_{11}\tilde{a}_{22}-\tilde{a}_{12}\tilde{a}_{21}~{\rm{and}}~{x}_{7}=\operatorname{det} \tilde{A},$  $$\sup_{(z_1,z_2)\in\mathbb D^2} |\mathcal G_{\mathcal F_{\tilde{A}}(z_2)}(z_1)|=\sup_{(z_1,z_2)\in\mathbb D^2} |\mathcal G_{\tilde{A}}\left(\left(\begin{smallmatrix}z_1 & 0 \\0 &z_{2}
				\end{smallmatrix}\right)\right)|\leq 1,$$ and  $\det\left(I_{2}-\left(\begin{smallmatrix}\tilde{a}_{11}&\tilde{a}_{12}\\\tilde{a}_{21} & \tilde{a}_{22}
				\end{smallmatrix}\right)\left(\begin{smallmatrix}z_1 & 0 \\0 &z_{2}
				\end{smallmatrix}\right)\right)\neq 0$ for all $(z_1,z_2)\in\mathbb D^2$.
				
				\item[$6^{\prime\prime}$.] There exists a $3\times 3$ matrix $\tilde{B}\in \mathcal M_{3\times 3}(\mathbb C)$  such that \small{$ x_1=\tilde{b}_{22},x_2=\tilde{b}_{33},x_3=\tilde{b}_{22}\tilde{b}_{33}-\tilde{b}_{23}\tilde{b}_{32},$} \\$ x_4=\tilde{b}_{11},x_{5}=\tilde{b}_{11}\tilde{b}_{22}-\tilde{b}_{12}\tilde{b}_{21}, x_{6}=\tilde{b}_{11}\tilde{b}_{33}-\tilde{b}_{13}\tilde{b}_{31}~{\rm{and}}~x_{7}=\operatorname{det} \tilde{B}$, $$\sup_{(z_1,z_3)\in\mathbb D^2} |\mathcal G_{\mathcal F_{\tilde{B}}(z_1)}(z_3)|=\sup_{(z_1,z_3)\in \mathbb D^2} |\mathcal G_{\tilde{B}}\left(\left(\begin{smallmatrix}z_1 & 0 \\0 &z_{3}
				\end{smallmatrix}\right)\right)|\leq 1,$$  and $\det\left(I_{2}-\left(\begin{smallmatrix}\tilde{b}_{11}&\tilde{b}_{13}\\\tilde{b}_{31} & \tilde{b}_{33}
				\end{smallmatrix}\right)\left(\begin{smallmatrix}z_1 & 0 \\0 &z_{3}
				\end{smallmatrix}\right)\right)\neq 0$ for all $(z_1,z_3)\in\mathbb D^2$.
									
								\end{enumerate}
							\end{thm}

								We get the following corollary as a result of the aforementioned theorem.
								\begin{cor}\label{coromain}
									For $\textbf{x}=(x_1,\ldots,x_7)\in \mathbb C^7$ the following are equivalent.
									\begin{enumerate}
										\item $\textbf{x}=(x_1,\ldots,x_7)\in \Gamma_{E(3;3;1,1,1)}$;		
										
										\item $\sup_{z_1\in \mathbb{D}}\{|\tilde{x}_1(z_1)-\overline{\tilde{x}_2(z_1) }\tilde{x}_3(z_1)|+\left|\tilde{x}_1(z_1)\tilde{x}_2(z_1)-\tilde{x}_3(z_1)\right|\} \leq 1-\inf_{z_1\in \mathbb{D}}\left|\tilde{x}_2(z_1)\right|^2$;
										\item $\sup_{z_1\in \mathbb{D}}\{|\tilde{x}_2(z_1)-\overline{\tilde{x}_1(z_1) }\tilde{x}_3(z_1)|+\left|\tilde{x}_1(z_1)\tilde{x}_2(z_1)-\tilde{x}_3(z_1)\right|\} \leq 1-\inf_{z_1\in \mathbb{D}}\left|\tilde{x}_1(z_1)\right|^2$;		
										\item $\sup_{z_1\in \mathbb{D}}\left\{\left|\tilde{x}_1(z_1)\right|^2-\left|\tilde{x}_2(z_1)\right|^2+\left|\tilde{x}_3(z_1)\right|^2+2\left|\tilde{x}_2(z_1)-\overline{\tilde{x}_1(z_1) }\tilde{x}_3(z_1)\right|\right\} \leq 1$ and if we assume  $\tilde{x}_1(z_1)\tilde{x}_2(z_1)=\tilde{x}_3(z_1)$ for all $z_1\in \mathbb{D}$, then $\sup_{z_1\in \mathbb{D}}|\tilde{x}_2(z_1)|\leq1$;
										\item $\sup_{z_1\in \mathbb{D}}\left\{-\left|\tilde{x}_1(z_1)\right|^2+\left|\tilde{x}_2(z_1)\right|^2+\left|\tilde{x}_3(z_1)\right|^2+2\left|\tilde{x}_1(z_1)-\overline{\tilde{x}_2 (z_1)}\tilde{x}_3(z_1)\right|\right\} \leq 1$ and if we assume $\tilde{x}_1(z_1)\tilde{x}_2(z_1)=\tilde{x}_3(z_1)$  for all $z_1\in \mathbb{D}$, then $\sup_{z_1\in \mathbb{D}}|\tilde{x}_1(z_1)|\leq1$;
										\item $\sup_{z_1\in \mathbb{D}}\{\left|\tilde{x}_1(z_1)\right|^2+\left|\tilde{x}_2(z_1)\right|^2-\left|\tilde{x}_3(z_1)\right|^2+2\left|\tilde{x}_1(z_1) \tilde{x}_2(z_1)-\tilde{x}_3(z_1)\right|\} \leq 1$;
										\item $\sup_{z_1\in \mathbb{D}}\left\{\left|\tilde{x}_1(z_1)-\overline{\tilde{x}_2(z_1)} \tilde{x}_3(z_1)\right|+\left|\tilde{x}_2(z_1)-\overline{\tilde{x}_1(z_1)} \tilde{x}_3(z_1)\right|\right\} \leq 1-\inf_{z_1\in \mathbb{D}}\left|\tilde{x}_3(z_1)\right|^2$;
										\item $\sup_{z_2\in \mathbb{D}}\{|\tilde{y}_1(z_2)-\overline{\tilde{y}_2(z_2)} \tilde{y}_3(z_2)|+\left|\tilde{y}_1(z_2) \tilde{y}_2(z_2)-\tilde{y}_3(z_2)\right|\} \leq 1-\inf_{z_2\in \mathbb{D}}\left|\tilde{y}_2(z_2)\right|^2$;
										\item $\sup_{z_2\in \mathbb{D}}\left\{\left|\tilde{y}_2(z_2)-\overline{\tilde{y}_1(z_2)} \tilde{y}_3(z_2)\right|+\left|\tilde{y}_1(z_2) \tilde{y}_2(z_2)-\tilde{y}_3(z_2)\right|\right\} \leq 1-\inf_{z_2\in \mathbb{D}}\left|\tilde{y}_1(z_2)\right|^2$;
										\item $\sup_{z_2\in \mathbb{D}}\left\{\left|\tilde{y}_1(z_2)\right|^2-\left|\tilde{y}_2(z_2)\right|^2+\left|\tilde{y}_3(z_2)\right|^2+2\left|\tilde{y}_2(z_2)-\overline{\tilde{y}_1(z_2)} \tilde{y}_3(z_2)\right|\right\} \leq 1$, and if we assume $\tilde{y}_1(z_2)\tilde{y}_2(z_2)=\tilde{y}_3(z_2)$ for all $z_2\in \mathbb{D}$, then $\sup_{z_2\in \mathbb{D}}|\tilde{y}_2(z_2)|\leq1$;
										\item $\sup_{z_2\in \mathbb{D}}\left\{-\left|\tilde{y}_1(z_2)\right|^2+\left|\tilde{y}_2(z_2)\right|^2+\left|\tilde{y}_3(z_2)\right|^2+2\left|\tilde{y}_1(z_2)-\overline{\tilde{y}_2(z_2)} \tilde{y}_3(z_2)\right|\right\} \leq 1$, and if we assume $\tilde{y}_1(z_2)\tilde{y}_2(z_2)=\tilde{y}_3(z_2)$ for all $z_2\in \mathbb{D}$, then $\sup_{z_2\in \mathbb{D}}|\tilde{y}_1(z_2)|\leq1$;
										\item $\sup_{z_2\in \mathbb{D}}\{\left|\tilde{y}_1(z_2)\right|^2+\left|\tilde{y}_2(z_2)\right|^2-\left|\tilde{y}_3(z_2)\right|^2+2\left|\tilde{y}_1(z_2) \tilde{y}_2(z_2)-\tilde{y}_3(z_2)\right|\} \leq 1$;
										\item $\sup_{z_2\in \mathbb{D}}\left\{\left|\tilde{y}_1(z_2)-\overline{\tilde{y}_2(z_2)} \tilde{y}_3(z_2)\right|+\left|\tilde{y}_2(z_2)-\overline{\tilde{y}_1(z_2)} \tilde{y}_3(z_2)\right|\right\} \leq 1-\inf_{z_2\in \mathbb{D}}\left|\tilde{y}_3(z_2)\right|^2$;
										\item $\sup_{z_3\in \mathbb{D}}\{|\tilde{z}_1(z_3)-\overline{\tilde{z}_2(z_3)} \tilde{z}_3(z_3)|+\left|\tilde{z}_1(z_3) \tilde{z}_2(z_3)-\tilde{z}_3(z_3)\right|\} \leq 1-\inf_{z_3\in \mathbb{D}}\left|\tilde{z}_2(z_3)\right|^2$;
										\item $\sup_{z_3\in \mathbb{D}}\left\{\left|\tilde{z}_2(z_3)-\overline{\tilde{z}_1(z_3)} \tilde{z}_3(z_3)\right|+\left|\tilde{z}_1(z_3) \tilde{z}_2(z_3)-\tilde{z}_3(z_3)\right|\right\} \leq 1-\inf_{z_3\in \mathbb{D}}\left|\tilde{z}_1(z_3)\right|^2$;
										\item $\sup_{z_3\in \mathbb{D}}\left\{\left|\tilde{z}_1(z_3)\right|^2-\left|\tilde{z}_2(z_3)\right|^2+\left|\tilde{z}_3(z_3)\right|^2+2\left|\tilde{z}_2(z_3)-\overline{\tilde{z}_1(z_3)} \tilde{z}_3(z_3)\right|\right\} \leq 1$, and if we assume $\tilde{z}_1(z_3)\tilde{z}_2(z_3)=\tilde{z}_3(z_3)$ for all $z_3\in \mathbb{D}$, then $\sup_{z_3\in \mathbb{D}}|\tilde{z}_2(z_3)|\leq1$;
										\item $\sup_{z_3\in \mathbb{D}}\left\{-\left|\tilde{z}_1(z_3)\right|^2+\left|\tilde{z}_2(z_3)\right|^2+\left|\tilde{z}_3(z_3)\right|^2+2\left|\tilde{z}_1(z_3)-\overline{\tilde{z}_2(z_3)} \tilde{z}_3(z_3)\right|\right\} \leq 1$, and if we assume $\tilde{z}_1(z_3)\tilde{z}_2(z_3)=\tilde{z}_3(z_3)$ for all $z_3\in \mathbb{D}$, then $\sup_{z_3\in \mathbb{D}}|\tilde{z}_1(z_3)|\leq1$;
										\item $\sup_{z_3\in \mathbb{D}}\{\left|\tilde{z}_1(z_3)\right|^2+\left|\tilde{z}_2(z_3)\right|^2-\left|\tilde{z}_3(z_3)\right|^2+2\left|\tilde{z}_1(z_3) \tilde{z}_2(z_3)-\tilde{z}_3(z_3)\right|\} \leq 1$;
										\item $\sup_{z_3\in \mathbb{D}}\left\{\left|\tilde{z}_1(z_3)-\overline{\tilde{z}_2(z_3)} \tilde{z}_3(z_3)\right|+\left|\tilde{z}_2(z_3)-\overline{\tilde{z}_1(z_3)} \tilde{z}_3(z_3)\right|\right\} \leq 1-\inf_{z_3\in \mathbb{D}}\left|\tilde{z}_3(z_3)\right|^2$.	
									\end{enumerate}
								\end{cor}

\begin{thm}\label{matix barAAAA}

Let $A\in \mathcal M_{3\times 3}(\mathbb C)$ with $\|A\|\leq1.$  Set   \small{$$x_1=a_{11}, x_2=a_{22}, x_3=\det \left(\begin{smallmatrix} a_{11} & a_{12}\\
					a_{21} & a_{22}
				\end{smallmatrix}\right), x_4=a_{33}, x_5=\det \left(\begin{smallmatrix}
					a_{11} & a_{13}\\
					a_{31} & a_{33}
				\end{smallmatrix}\right), x_6=\det  \left(\begin{smallmatrix}
					a_{22} & a_{23}\\
					a_{32} & a_{33}\end{smallmatrix}\right) ~{\rm{and}}~x_7=\det A.$$} Then $\textbf{x}=(x_1,\ldots,x_7)\in \Gamma_{E(3;3;1,1,1)}.$				
			\end{thm}

\begin{proof}
Observe that $\|\left(\begin{smallmatrix}a_{22} & a_{23} \\a_{32} &a_{33}
								\end{smallmatrix}\right)\left(\begin{smallmatrix}z_2 & 0 \\0 &z_{3}
								\end{smallmatrix}\right)\|<1$ for all $z_2,z_3\in \mathbb D.$ This means that  $\det\left(I_2-\left(\begin{smallmatrix}a_{22} & a_{23} \\a_{32} &a_{33}
								\end{smallmatrix}\right)\left(\begin{smallmatrix}z_2 & 0 \\0 &z_{3}
								\end{smallmatrix}\right)\right)\neq 0$ for all $z_2,z_3\in \mathbb D.$ As a result, we get $\textbf{x}_{J^{(1)}}^{\prime}\in \Gamma_{E(2;2;1,1)}$ by the characterization of tetrablock [Theorem $2.3$, \cite{Abouhajar}].
								
Note that for all $z_2,z_3\in \mathbb C$ with $\det\left(I_{2}-\left(\begin{smallmatrix}a_{22}&a_{23}\\a_{32} & a_{33}
				\end{smallmatrix}\right)\left(\begin{smallmatrix}z_2 & 0 \\0 &z_{3}
				\end{smallmatrix}\right)\right)\neq 0,$ we have
				\small{\begin{align}\label{mathcalG11}
						\mathcal G_{A}\left(\left(\begin{smallmatrix}z_2 & 0 \\0 &z_{3}
						\end{smallmatrix}\right)\right)\nonumber &=a_{11}+\left(\begin{smallmatrix}a_{12}&a_{13}
						\end{smallmatrix}\right)\left(\begin{smallmatrix}z_2 & 0 \\0 &z_{3}
						\end{smallmatrix}\right)\left(I_{2}-\left(\begin{smallmatrix}a_{22}&a_{23}\\a_{32} & a_{33}
						\end{smallmatrix}\right)\left(\begin{smallmatrix}z_2 & 0 \\0 &z_{3}
						\end{smallmatrix}\right)\right)^{-1}\left(\begin{smallmatrix}a_{21}\\a_{31}
						\end{smallmatrix}\right)\\\nonumber&=\frac{a_{11}-z_2(a_{11}a_{22}-a_{12}a_{21})-z_3(a_{33}a_{11}-a_{13}a_{31})+z_2z_3\det(A)}{1-z_2a_{22}-z_3a_{33}+z_2z_3(a_{22}a_{33}-a_{23}a_{32})}\\\nonumber&=\frac{x_1-z_2x_3-z_3x_5+z_2z_3x_7}{1-z_2x_2-z_3x_4+z_2z_3x_6}\\&=\Psi^{(1)}(\textbf{z}_{J^{(1)}},\textbf{x}).
				\end{align}} 
From Proposition \ref{matrix AAA}, we get  $|\mathcal G_{A}\left(\left(\begin{smallmatrix}z_2 & 0 \\0 &z_{3}
				\end{smallmatrix}\right)\right)|\leq 1$ for all $z_2,z_3\in {\mathbb D}$.	 Hence, from Theorem \ref{phii1}, we conclude that $\mathbf{x}\in \Gamma_{E(3;3;1,1,1)}$	 This completes the proof.

\end{proof}

\subsection{Relation between the domains $G_{E(3;2;1,2)}$ and $G_{E(3;3;1,1,1)}$ and $\Gamma_{E(3;2;1,2)}$ and $\Gamma_{E(3;3;1,1,1)}$}
			
			For $n\geq 2$, G. Bharali introduced the domain $\mathbb E_n$ \cite{bha}.
			Throughout the work, we use the notation $G_{E(3;2;1,2)}$ rather than the $\mathbb E_3$, where $E(3;2;1,2)$ is the vector subspace of $\mathcal M_{3\times 3}(\mathbb C)$ of the following form
			$$E(3;2;1,2)=\left\{\left(\begin{array}{cc}z_1& 0\\0 &z_2I_2
			\end{array}\right)\in \mathbb{C}^{3\times 3}:z_1,z_2\in \mathbb{C} \right\}.$$ In this subsection, we establish the relation between the domain $G_{E(3;2;1,2)}$ and $G_{E(3;3;1,1,1)}$.  We recall the definition of the domain $G_{E(3;2;1,2)}$ from \cite{bha}. We denote  the closure of $G_{E(3;2;1,2)}$ as $\Gamma_{E(3;2;1,2)}$. Let $\tilde{\textbf{x}}=(x_1,x_2,x_3,y_1,y_2)\in \mathbb{C}^{5}$ and for $\tilde{\textbf{z}}=(z_1,z_2)\in \bar{\mathbb{D}}^2,$ consider the polynomial $$R_3(\tilde{\textbf{z}};\tilde{\textbf{x}})=\left(y_2z_2^2-y_1z_2+1\right)-z_1\left(x_3z_2^2-x_2z_2+x_1\right)=Q_3(z_2;y_1,y_2)-z_1P_3(z_2;x_1,x_2,x_3).$$ 
We define the domains $\mathcal  A_{E(3;2;1,2)}^{(1)}$ and $\mathcal  B_{E(3;2;1,2)}^{(1)}$ as follows
			$$\mathcal  A_{E(3;2;1,2)}^{(1)}=\left \{\tilde{\textbf{x}}\in \mathbb{C}^{5}:R_3(\tilde{\textbf{z}};\tilde{\textbf{x}})\neq0, \,\, \text{for all} \,\, \tilde{\textbf{z}}\in \bar{\mathbb{D}}^2 \right \}$$ and $$\mathcal  B_{E(3;2;1,2)}^{(1)}=\left \{\tilde{\textbf{x}}\in \mathbb{C}^{5}:R_3(\tilde{\textbf{z}};\tilde{\textbf{x}})\neq0, \,\, \text{for all} \,\, \tilde{\textbf{z}}\in \mathbb{D}^2 \right \}.$$ Let us consider the polynomial map $\tilde{\pi}_{{E(3;2;1,2)}}:\mathcal M_{3\times 3}(\mathbb C)\to \mathbb C^5$ \cite{bha} which is defined by 
			\begin{equation}\label{pi123}\tilde{\pi}_{{E(3;2;1,2)}}(A):=\left( a_{11},\det \left(\begin{smallmatrix} a_{11} & a_{12}\\
					a_{21} & a_{22}
				\end{smallmatrix}\right)+\det \left(\begin{smallmatrix}
					a_{11} & a_{13}\\
					a_{31} & a_{33}
				\end{smallmatrix}\right),\operatorname{det}A, a_{22}+a_{33}, \det  \left(\begin{smallmatrix}
					a_{22} & a_{23}\\
					a_{32} & a_{33}\end{smallmatrix}\right)\right).
					\end{equation}			
			
		Let $$G_{E(3;2;1,2)}:=\{\tilde{\pi}_{{E(3;2;1,2)}}(A):\mu_{E(3;2;1,2)}(A)<1\}.$$
The domain $G_{E(3;2;1,2)}$ is known as $\mu_{1,3}-$\textit{quotients} \cite{bha}. For  $\tilde{\textbf{x}}=(x_1,x_2,x_3,y_1,y_2)\in \mathbb{C}^{5},$ we consider the rational function $$\Psi_{3}(z_2,\tilde{\textbf{x}})=\frac{P_3(z_2;x_1,x_2,x_3)}{Q_3(z_2;y_1,y_2)}=\frac{x_3z_2^2-x_2z_2+x_1}{y_2z_2^2-y_1z_2+1}~~~{\rm{ for ~all }}~z_2\in \mathbb C ~{\rm{with}}~y_2z_2^2-y_1z_2+1\neq0.$$ 
			The following proposition gives the characterization of $G_{E(3;2;1,2)},$ whose proof is discussed in [Theorem $3.5$,~\cite{bha}].
			\begin{prop}\label{bhch}
				For $\tilde{\textbf{x}}\in \mathbb{C}^{5}$ the following are equivalent
				\begin{enumerate}
					\item $\tilde{\textbf{x}}\in G_{E(3;2;1,2)}$.
					\item $\tilde{\textbf{x}} \in \mathcal  A_{E(3;2;1,2)}^{(1)}.$
					\item $y_2z_2^2-y_1z_2+1\neq0$ for all $z_2\in \bar{\mathbb{D}}$ and $\|\Psi_{3}(\cdot,\tilde{\textbf{x}})\|_{H^{\infty}(\bar{\mathbb{D}})}:=\sup_{z_2\in \bar{\mathbb{D}}}|\Psi_{3}(z_2,\tilde{\textbf{x}})|<1$.  Moreover, if  $x_3=x_1y_2,x_2=x_1y_1$ then $|x_1|<1$.
					\item For every  $z_1\in \bar{\mathbb{D}}$, we have
					$\left(\frac{y_1-z_1x_2}{1-z_1x_1},\frac{y_2-z_1x_3}{1-z_1x_1}\right)\in G_{E(2;1;2)}.$ 
				\end{enumerate}
							\end{prop} 
			We  characterize the domain $G_{E(3;2;1,2)}$ by using the realization formula. For $z_2=z_3$, from \eqref{ga} and \eqref{et}, we have 
			\begin{align}\label{gaa}\tilde{\gamma}(z_2)\nonumber&=\left(I_{2}-\left(\begin{smallmatrix}a_{22} & a_{23} \\a_{32} &a_{33}
				\end{smallmatrix}\right)\left(\begin{smallmatrix}z_2 & 0 \\0 &z_{2}
				\end{smallmatrix}\right)\right)^{-1}\left(\begin{smallmatrix}a_{21} \\ a_{31}
				\end{smallmatrix}\right)\\&=\left(\begin{smallmatrix}\tilde{\gamma}_1(z_2) \\ \tilde{\gamma}_2(z_2)
				\end{smallmatrix}\right),
			\end{align} 
			and
			\begin{align}\label{ett}\tilde{\eta}(z_2)\nonumber&=\left(\begin{smallmatrix}1\\\left(\begin{smallmatrix}z_2 & 0 \\0 &z_{2}
					\end{smallmatrix}\right)\tilde{\gamma}\left(\left(\begin{smallmatrix}z_2 & 0 \\0 &z_{2}
					\end{smallmatrix}\right)\right)
				\end{smallmatrix}\right)\\&=\left(\begin{smallmatrix}1\\z_2\tilde{\gamma}_1(z_2) \\ z_2\tilde{\gamma}_2(z_2)
				\end{smallmatrix}\right)
			\end{align}
			for all $z_2\in\mathbb C$ such that $\det(C(z_2))\neq 0,$ where \small{$$C(z_2)=\left(I_{2}-\left(\begin{smallmatrix}a_{22} & a_{23} \\a_{32} &a_{33}
				\end{smallmatrix}\right)\left(\begin{smallmatrix}z_2 & 0 \\0 &z_{2}
				\end{smallmatrix}\right)\right),\tilde{\gamma}_1(z_2)=\frac{(1-a_{33}z_2)a_{21}+z_2a_{23}a_{31}} {\det(C(z_2))},$$ and $$\tilde{\gamma}_2(z_2)=\frac{(1-a_{22}z_2)a_{31}+z_2a_{32}a_{21}} {\det(C(z_2))}.$$}
			
			\begin{thm}\label{matix AE}
				Let $\tilde{\textbf{x}}=(x_1,x_2,x_3,y_1,y_2)\in \mathbb C^5.$ Then the following are equivalent:
				\begin{enumerate}
					\item $\tilde{\textbf{x}}=(x_1,x_2,x_3,y_1,y_2)\in G_{E(3;2;1,2)}$.
					
					\item There exists $3\times 3$ matrix $A\in \mathcal M_{3\times 3}(\mathbb C)$  such that $$x_1=a_{11},x_2=\det \left(\begin{smallmatrix} a_{11} & a_{12}\\
					a_{21} & a_{22}
				\end{smallmatrix}\right)+\det \left(\begin{smallmatrix}
					a_{11} & a_{13}\\
					a_{31} & a_{33}
				\end{smallmatrix}\right),x_3=\operatorname{det}A, y_1=a_{22}+a_{33}, y_2=\det  \left(\begin{smallmatrix}
					a_{22} & a_{23}\\
					a_{32} & a_{33}\end{smallmatrix}\right),$$  \begin{equation}\label{tilgammaa12}
					(1-|z_2|^2)(\overline{\tilde{\gamma}_1(z_2)}\tilde{\gamma}_1(z_2)+\overline{\tilde{\gamma}_2(z_2)}\tilde{\gamma}_2(z_2))+{\eta(z_2)}^*(I_{3}-A^*A)\eta(z_2)>0, 
				\end{equation}
				and  $\det\left(I_{2}-\left(\begin{smallmatrix}a_{22}&a_{23}\\a_{32} & a_{33}
				\end{smallmatrix}\right)\left(\begin{smallmatrix}z_2 & 0 \\0 &z_{2}
				\end{smallmatrix}\right)\right)\neq 0$ for all $z_2\in \bar{\mathbb D}$.
				\end{enumerate}
			
			\end{thm}
			\begin{proof}
				Suppose $\tilde{\textbf{x}}=(x_1,x_2,x_3,y_1,y_2)\in G_{E(3;2;1,2)}$. Then it follows from [Lemma 2.3, \cite{bha}] that there exists $3\times 3$ matrix $A\in \mathcal M_{3\times 3}(\mathbb C)$  such that  \begin{equation}\label{mu13qu}x_1=a_{11},x_2=\det \left(\begin{smallmatrix} a_{11} & a_{12}\\
					a_{21} & a_{22}
				\end{smallmatrix}\right)+\det \left(\begin{smallmatrix}
					a_{11} & a_{13}\\
					a_{31} & a_{33}
				\end{smallmatrix}\right),x_3=\operatorname{det}A, y_1=a_{22}+a_{33}, y_2=\det  \left(\begin{smallmatrix}
					a_{22} & a_{23}\\
					a_{32} & a_{33}\end{smallmatrix}\right).\end{equation} Notice that for all $z_2\in \mathbb C$ with $\det\left(I_{2}-\left(\begin{smallmatrix}a_{22}&a_{23}\\a_{32} & a_{33}
				\end{smallmatrix}\right)\left(\begin{smallmatrix}z_2 & 0 \\0 &z_{2}
				\end{smallmatrix}\right)\right)\neq 0,$ we have
				\small{\begin{align}\label{mathcalGE}
						\mathcal G_{A}\left(\left(\begin{smallmatrix}z_2 & 0 \\0 &z_{2}
						\end{smallmatrix}\right)\right)\nonumber &=a_{11}+\left(\begin{smallmatrix}a_{12}&a_{13}
						\end{smallmatrix}\right)\left(\begin{smallmatrix}z_2 & 0 \\0 &z_{2}
						\end{smallmatrix}\right)\left(I_{2}-\left(\begin{smallmatrix}a_{22}&a_{23}\\a_{32} & a_{33}
						\end{smallmatrix}\right)\left(\begin{smallmatrix}z_2 & 0 \\0 &z_{2}
						\end{smallmatrix}\right)\right)^{-1}\left(\begin{smallmatrix}a_{21}\\a_{31}
						\end{smallmatrix}\right)\\\nonumber&=\frac{a_{11}-z_2(a_{11}a_{22}-a_{12}a_{21})-z_2(a_{33}a_{11}-a_{13}a_{31})+z_2^2\det(A)}{1-z_2a_{22}-z_2a_{33}+z_2^2(a_{22}a_{33}-a_{23}a_{32})}\\\nonumber&=\frac{x_1-z_2x_2+z_2^2x_3}{1-z_2y_1+z_2^2y_2}\\&=\Psi_3(z_2,\tilde{\textbf{x}}).
				\end{align}}
				From Proposition \ref{bhch}, we observe that $\tilde{\textbf{x}}\in G_{E(3;2;1,2)}$ if and only if  $$ y_2z_2^2-y_1z_2+1\neq 0 ~{\rm{for ~all }}~z_2\in \bar{\mathbb{D}}~{\rm{ and }}~\|\Psi_3(\cdot,\tilde{\textbf{x}})\|_{H^{\infty}(\bar{\mathbb{D}})}=\|\Psi_3(\cdot,\tilde{\textbf{x}})\|_{H^{\infty}(\mathbb{T})}< 1.$$ Thus, from \eqref{mathcalGE}, it yields that 
				$$\sup_{z_2\in\bar{\mathbb D}} |\mathcal G_{A}\left(\left(\begin{smallmatrix}z_2 & 0 \\0 &z_{2}
				\end{smallmatrix}\right)\right)|<1~{\rm{and}}~\det\left(I_{2}-\left(\begin{smallmatrix}a_{22}&a_{23}\\a_{32} & a_{33}
				\end{smallmatrix}\right)\left(\begin{smallmatrix}z_2 & 0 \\0 &z_{2}
				\end{smallmatrix}\right)\right)\neq 0$$ for all $z_2\in\bar{\mathbb D}$ and hence we get
				\begin{equation}\label{GAAA}
					1-\overline{\mathcal G_{A}\left(\left(\begin{smallmatrix}z_2 & 0 \\0 &z_{2}
						\end{smallmatrix}\right)\right)}\mathcal G_{A}\left(\left(\begin{smallmatrix}z_2 & 0 \\0 &z_{2}
					\end{smallmatrix}\right)\right)>0 ~{\rm{and}}~\det\left(I_{2}-\left(\begin{smallmatrix}a_{22}&a_{23}\\a_{32} & a_{33}
				\end{smallmatrix}\right)\left(\begin{smallmatrix}z_2 & 0 \\0 &z_{2}
				\end{smallmatrix}\right)\right)\neq 0 ~~{\rm{for~all}}~~z_2\in\bar{\mathbb D}.
				\end{equation}
				From \eqref{AAAA norm} and \eqref{GAAA}, it implies that
				\begin{equation}\label{tilgammaa}
					(1-|z_2|^2)(\overline{\tilde{\gamma}_1(z_2)}\tilde{\gamma}_1(z_2)+\overline{\tilde{\gamma}_2(z_2)}\tilde{\gamma}_2(z_2))+{\eta(z_2)}^*(I_{3}-A^*A)\eta(z_2)>0, 
				\end{equation}
				and $\det\left(I_{2}-\left(\begin{smallmatrix}a_{22}&a_{23}\\a_{32} & a_{33}
				\end{smallmatrix}\right)\left(\begin{smallmatrix}z_2 & 0 \\0 &z_{2}
				\end{smallmatrix}\right)\right)\neq 0$ for all $z_2\in \bar{\mathbb D}.$ This completes the proof.
				\end{proof}
The proof of the following theorem is same as Theorem \ref{matix A}. Therefore, we skip the proof.		
\begin{thm}\label{matix ABC}
%
Let $A\in \mathcal M_{3\times 3}(\mathbb C)$ with $\|A\|<1.$  Set $$x_1=a_{11},x_2=\det \left(\begin{smallmatrix} a_{11} & a_{12}\\
					a_{21} & a_{22}
				\end{smallmatrix}\right)+\det \left(\begin{smallmatrix}
					a_{11} & a_{13}\\
					a_{31} & a_{33}
				\end{smallmatrix}\right),x_3=\operatorname{det}A, y_1=a_{22}+a_{33}, y_2=\det  \left(\begin{smallmatrix}
					a_{22} & a_{23}\\
					a_{32} & a_{33}\end{smallmatrix}\right).$$  Then $\textbf{x}=(x_1,x_2,x_3,y_1,y_2)\in G_{E(3;2;1,2)}.$				
			\end{thm}			
					
The following theorem tells the relation between $G_{E(3;2;1,2)}$ and $G_{E(3;3;1,1,1)}$.
			\begin{thm}\label{matix AE12}
				Suppose $\textbf{x}=(x_1,x_2,x_3,x_4,x_5,x_6,x_7) \in \mathbb{C}^7 .$ Then the following conditions are equivalent:
\begin{enumerate}
\item  $\textbf{x} \in G_{E(3;3;1,1,1)}$;
\item  $(x_1,x_3+\eta x_5,\eta x_7,x_2+\eta x_4,\eta x_6) \in G_{E(3;2;1,2)}$ for all $\eta \in \mathbb{T}$;
\item $(x_1,x_3+\eta x_5,\eta x_7,x_2+\eta x_4,\eta x_6) \in G_{E(3;2;1,2)}$ for all $\eta \in \bar{\mathbb{D}}.$
\end{enumerate}
			\end{thm}
			\begin{proof}
First we prove that $(1)$ implies $(2).$ By Theorem \ref{phi}, it follows that $\textbf{x}\in G_{E(3;3;1,1,1)}$ if and only if  $$(x_2,x_4,x_6)\in G_{E(2;2;1,1)},$$ and  $$\sup_{z_1,z_2\in \mathbb{D}}|\Psi^{(1)}(z_1,z_2,\textbf{x})|=\|\Psi^{(1)}(\cdot,\textbf{x})\|_{H^{\infty}(\bar{\mathbb{D}}^{2})}=\|\Psi^{(1)}(\cdot,\textbf{x})\|_{H^{\infty}(\mathbb{T}^{2})}=\sup_{z_1,z_2\in \mathbb{T}}|\Psi^{(1)}(z_1,z_2,\textbf{x})|< 1.$$ 	 Notice that 
				\begin{align}\label{phipsi}
					\sup_{|z_2|\leq 1, |\eta|=1}\Big|\Psi_3(z_2,(x_1,x_3+\eta x_5,\eta x_7,x_2+\eta x_4,\eta x_6)) \Big|\nonumber&=\sup_{|z_2|\leq 1, |\eta|=1}\Big|\frac{\eta x_7z_2^2-(x_3+\eta x_5)z_2+x_1}{\eta x_6z_2^2-(x_2+\eta x_4)z_2+1}\Big|\\\nonumber&=\sup_{|z_2|\leq 1, |\eta|=1}\Big|\frac{\eta z_2z_2x_7-x_3z_2-\eta z_2 x_5+x_1}{\eta z_2 z_2x_6-x_2z_2-\eta  z_2x_4+1}\Big|\\\nonumber&=\sup_{|z_1|\leq 1,|z_2|\leq 1}\Big|\frac{z_1z_2x_7-x_3z_2-z_1x_5+x_1}{z_1z_2x_6-x_2z_2-z_1x_4+1}\Big|\\&= \sup_{|z_1|\leq1, |z_2|\leq 1}|\Psi^{(1)}(z_1,z_2,\textbf{x})|
\nonumber\\&=\|\Psi^{(1)}(\cdot,\textbf{x})\|_{H^{\infty}(\mathbb{T}^2)}.
				\end{align}
				
				However,  from [Lemma $3.2$, \cite{tirt}], if $(x_2,x_4,x_6)\in G_{E(2;2;1,1)}$ if and only if $(x_2+\eta x_4, \eta x_6)\in G_{E(2;1;2)}$ for all $\eta \in \mathbb T$ and hence from \eqref{phipsi} and Proposition \ref{bhch}, it yields that  $$(x_1,x_3+\eta x_5,\eta x_7,x_2+\eta x_4,\eta x_6) \in G_{E(3;2;1,2)}~{\rm{ for~ all }}~\eta \in \mathbb{T}.$$
				
In order to show that $(2)$ implies $(1)$, let  $(x_1,x_3+\eta x_5,\eta x_7,x_2+\eta x_4,\eta x_6) \in G_{E(3;2;1,2)}$ for all $\eta \in \mathbb{T}$.  Then $(x_2+\eta x_4, \eta x_6)\in G_{E(2;1;2)}$ which implies $(x_2,x_4,x_6)\in G_{E(2;2;1,1)}$. Hence from \eqref{phipsi} and Theorem \ref{phi}, we conclude that $\textbf{x}\in G_{E(3;3;1,1,1)}$. 

Similarly, we can prove that $(1)$ and $(3)$  are equivalent using the same reasoning. 
This completes the proof.

			\end{proof}
Let $$\Gamma_{E(3;2;1,2)}:=\{\tilde{\pi}_{{E(3;2;1,2)}}(A):\mu_{E(3;2;1,2)}(A)\leq1\}.$$

The characterization of $\Gamma_{E(3;2;1,2)}$ is provided in the following proposition \cite{bha}.
			
			\begin{prop}\label{bhchh}
				For $\tilde{\textbf{x}}\in \mathbb{C}^{5}$ the following are equivalent
				\begin{enumerate}
					\item $\tilde{\textbf{x}}\in \Gamma_{E(3;2;1,2)}$.
					\item $\tilde{\textbf{x}} \in \mathcal  B_{E(3;2;1,2)}^{(1)}.$
					\item $y_2z_2^2-y_1z_2+1\neq0$ for all $z_2\in \mathbb{D}$ and $\|\Psi_{3}(\cdot;\tilde{\textbf{x}})\|_{H^{\infty}(\mathbb{D})}\leq1$.  Moreover, if  $x_3=x_1y_2,x_2=x_1y_1$ then $|x_1|\leq1$.
					\item For every  $z_1\in \mathbb{D}$, we have
					$\left(\frac{y_1-z_1x_2}{1-z_1x_1},\frac{y_2-z_1x_3}{1-z_1x_1}\right)\in \Gamma_{E(2;1;2)}.$ 
				\end{enumerate}
							\end{prop} 		
The proof of the following theorem is same as the proof of Theorem \ref{matix AE}. Therefore, we skip the proof.
\begin{thm}\label{matix AEE}
				For $\tilde{\textbf{x}}=(x_1,x_2,x_3,y_1,y_2)\in \mathbb C^5$ the following are equivalent.
				\begin{enumerate}
					\item $\tilde{\textbf{x}}=(x_1,x_2,x_3,y_1,y_2)\in \Gamma_{E(3;2;1,2)}$;
					
					\item There exists $3\times 3$ matrix $A\in \mathcal M_{3\times 3}(\mathbb C)$  such that $$x_1=a_{11},x_2=\det \left(\begin{smallmatrix} a_{11} & a_{12}\\
					a_{21} & a_{22}
				\end{smallmatrix}\right)+\det \left(\begin{smallmatrix}
					a_{11} & a_{13}\\
					a_{31} & a_{33}
				\end{smallmatrix}\right),x_3=\operatorname{det}A, y_1=a_{22}+a_{33}, y_2=\det  \left(\begin{smallmatrix}
					a_{22} & a_{23}\\
					a_{32} & a_{33}\end{smallmatrix}\right),$$ \begin{equation}\label{tilgammaa1234}
					(1-|z_2|^2)(\overline{\tilde{\gamma}_1(z_2)}\tilde{\gamma}_1(z_2)+\overline{\tilde{\gamma}_2(z_2)}\tilde{\gamma}_2(z_2))+{\eta(z_2)}^*(I_{3}-A^*A)\eta(z_2)\geq 0, 
				\end{equation}
				and $\det\left(I_{2}-\left(\begin{smallmatrix}a_{22}&a_{23}\\a_{32} & a_{33}
				\end{smallmatrix}\right)\left(\begin{smallmatrix}z_2 & 0 \\0 &z_{2}
				\end{smallmatrix}\right)\right)\neq 0$ for all $z_2\in \mathbb D$.
				\end{enumerate}
			
			\end{thm}
The relationship between $\Gamma_{E(3;2;1,2)}$ and $\Gamma_{E(3;3;1,1,1)}$ is described in the following theorem. The following theorem's proof is comparable to that of Theorem \ref{matix AE12}. Consequently, we omit the proof.
			\begin{thm}\label{matix AE123}
				Assume that $\textbf{x}=(x_1,x_2,x_3,x_4,x_5,x_6,x_7) \in \mathbb{C}^7 .$ Then $\textbf{x} \in \Gamma_{E(3;3;1,1,1)}$ if and only if $(x_1,x_3+\eta x_5,\eta x_7,x_2+\eta x_4,\eta x_6)\in \Gamma_{E(3;2;1,2)}$ for all $\eta \in \mathbb{T}$. 
			\end{thm}
The proof of the following theorem is similar to Theorem \ref{matix barAAAA}. Thus, we omit the proof.
\begin{thm}\label{matix ABCD}

Let $A\in \mathcal M_{3\times 3}(\mathbb C)$ with $\|A\|\leq1.$  Set $$x_1=a_{11},x_2=\det \left(\begin{smallmatrix} a_{11} & a_{12}\\
					a_{21} & a_{22}
				\end{smallmatrix}\right)+\det \left(\begin{smallmatrix}
					a_{11} & a_{13}\\
					a_{31} & a_{33}
				\end{smallmatrix}\right),x_3=\operatorname{det}A, y_1=a_{22}+a_{33}, y_2=\det  \left(\begin{smallmatrix}
					a_{22} & a_{23}\\
					a_{32} & a_{33}\end{smallmatrix}\right).$$  Then $\tilde{\textbf{x}}=(x_1,x_2,x_3,y_1,y_2)\in \Gamma_{E(3;2;1,2)}.$				
			\end{thm}								
%
\section{Geometric properties of  $G_{E(3;3;1,1,1)},$ $\Gamma_{E(3;3;1,1,1)},G_{E(3;2;1,2)}$ and  $\Gamma_{E(3;2;1,2)}$}

In this section, we prove that the domain  $G_{E(3;3;1,1,1)}$ and  $\Gamma_{E(3;3;1,1,1)}$ are non-convex domain and starlike domain about origin but not circled. Additionally, we  present an alternative  proof to show that $\Gamma_{E(3;3;1,1,1)}$ is polynomially convex and linear convex. 
 We also describe the closed boundaries of $G_{E(3;2;1,2)}$ and $\Gamma_{E(3;2;1,2)}$ and show how they are related.

\subsection{Geometric properties of  $G_{E(3;3;1,1,1)}$ and  $\Gamma_{E(3;3;1,1,1)}$}
The geometric characteristics of $G_{E(3;3;1,1,1)}$ and $\Gamma_{E(3;3;1,1,1)}$ are discussed in this section. Firstly, both the domains $G_{E(3;3;1,1,1)}$ and  $\Gamma_{E(3;3;1,1,1)}$ are non-convex domain, indeed,  if $\textbf{x}=(1,i,i,1,1,i,i)$ and $\textbf{y}=(-i,1,-i,-i,-1,i,1),$
then $\textbf{x},\textbf{y}\in \Gamma_{E(3;3;1,1,1)}$ but $\frac{\textbf{x}+\textbf{y}}{2}\notin \Gamma_{E(3;3;1,1,1)}.$  
\begin{thm}
Let $\textbf{x}=(x_1,x_2,x_3,x_4,x_5,x_6,x_7) \in \Gamma_{E(3;3;1,1,1)}.$ Then it satisfies the following conditions
\begin{enumerate}
\item $(x_1,rx_2,rx_3,rx_4,rx_5,rx_6,rx_7)  \in \Gamma_{E(3;3;1,1,1)}$ for all $r$ with $0\leq r < 1.$
\item $(rx_1,x_2,rx_3,rx_4,rx_5,rx_6,rx_7)  \in \Gamma_{E(3;3;1,1,1)}$ for all $r$ with $0\leq r <1.$
\item $(rx_1,rx_2,rx_3,x_4,rx_5,rx_6,rx_7)  \in \Gamma_{E(3;3;1,1,1)}$ for all $r$ with $0\leq r <1.$

\end{enumerate}
	
\end{thm}
\begin{proof}
 We prove that if $\textbf{x}=(x_1,x_2,x_3,x_4,x_5,x_6,x_7) \in \Gamma_{E(3;3;1,1,1)},$ then $(x_1,rx_2,rx_3,rx_4,rx_5,rx_6,rx_7) \in \Gamma_{E(3;3;1,1,1)}$ for all $r$ with $0\leq r < 1$.  The proof of other cases follows similarly. For $\textbf{x}=(x_1,x_2,x_3,x_4,x_5,x_6,x_7) \in \Gamma_{E(3;3;1,1,1)}$ and $\textbf{z}_{J^{(2)}}=(z_1,z_3)\in \mathbb{D}^2$, note that
	\begin{align}\label{psi111}
		\Psi^{(2)}(\textbf{z}_{J^{(2)}},\textbf{x})\nonumber&=\frac{x_2-z_1x_3-z_3x_6+z_1z_3x_7}{1-z_1x_1-z_3x_4+z_1z_3x_5}\\ \nonumber &= \frac{\frac{x_2-z_1x_3}{1-z_1x_1}- z_3\frac{x_6-z_1x_7}{1-z_1x_1}}{1- z_3\frac{x_4-z_1 x_5}{1-z_1 x_1}}
		\\ \nonumber&=\frac{\tilde{x}_1(z_1)-z_3\tilde{x}_3(z_1)}{1-z_3\tilde{x}_2(z_1)}\\&=\Psi(z_3,(\tilde{x}_1(z_1),\tilde{x}_2(z_1),\tilde{x}_3(z_1))). 
	\end{align}
	
From Theorem \ref{phii1} and \eqref{psi111}, we deduce that $\textbf{x}\in \Gamma_{E(3;3;1,1,1)}$ if and only if 
\begin{align}\label{psi171} \textbf{x}_{J^{(2)}}^{\prime}\in \Gamma_{E(2;2;1,1)}~{\rm{and }}~ \sup_{(z_1,z_2)\in \mathbb{D}^2}|\Psi(z_3,(\tilde{x}_1(z_1),\tilde{x}_2(z_1),\tilde{x}_3(z_1)))|\leq 1.\end{align}
The condition 
$\sup_{(z_1,z_2)\in \mathbb{D}^2}|\Psi(z_3,(\tilde{x}_1(z_1),\tilde{x}_2(z_1),\tilde{x}_3(z_1)))|\leq 1$ implies that  
\begin{align}\label{starl2} |1-z_3\tilde{x}_2(z_1)|^2-|\tilde{x}_1(z_1)-z_3\tilde{x}_3(z_1)|^2\geq 0\;\;\mbox{for}\;\;z_1,z_3\in \mathbb{D}.\end{align}
Since $0\leq r<1,$ so we obtain $1-r>0~{\rm{and}}~1+r-2r\Re(z_3\tilde{x}_2(z_1)\geq |1-rz_3\tilde{x}_2(z_1)|^2\geq 0.$ 
It is easy to see that
\begin{align}\label{starl}
		|1-rz_3\tilde{x}_2(z_1)|^2-|r\tilde{x}_1(z_1)-rz_3\tilde{x}_3(z_1)|^2\nonumber&=r^2\{|1-z_3\tilde{x}_2(z_1)|^2-|\tilde{x}_1(z_1)-z_3\tilde{x}_3(z_1)|^2\}\\&+(1-r)(1+r-2r\Re(z_3\tilde{x}_2(z_1))).
	\end{align}
From \eqref{starl2} and \eqref{starl}, we get 
$$|1-rz_3\tilde{x}_2(z_1)|^2-|r\tilde{x}_1(z_1)-rz_3\tilde{x}_3(z_1)|^2\geq 0\;\;\; \mbox{for}\;\; z_1,z_3\in \mathbb{D}$$ which gives

\begin{align}\label{psi71}
\sup_{(z_1,z_3)\in \mathbb{D}^2}|\Psi(z_3,(r\tilde{x}_1(z_1),r\tilde{x}_2(z_1),r\tilde{x}_3(z_1)))|\leq 1.
\end{align}
Note that $\textbf{x}_{J^{(2)}}^{\prime}=(x_1,x_4,x_5)\in \Gamma_{E(2;2;1,1)}$ if and only if $\sup_{z\in \mathbb{D}}|\frac{x_4-zx_5}{1-zx_1}|\leq 1$. For $0\leq r <1$, we have
 $$\sup_{z\in \mathbb{D}}\left|\frac{rx_4-z rx_5}{1-zx_1}\right|\leq \sup_{z\in \mathbb{D}}\left|\frac{x_4-zx_5}{1-zx_1}\right|\leq 1.$$ Thus $(x_1, rx_4, r x_5)$ belongs to $\Gamma_{E(2;2;1,1)}$.
Hence from \eqref{psi171} and \eqref{psi71}, we have $(x_1,rx_2,rx_3,rx_4,rx_5,rx_6,rx_7)$ belongs to 
$\Gamma_{E(3;3;1,1,1)}$. This completes the proof.

\end{proof}
\begin{cor}\label{simplyconn}
$\Gamma_{E(3;3;1,1,1)}$ is simply connected.
\end{cor}

\begin{proof}          
Let $\gamma:[0,1]\to \Gamma_{E(3;3;1,1,1)}$ be a path such that $\gamma(0)=\gamma(1)=(0,0,0,0,0,0,0)$.\\Let $\gamma(t)=(\gamma_1(t),\gamma_2(t),\gamma_3(t),\gamma_4(t),\gamma_5(t),\gamma_6(t),\gamma_7(t))$. Define a map
$H:[0,1]\times [0,1]\to \Gamma_{E(3;3;1,1,1)}$ as folows
{\scriptsize{\begin{eqnarray*}
H(s,t)= \begin{cases}
               (2s \gamma_1(t),0,0,0,0,0,0,0) , & \mbox{if}\;\; s\in [0,\tfrac{1}{2}]\\
       (\gamma_1(t),(2s-1)\gamma_2(t),(2s-1)\gamma_3(t),(2s-1)\gamma_4(t),(2s-1)\gamma_5(t),(2s-1)\gamma_6(t),(2s-1)\gamma_7(t))       & \mbox{if}\;\; s\in [\tfrac{1}{2},1].
            \end{cases}
\end{eqnarray*}}}
Thus $\gamma$ is  homotopic to the constant path at $(0,0,0,0,0,0,0)$. Hence  $\Gamma_{E(3;3;1,1,1)}$ is simply connected.        
\end{proof}

\begin{rem}	The point $\textbf{x}=(1,1,1,1,1,1,1)\in \Gamma_{E(3;3;1,1,1)}$ but $i\textbf{x}=(i,i,i,i,i,i,i)$ does not belong to $\Gamma_{E(3;3;1,1,1)}$ which indicates that neither $\Gamma_{E(3;3;1,1,1)}$ nor $G_{E(3;3;1,1,1)}$ is circled.
\end{rem}
Zaplolowski [Proposition $3.20$, \cite{Pawel}] established that  the domain $\Gamma_{E(n;s;r_1,\cdots,r_s)}$ is polynomially convex if $r_2=r_3=\cdots=r_s=1$ with $\sum_{i=1}^{s}r_i=n$. In the following theorem, we provide a new proof to  demonstrate that $\Gamma_{E(3;3;1,1,1)}$ is a polynomially convex.

\begin{thm}
	$\Gamma_{E(3;3;1,1,1)}$  is polynomially convex.
\end{thm}

\begin{proof}
Let $\textbf{a}=(a_1,a_2,\ldots,a_7)\in \mathbb{C}^3\setminus \Gamma_{E(3;3;1,1,1)}$. 
By Theorem \ref{maingamma}, $$\textbf{a}=(a_1,a_2,\ldots,a_7) \notin \Gamma_{E(3;3;1,1,1)}$$ if and only if either \begin{itemize}
\item $(a_1,a_2,a_3) \notin \Gamma_{E(2;2;1,1)}$ or $(a_1,a_4,a_5) \notin \Gamma_{E(2;2;1,1)}$ or $(a_2,a_4,a_6) \notin \Gamma_{E(2;2;1,1)}$.
\end{itemize}
or 
\begin{itemize}
\item
 $(a_1,a_2,a_3) \in \Gamma_{E(2;2;1,1)}$, there exists $(z_1,w_1)\in \mathbb{D}^2$ such that $|\Psi^{(3)}((z_1,w_1),\textbf{a})|>1$;
 and $(a_1,a_4,a_5) \in \Gamma_{E(2;2;1,1)}$, there exists $(z_0,w_0)\in \mathbb{D}^2$ such that
$$\left|\frac{a_2-z_0a_3-w_0 a_6+ z_0w_0 a_7}{1-z_0 a_1-w_0 a_4+ z_0w_0 a_5}\right|=\left|\Psi^{(2)}((z_0,w_0),\textbf{a})\right|>1;$$
 and $(a_2,a_4,a_6) \in \Gamma_{E(2;2;1,1)}$, there exists $(z_2,w_2)\in \mathbb{D}^2$ such that $|\Psi^{(1)}((z_2,w_2),\textbf{a})|>1$.
\end{itemize}

\noindent $\textbf{Case 1}$: If $(a_1,a_4,a_5)\notin \Gamma_{E(2;2;1,1)}$, by [Theorem 2.9, \cite{Abouhajar}], there exists a polynomial $\tilde{f}$ such that
$|\tilde{f}|\leq 1$ on $\Gamma_{E(2;2;1,1)}$ and $|\tilde{f}(a_1,a_4,a_5)|>1$. Consider a polynomial $f$ defined on $\mathbb{C}^7$ by
$f(x_1,x_2,x_3,x_4,x_5,x_6,x_7)=\tilde{f}(x_1,x_4,x_5)$. Clearly, $|f|\leq 1$ on $\Gamma_{E(3;3;1,1,1)}$ and $$|f(\textbf{a})|=|\tilde{f}(a_1,a_4,a_5)|>1.$$
Similarly, if $(a_1,a_2,a_3)\notin \Gamma_{E(2;2;1,1)}$ or $(a_2,a_4,a_6)\notin \Gamma_{E(2;2;1,1)}$, then there exists a polynomial $f$ such that $|f|\leq 1$ on 
 $\Gamma_{E(3;3;1,1,1)}$ and $|f(\textbf{a})|>1$.

\noindent $\textbf{Case 2}$: 
Let $\textbf{a}\in \overline{\mathbb{D}}^7$, $(a_1,a_4,a_5) \in \Gamma_{E(2;2;1,1)}$, there exists $(z_0,w_0)\in \mathbb{D}^2$ such that $$\left|\Psi^{(2)}((z_0,w_0),\textbf{a})\right|>1.$$

Let $\textbf{x}=(x_1,x_2,x_3,x_4,x_5,x_6,x_7)\in \mathbb{C}^7$ and $B_{\textbf{x}}=\begin{pmatrix} b_{11} & b_{12}\\
b_{21} & b_{22}\end{pmatrix}$ such that $b_{11}=x_1, b_{22}=x_4$ and $\det (B_{\textbf{x}})=x_5$.
Let $f_N$ be the polynomial given by 
$$f_N(\textbf{x})=(x_2-z_0x_3-w_0x_6+z_0w_0x_7)  \det \left(I+B_{\textbf{x}}(\begin{smallmatrix} z_0 & 0\\ 0 & w_0\end{smallmatrix})+ \left(B_{\textbf{x}}(\begin{smallmatrix} z_0 & 0\\ 0 & w_0\end{smallmatrix})\right)^2+\ldots+ \left(B_{\textbf{x}}(\begin{smallmatrix} z_0 & 0\\ 0 & w_0\end{smallmatrix})\right)^N\right).$$ 

Let $\textbf{y}\in \overline{\mathbb{D}}^7$ with $(y_1,y_4,y_5)\in \Gamma_{E(2;2;1,1)}$, there exists a matrix $B_{\textbf{y}}=\begin{pmatrix} b_{11} & b_{12}\\
b_{21} & b_{22}\end{pmatrix}$ with $\|B_{\textbf{y}}\|\leq 1$ such that $b_{11}=y_1, b_{22}=y_4, \det(B_{\textbf{y}})=y_5$. We have

\small{\begin{eqnarray*}
&&|f_N(\textbf{y})-\psi^{(2)}((z_0,w_0),\textbf{y})|\\
&=&|y_2-z_0y_3-w_0y_6+z_0w_0y_7| \left|\det \left(I+B_{\textbf{y}}(\begin{smallmatrix} z_0 & 0\\ 0 & w_0\end{smallmatrix})+ \ldots+ \left(B_{\textbf{y}}(\begin{smallmatrix} z_0 & 0\\ 0 & w_0\end{smallmatrix})\right)^N\right)-
\frac{1}{1-z_0y_1-w_0y_4+z_0w_0y_5}\right|\\
&=& |y_2-z_0y_3-w_0y_6+z_0w_0y_7| \left|\det \left(I+B_{\textbf y}(\begin{smallmatrix} z_0 & 0\\ 0 & w_0\end{smallmatrix})+ \ldots+ \left(B_{\textbf y}(\begin{smallmatrix} z_0 & 0\\ 0 & w_0\end{smallmatrix})\right)^N\right)-
\det \left(I-B_{\textbf y}(\begin{smallmatrix} z_0 & 0\\ 0 & w_0\end{smallmatrix}) \right)^{-1}\right|\\
&\leq & 4 \left|\det \left(I+B_{\textbf y}(\begin{smallmatrix} z_0 & 0\\ 0 & w_0\end{smallmatrix})+ \ldots+ \left(B_{\textbf y}(\begin{smallmatrix} z_0 & 0\\ 0 & w_0\end{smallmatrix})\right)^N\right)-
\det \left(I-B_{\textbf y}(\begin{smallmatrix} z_0 & 0\\ 0 & w_0\end{smallmatrix}) \right)^{-1}\right|.
\end{eqnarray*}}

Let $0< \epsilon < \frac{1}{3} (|\psi^{(2)}((z_0,w_0),\textbf{a})|-1)$. Choose $N$ so large that
$|f_N(\textbf {y})-\psi^{(2)}((z_0,w_0),\textbf{y})|<\epsilon$ for all $y\in \overline{\mathbb{D}}^7$ with $(y_1,y_4,y_5)\in \Gamma_{E(2;2;1,1)}$. Thus, we have
\begin{align}\label{de}
|\psi^{(2)}((z_0,w_0),\textbf{y})|-\epsilon < |f_N(\textbf{y})|< |\psi^{(2)}((z_0,w_0),\textbf{y})|+\epsilon, \; \mbox{for all}\; \textbf{y}\in \overline{\mathbb{D}}^7\;\mbox{with}\; (y_1,y_4,y_5)\in \Gamma_{E(2;2;1,1)}.
\end{align}
From \eqref{de}, we have 
$$|f_N(\textbf{y})|<1+\epsilon\; \mbox{for all}\; \textbf{y}\in \Gamma_{E(3;3;1,1,1)}, $$
and 
\begin{eqnarray*}
|f_N(\textbf{a})| &>& |\psi^{(2)}((z_0,w_0),\textbf{a})|-\epsilon \\
&>& (3 \epsilon +1)-\epsilon\\
&=& (1+2\epsilon).
\end{eqnarray*}
We take the polynomial $f(\textbf{x})= (1+\epsilon)^{-1}f_N(\textbf{x})$ for $\textbf{x}\in \mathbb{C}^7$. Clearly, $|f|\leq 1$ on $\Gamma_{E(3;3;1,1,1)}$
and $|f(\textbf{a})|>1$. This completes the proof.
\end{proof}
%
%

\begin{defn}
	A domain $G$ is called an analytic retract of a domain $D$ if there exist analytic maps $\theta: G \longrightarrow D$, $\iota: D \longrightarrow G$ such that $\iota \circ \theta=\mathrm{id}_{G}$.
\end{defn}  

\begin{lem}\label{retract}
	The mappings
	\begin{enumerate}
		\item $G_{E(2;2;1,1)} \ni \textbf{x}^{\prime} \stackrel{\theta}{\longmapsto}\left(\textbf{x}^{\prime}, 0\right) \in G_{E(3;3;1,1,1)}$, and
		\item $G_{E(3;3;1,1,1)} \ni\left(\textbf{x}^{\prime}, \textbf{x}^{\prime \prime}\right) \stackrel{\iota}{\mapsto} \textbf{x}^{\prime} \in G_{E(2;2;1,1)}$
	\end{enumerate}
	are well defined.
\end{lem}
\begin{proof}
	Suppose $\textbf{x}^{\prime}=(x_1,x_2,x_3)\in G_{E(2;2;1,1)}$ and $\textbf{x}=\left(\textbf{x}^{\prime}, 0,0,0,0\right) \in \mathbb{C}^{7}$. We have
	$$
	R^{(3;3,1,1,1)}_{\textbf{x}}(\textbf{z})=R^{(3;3;1,1,1)}_{\left(\textbf{x}^{\prime}, 0\right)}(\textbf{z})=1+\sum_{j=1}^{7}(-1)^{\left|\alpha^{j}\right|} x_{j} \textbf{z}^{\alpha^{j}}=1-x_1z_1-x_2z_2+x_3z_1z_2=R^{(2;2;1,1)}_{\textbf{x}^{\prime}}\left(\textbf{z}^{\prime}\right)
	$$
	for all $\textbf{z}=\left(z_1,z_2,z_3\right)=(\textbf{z}^{\prime},z_3) \in \mathbb{C}^{3}$. Since $\textbf{x}^{\prime} \in G_{E(2;2;1,1)}$, so  $R^{(2;2;1,1)}_{\textbf{x}^{\prime}}(\textbf{z}^{\prime}) \neq 0$ for  all $\textbf{z}^{\prime} \in \bar{\mathbb{D}}^2$ which implies that $\left(\textbf{x}^{\prime}, 0\right)\in G_{E(3;3;1,1,1)}.$ Hence the map $\theta$ is well defined. 
	
	Let $\textbf{x} \in G_{E(3;3;1,1,1)}$,  in other words $R^{(3;3;1,1,1)}_{\textbf{x}}(\textbf{z})\neq 0$ for all $\textbf{z}\in \mathbb D^3.$ In particular, $R^{(3;3;1,1,1)}_{\textbf{x}}\left(\textbf{z}^{\prime}, 0\right) \neq 0$ for all $\textbf{z}^{\prime} \in \bar{\mathbb{D}}^2$.   Since $R^{(3;3;1,1,1)}_{\textbf{x}}\left(\textbf{z}^{\prime}, 0\right)=R^{(2;2;1,1)}_{\textbf{x}^{\prime}}(\textbf{z}^{\prime})$, so $\textbf{x}^{\prime} \in G_{E(2;2;1,1)}$. Therefore, the map $\iota$ is well defined.
\end{proof}

\begin{lem}\label{analytic retr 1}
	$G_{E(2;1,2)}$ is an analytic retraction of $G_{E(3;3;1,1,1)}$.
\end{lem}

\begin{proof}
	Let $\theta_1$ and $i_1$ be  the maps from $G_{E(2;1,2)}$ to $G_{E(2;2;1,1)} $ and $G_{E(2;2;1,1)} $ to $G_{E(2;1,2)}$ respectively, which are as follows:
	$$\theta_1(s,p)=(\frac{s}{2},\frac{s}{2}, p)~{\rm {and}}~ i_1(x_1,x_2,x_3)=(x_1+x_2,x_3).$$
	Then $G_{E(2;1,2)}$ is an analytic retraction of $G_{E(2;2;1,1)}$. Set $i^{\prime}=i_1\circ i$ and $\theta^{\prime}=\theta \circ \theta_1$.
	Notice that $\theta^{\prime}$ and $i^{\prime}$ are the maps from $G_{E(2;1,2)}$ to $G_{E(3;3;1,1,1)} $ and $G_{E(3;3;1,1,1)} $ to $G_{E(2;1,2)}$ respectively and $i^{\prime}\circ \theta^{\prime}=id_{G_{E(2;1,2)}}$. This shows that $G_{E(2;1,2)}$ is an analytic retraction of $G_{E(3;3;1,1,1)}$.
\end{proof}
\begin{lem}\label{analytic retr}
	$G_{E(3;2;1,2)}$ is an analytic retraction of $G_{E(3;3;1,1,1)}$.
\end{lem}

\begin{proof}
	Let $\theta_2$ and $i_2$ be the maps from $G_{E(3;2;1,2)}$ to $G_{E(3;3;1,1,1)} $ and $G_{E(3;3;1,1,1)} $ to $G_{E(3;2;1,2)}$ respectively which are as follows:
	$$\theta_2(y_1,\ldots,y_5)=(y_1,\frac{y_4}{2}, \frac{y_2}{2},\frac{y_4}{2},\frac{y_2}{2},y_5,y_3)~{\rm {and}}~ i_2(x_1,\ldots,x_7)=(x_1,x_3+x_5,x_7,x_2+x_4,x_6).$$
	Notice that $i_2\circ \theta_2=id_{G_{E(3;2;1,2)}}$. This shows that $G_{E(3;2;1,2)}$ is an analytic retraction of $G_{E(3;3;1,1,1)}$.
\end{proof}

We recall the definition of linearly convex domain from \cite{Pawel}. 
\begin{defn}
	A domain $D\subset \mathbb C^n$ is called linearly convex if its complement is a union of affine complex hyperplane.
\end{defn}
P. Zapalowski showed that if $r_2=\cdots=r_s=1$, then the domain $G_{E(n;s;r_1,1,\cdots,1)}$ is linearly convex and, as a result, pseudoconvex [Proposition 3.18,\cite{Pawel}]. In this subsection, we provide a different proof of linear convexity of $G_{E(3;3;1,1,1)}$. To show this, let us   consider  the map $\Phi_{\eta}:\mathbb C^7\to \mathbb C^5$ defined by
$$\Phi_{\eta}(x_1,\cdots,x_7)=(x_1,x_3+\eta x_5,\eta x_7,x_2+\eta x_4,\eta x_6)~{\rm{ for~ all ~}} \textbf{x}\in G_{E(3;3;1,1,1)}~{\rm{ and}}~ \eta\in \mathbb {\bar{D}}.$$

\begin{prop}
	$G_{E(3;3;1,1,1)} $ is linearly convex.
	
\end{prop}
\begin{proof}
	Let $\textbf{x}\notin G_{E(3;3;1,1,1)}.$ By Theorem \ref{matix AE12}, there exists $\eta\in \mathbb {\bar{D}}$ such that $\Phi_{\eta}(\textbf{x})\notin G_{E(3;2;1,2)} .$ Since $G_{E(3;2;1,2)} $ is linearly convex, so we can find a complex hyperplane $$l=\{(y_1,\cdots,y_5):\sum_{i=1}^{5}a_iy_i=c\}$$ with $\Phi_{\eta}(\textbf{x})\in l$ and $l\cap G_{E(3;2;1,2)} =\emptyset.$ It is easy to check that  $\textbf{x}\in L$ and $G_{E(3;3;1,1,1)}\cap L=\emptyset$, where $$L=\{(z_1,\cdots,z_7):a_1z_1+a_4z_2+a_2 z_3+a_4\eta z_4+a_2\eta z_5+a_5\eta z_6+a_3\eta z_7=c\}.$$
This shows that $G_{E(3;3;1,1,1)} $ is linearly convex.	
	
\end{proof}


\subsection{Geometric properties of  $G_{E(3;2;1,2)}$ and  $\Gamma_{E(3;2;1,2)}$}
We first show that both the domains $G_{E(3;2;1,2)}$ and  $\Gamma_{E(3;2;1,2)}$ are non-convex domain. In order to show this suppose $\textbf{x}=(1,1+i,i+1,i,i)$ and $\textbf{y}=(-i,1-i,-i-1,i,-1).$
Then, it is simple to verify that $\textbf{x},\textbf{y}\in \Gamma_{E(3;2;1,2)}$ but $\frac{\textbf{x}+\textbf{y}}{2}\notin \Gamma_{E(3;2;1,2)}.$ This shows  the non-convexity of the domains $G_{E(3;2;1,2)}$ and  $\Gamma_{E(3;2;1,2)}.$

\begin{prop}
$\Gamma_{E(3;2;1,2)}$ is simply connected.
\end{prop}
\begin{proof}
Proof follows from Corollary \ref{simplyconn} and Lemma \ref{analytic retr}.
\end{proof}

\begin{rem}The point $\tilde{\textbf{x}}=(1,2,2,1,1)\in \Gamma_{E(3;2;1,2)}$ but $i\tilde{\textbf{x}}=(i,2i,2i,i,i)$ does not belong to $\Gamma_{E(3;2;1,2)}$ which indicates that neither $\Gamma_{E(3;2;1,2)}$ nor $G_{E(3;2;1,2)}$ is circular.
\end{rem}

\subsection{Closed boundary of $\Gamma_{E(3;3;1,1,1)} $ and $\Gamma_{E(3;2;1,2)} $}
Let $\Omega$ be a domain in $ \mathbb C^n$ and $\bar{\Omega}$ denote its closure. Let $C\subseteq \bar{\Omega}.$ We say that $C$ is a boundary  for $\Omega$ if every function in $\mathcal A(\Omega)$ attains its maximum modulus on $C$. The smallest closed boundary of $\Omega,$ contained in all the closed boundaries of $\Omega$, is called the distinguished boundary of $\Omega$ (or the Shilov boundary of $\mathcal A(\Omega)$). We denote it as $b\Omega.$ 
The theory of uniform algebras [Corollary $2.2.10$, \cite{Browder}] indicates that when $\bar{\Omega}$ is polynomially convex, then the distinguished boundary of $\Omega$ is non-empty. Hence distinguished boundary of $\Gamma_{E(3;3;1,1,1)} $ and $\Gamma_{E(3;2;1,2)}$ are non-empty. We investigate the closed boundaries of $\Gamma_{E(3;3;1,1,1)} $ and $\Gamma_{E(3;2;1,2)}$, respectively, in this section. We also establish the relation between the closed boundary of $\Gamma_{E(3;3;1,1,1)} $  and $\Gamma_{E(3;2;1,2)} .$ 

In the space $\mathcal M_{3\times 3}(\mathbb C)$ of $3\times 3$ complex matrices, let $\mathbb B$ represent the open unit ball with the standard operator norm and $\bar{\mathbb B}$ denote its closure. 
From Theorem \ref{matix A} and Theorem \ref{matix barAAAA}, it is evident $\pi_{E(3;3;1,1,1)}(\mathbb B)\subseteq G_{E(3;3;1,1,1)}$ and $\pi_{E(3;3;1,1,1)}(\bar{\mathbb B})\subseteq \Gamma_{E(3;3;1,1,1)} .$  Let $\mathcal U(3)$ be the set of all $3\times 3$ unitary matrix. Let $$K=\{\textbf{x}=(x_1,\ldots,x_7)\in  \Gamma_{E(3;3;1,1,1)}:x_1=\bar{x}_6x_7, x_3=\bar{x}_4x_7,x_5=\bar{x}_2x_7 ~{\rm{and}}~|x_7|=1\}.$$ 
\begin{thm}\label{distinguish}
$\pi_{E(3;3;1,1,1)}(\mathcal U(3))\subseteq K.$
\end{thm}
\begin{proof}
Suppose that $\textbf{x}\in \pi_{E(3;3;1,1,1)}(\mathcal U(3)),$  that is, $\textbf{x}=\pi_{E(3;3;1,1,1)}(U_1)$ for some $U_1\in \mathcal U(3).$ Then by Theorem \ref{matix barAAAA}, we have $\textbf{x}\in \Gamma_{E(3;3;1,1,1)}$.    Note that $\textbf{x}\in \Gamma_{E(3;3;1,1,1)}$ if and only if $$\left(\tilde{z}_1(z_3)=\frac{x_1-z_3x_5}{1-x_4z_3},\tilde{z}_2(z_3)=\frac{x_2-z_3x_6}{1-x_4z_3}~{\rm{and}}~\tilde{z}_3(z_3)=\frac{x_3-z_3x_7}{1-x_4z_3}\right)\in\Gamma_{E(2;2;1,1)}$$ for all  $z_3\in \mathbb D.$ As $\left(\tilde{z}_1(z_3),\tilde{z}_2(z_3),\tilde{z}_3(z_3)\right)\in\Gamma_{E(2;2;1,1)}$ for all  $z_3\in \mathbb D,$ then by characterization of $\Gamma_{E(2;2;1,1)}$ we have $|\tilde{z}_3(z_3)|\leq 1$ for all $z_3\in \mathbb D$, which gives 
\begin{equation}\label{cgargama}|x_3-\bar{x}_4x_7|+|x_4-\bar{x}_3x_7|\leq 1-|x_7|^2.
\end{equation}
Since $U_1\in \mathcal U(3)$ and $|x_7|=|\det(U_1)|=1$, it follows from \eqref{cgargama} that $x_3=\bar{x}_4x_7$ and $x_4=\bar{x}_3x_7$. As $|x_4|\leq 1$, we consider the following two cases: 

\noindent{\bf{Case $1$}:} Assume that $|x_4|=1.$ Since $|\tilde{z}_i(z_3)|\leq 1$  for $i=1,2$  for all $z_3\in \mathbb D$, then again by characterization of $\Gamma_{E(2;2;1,1)}$ we have $x_2=\bar{x}_4x_6, x_6=x_2x_4,x_1=\bar{x}_4x_5$ and $x_5=x_1x_4.$ Since $\textbf{x}\in \Gamma_{E(3;3;1,1,1)},$ we obtain $(x_1,x_2,x_3)\in \Gamma_{E(2;2;1,1)}.$ However, $|x_3|=1$ and $(x_1,x_2,x_3)\in \Gamma_{E(2;2;1,1)}$, it implies that $x_1=\bar{x}_2x_3$ and $x_2=\bar{x}_1x_3.$ Also, we have $$\bar{x}_2 x_7=x_1\bar{x}_3x_7=x_1x_4=x_5~{\rm{and}}~ \bar{x}_1x_7=x_2\bar{x}_3x_7=x_2x_4=x_6.$$ This shows that $\textbf{x}\in K.$

\noindent{\bf{Case $2$:}} 
Suppose that $|x_4|<1.$ Then $1-x_4z_3\neq 0$ for all $z_3\in \bar{\mathbb D}.$ Note that as $x_3=\bar{x}_4x_7$  and $|x_7|=1$, then $|\tilde{z}_3(z_3)|=1$ for all $z_3\in \mathbb T.$ As a result,  $|\tilde{z}_3(z_3)|\leq 1$ for all $z_3\in \bar{\mathbb D}.$ Consequently, by maximum modulus theorem $$\sup_{z_3\in \bar{\mathbb D}}|\tilde{z}_3(z_3)|=\sup_{z_3\in {\mathbb T}}|\tilde{z}_3(z_3)|=1.$$
Since $U_1\in \mathcal U(3)$, it follows from \eqref{A norm} that $\mathcal F_{U_1}^*(z_3)\mathcal F_{U_1}(z_3)=I_{2\times 2}$ for all $z_3\in \mathbb T.$ This shows that $\mathcal F_{U_1}(z_3)$ is unitary for all $z_3\in \mathbb T$ and hence $(\tilde{z}_1(z_3),\tilde{z}_2(z_3),\tilde{z}_3(z_3))\in b\Gamma_{E(2;2;1,1)}$ for all $z_3\in \mathbb T. $ Then by characterization of $b\Gamma_{E(2;2;1,1)}$ [Theorem $7.1$, \cite{Abouhajar}], it follows that $\tilde{z}_1(z_3)=\overline{\tilde{z}_2(z_3)}\tilde{z}_3(z_3)$ for all  $z_3\in \mathbb T$, which gives 
\begin{equation}\label{tildez1}x_1+\bar{x}_4x_5=\bar{x}_2x_3+\bar{x}_6x_7, x_5=\bar{x}_2x_7~{\rm{and}}~x_1\bar{x}_4=\bar{x}_6x_3.\end{equation}
From \eqref{tildez1}, we conclude that $x_1=\bar{x}_6x_7$ and hence $\textbf{x}\in K.$
This completes the proof.

\end{proof}
For $\textbf{x}=(x_1,x_2,x_3,x_4,x_5,x_6,x_7) \in \Gamma_{E(3;3;1,1,1)}$ and $(z_2,z_3)\in \mathbb{D}^2$, note that
	\begin{align}\label{psi11154}
		\Psi^{(1)}((z_2,z_3),\textbf{x})\nonumber&=\frac{x_1-z_2x_3-z_3x_5+z_2z_3x_7}{1-z_2x_2-z_3x_4+z_2z_3x_6}\\ \nonumber &= \frac{\frac{x_1-z_3x_5}{1-z_3x_4}- z_2\frac{x_3-z_3x_7}{1-z_3x_4}}{1- z_2\frac{x_2-z_3 x_6}{1-z_3 x_4}}
		\\ \nonumber&=\frac{\tilde{z}_1(z_3)-z_2\tilde{z}_3(z_3)}{1-z_2\tilde{z}_2(z_3)}\\&=\Psi(z_2,(\tilde{z}_1(z_3),\tilde{z}_2(z_3),\tilde{z}_3(z_3))),
	\end{align}
and 
\begin{align}\label{psi111546}
		\Psi^{(1)}((z_2,z_3),\textbf{x})\nonumber&=\frac{x_1-z_2x_3-z_3x_5+z_2z_3x_7}{1-z_2x_2-z_3x_4+z_2z_3x_6}\\  &=\Psi(z_3,(\tilde{y}_1(z_2),\tilde{y}_2(z_2),\tilde{y}_3(z_2))). 
	\end{align}
Similarly, for $\textbf{x}=(x_1,x_2,x_3,x_4,x_5,x_6,x_7) \in \Gamma_{E(3;3;1,1,1)}$, we observe that
\begin{align}\label{psi1115467}
		\Psi^{(2)}((z_1,z_3),\textbf{x})=\Psi(z_3,(\tilde{x}_1(z_1),\tilde{x}_2(z_1),\tilde{x}_3(z_1)))=\Psi(z_1,(\tilde{z}_1(z_3),\tilde{z}_2(z_3),\tilde{z}_3(z_3)))~{\rm{for~all~}} z_1,z_3\in \mathbb D,
	\end{align}
and \begin{align}\label{psi11154678}
		\Psi^{(3)}((z_1,z_2),\textbf{x})=\Psi(z_2,(\tilde{x}_2(z_1),\tilde{x}_1(z_1),\tilde{x}_3(z_1)))=\Psi(z_1,(\tilde{y}_2(z_2),\tilde{y}_1(z_2),\tilde{y}_3(z_2)))~{\rm{for~all~}} z_1,z_2\in \mathbb D,
	\end{align} respectively. Set $\tilde{\textbf{x}}^{(z_1)}=(\tilde{x_1}(z_1),\tilde{x_2}(z_1),\tilde{x_3}(z_1)), \tilde{\textbf{y}}^{(z_2)}= (\tilde{y_1}(z_2),\tilde{y_2}(z_2),\tilde{y_3}(z_2))$ and $\tilde{\textbf{z}}^{(z_3)}=(\tilde{z_1}(z_3),\tilde{z_2}(z_3),\tilde{z_3}(z_3)).$

\begin{thm}\label{relation bw}
	Let $\textbf{x}\in \mathbb{C}^7$. Then the following conditions are equivalent

	\begin{enumerate}
		\item $\textbf{x}\in K=\{(x_1,\dots,x_7) \in \Gamma_{E(3;3;1,1,1)}:x_1=\bar{x}_6x_7,x_2=\bar{x}_5x_7,x_4=\bar{x}_3x_7,|x_7|=1\}$.
		
		\item 	$\begin{cases}
			\tilde{\textbf{z}}^{(z_3)}\in b\Gamma_{E(2;2;1,1)}~\text{and}~\tilde{\textbf{y}}^{(z_2)}\in b\Gamma_{E(2;2;1,1)}~\text{for all}~z_2,z_3\in \mathbb{T}~\text{with}~|x_4|<1 ~\text{and}~|x_2|<1 \\
			\tilde{\textbf{z}}^{(z_3)}\in b\Gamma_{E(2;2;1,1)}~~\text{and}~\tilde{\textbf{y}}^{(z_2)}\in b\Gamma_{E(2;2;1,1)}for~all~z_2,z_3\in \mathbb{D}~with~|x_4|=1\\
			\tilde{\textbf{z}}^{(z_3)}\in b\Gamma_{E(2;2;1,1)}~~\text{and} ~\tilde{\textbf{y}}^{(z_2)}\in b\Gamma_{E(2;2;1,1)}~for~all~z_2,z_3\in \mathbb{D}~with~|x_2|=1;
		\end{cases}$
	    \item 	$\begin{cases}
	\tilde{\textbf{x}}^{(z_1)}\in b\Gamma_{E(2;2;1,1)}~\text{and}~\tilde{\textbf{y}}^{(z_2)}\in b\Gamma_{E(2;2;1,1)}~\text{for all}~z_1,z_2\in \mathbb{T}~\text{with}~|x_1|<1 ~\text{and}~|x_2|<1 \\
		  \tilde{\textbf{x}}^{(z_1)}\in b\Gamma_{E(2;2;1,1)}~~\text{and}~\tilde{\textbf{y}}^{(z_2)}\in b\Gamma_{E(2;2;1,1)}~for~all~z_1,z_2\in \mathbb{D}~with~|x_1|=1\\
			 \tilde{\textbf{x}}^{(z_1)}\in b\Gamma_{E(2;2;1,1)}~~\text{and}~\tilde{\textbf{y}}^{(z_2)}\in b\Gamma_{E(2;2;1,1)}~for~all~z_1,z_2\in \mathbb{D}~with~|x_2|=1;
	     \end{cases}$
		\item 	$\begin{cases}
			\tilde{\textbf{z}}^{(z_3)}\in b\Gamma_{E(2;2;1,1)}~\text{and}~\tilde{\textbf{x}}^{(z_1)}\in b\Gamma_{E(2;2;1,1)}~\text{for all}~z_2,z_3\in \mathbb{T}~\text{with}~|x_4|<1 ~\text{and}~|x_1|<1 \\
		\tilde{\textbf{z}}^{(z_3)}\in b\Gamma_{E(2;2;1,1)}~\text{and}~\tilde{\textbf{x}}^{(z_1)}\in b\Gamma_{E(2;2;1,1)}~for~all~z_1,z_3\in \mathbb{D}~with~|x_1|=1\\
			\tilde{\textbf{z}}^{(z_3)}\in b\Gamma_{E(2;2;1,1)}~\text{and}~\tilde{\textbf{x}}^{(z_1)}\in b\Gamma_{E(2;2;1,1)}~for~all~z_1,z_3\in \mathbb{D}~with~|x_4|=1.
	        	\end{cases}$
	   	\end{enumerate}
\end{thm}
\begin{proof}
We prove $$\begin{matrix}
	(2) & \Leftrightarrow & (1) & \Leftrightarrow(3) \\
	& & \Updownarrow & & \\
	& & (4) &  & 
\end{matrix}.$$ 
First we show that $(1)\Leftrightarrow(2)$. Suppose that $\textbf{x}\in K.$ 
In order to show $(1)\implies (2)$, we must take into consideration the following two cases:
\newline \textbf{Case $1:$} We assume that  $|x_4|<1$ and $|x_2|<1.$ As $|x_4|<1$ and $|x_2|<1,$ we have $1-x_4z_3\neq0$ and $1-x_2z_2\neq 0$  for all $z_2,z_3\in \bar{\mathbb D}$. In order to show that $(\tilde{z_1}(z_3),\tilde{z_2}(z_3),\tilde{z_3}(z_3))\in b\Gamma_{E(2;2;1,1)}$ for all $z_3\in \mathbb{T}$, by the characterization of $b\Gamma_{E(2;2;1,1)},$ it is sufficient to establish that $|\tilde{z_3}(z_3)|=1$  and $\overline{\tilde{z_2}(z_3)}\tilde{z_3}(z_3)=\tilde{z_1}(z_3)$ for all $z_3\in \mathbb T$.
Notice that as $|x_7|=1,$ we get
\begin{equation}\label{equation1}
|\tilde{z_3}(z_3)|=|\frac{x_3-z_3x_7}{1-x_4z_3}|=|\frac{\bar{x}_4x_7-z_3x_7}{1-x_4z_3}|=|x_7||\frac{\bar{x}_4-z_3}{1-x_4z_3}|=1~\text{for all}~z_3 \in \mathbb{T}.
\end{equation} 
Since $\textbf{x}\in K$, so $\textbf{x}\in \Gamma_{E(3;3;1,1,1)}$. By using the characterization of $\Gamma_{E(3;3;1,1,1)}$ and $\Gamma_{E(2;2;1,1)}$, we observe that $|\tilde{z_2}(z_3)|\leq1$ for all $z_3\in \mathbb T.$ For $\textbf{x}\in K$, we have 
\begin{align}\label{equation3}
\overline{\tilde{z_2}(z_3)}\tilde{z_3}(z_3)&=\frac{\overline{x_2-z_3x_6}}{1-\overline{x_4z_3}}.\frac{x_3-z_3x_7}{1-x_4z_3}\nonumber\\
&=\frac{\bar{x}_2x_3+\bar{x}_6x_7-z_3\bar{x}_2x_7-\bar{z}_3\bar{x}_6x_3}{(1-\overline{x_4z_3})(1-x_4z_3)}\nonumber\\
&=\frac{\bar{x}_2\bar{x}_4x_7+\bar{x}_6x_7-z_3\bar{x}_2x_7-\bar{z}_3\bar{x}_6\bar{x}_4x_7}{(1-\overline{x_4z_3})(1-x_4z_3)}\nonumber\\
&=\frac{(1-\overline{x_4z_3})(\bar{x}_6x_7-z_3\bar{x}_2x_7)}{(1-\overline{x_4z_3})(1-x_4z_3)}\nonumber\\ &=\frac{x_1-z_3x_5}{1-x_4z_3}\nonumber\\&=\tilde{z_1}(z_3).
\end{align}
Thus from  \eqref{equation1} and \eqref{equation3} we conclude that $(\tilde{z_1}(z_3),\tilde{z_2}(z_3),\tilde{z_3}(z_3)) \in b\Gamma_{E(2;2;1,1)}$ for all $z_3\in \mathbb{T}$. By using the similar argument, we demonstrate  that $\tilde{\textbf{y}}^{(z_2)}\in b\Gamma_{E(2;2;1,1)}$ for all $z_2\in \mathbb T.$
\newline \textbf{Case $2:$} Assume that $|x_4|=1$.  Then we get that $1-x_4z_3\neq0$ for all $z_3\in \mathbb{D}$. Since $\textbf{x}\in K$, we can infer that  $|x_3|=|\bar{x}_4x_7|=1$.  As $\textbf{x}\in \Gamma_{E(3;3;1,1,1)}$, by the characterization of $\Gamma_{E(3;3;1,1,1)}$, we have $(x_1,x_4,x_5),(x_1,x_2,x_3)$ and $(x_2,x_4,x_6) \in \Gamma_{E(2;2;1,1)}$.  Because $|x_7|=|x_3|=|x_4|=1,$ by the characterization of $ \Gamma_{E(2;2;1,1)},$ we obtain $$x_1x_4=x_5,~x_1=\bar{x}_4x_5,~x_4x_2=x_6,~x_2=\bar{x}_4x_6,~ x_1=\bar{x}_2x_3 ~{\rm{and}}~x_2=\bar{x}_1x_3.$$ As $|x_4|=1,$ we note that
\begin{equation}\label{equation4}
	|\tilde{z_3}(z_3)|=|\frac{x_3-z_3x_7}{1-x_4z_3}|=|\frac{\bar{x}_4x_7-z_3x_7}{1-x_4z_3}|=|x_7||\frac{\bar{x}_4(1-x_4z_3)}{1-x_4z_3}|=1~\text{for all}~z_3 \in \mathbb{D}.
\end{equation} 
By employing an analogous reasoning as in Case $1$, we can also show that $\overline{\tilde{z_2}(z_3)}\tilde{z_3}(z_3)=\tilde{z_1}(z_3)$ for all $z_3\in \mathbb D$ and hence we conclude that $(\tilde{z_1}(z_3),\tilde{z_2}(z_3),\tilde{z_3}(z_3))\in b\Gamma_{E(2;2;1,1)}$ for all $z_3\in \mathbb{D}$.
\newline \textbf{Case $3:$} Suppose that $|x_2|=1$. The proof is similar to Case $2$.
\newline We now show that $(2)\implies (1).$  To demonstrate this, we must take into account the following cases: 
\newline  \textbf{Case $1:$}  Assume that $ \tilde{\textbf{z}}^{(z_3)}\in b\Gamma_{E(2;2;1,1)}~\text{and}~\tilde{\textbf{y}}^{(z_2)}\in b\Gamma_{E(2;2;1,1)}~\text{for all}~z_2,z_3\in \mathbb{T}~\text{with}~|x_4|<1 ~\text{and}~|x_2|<1.$  
As $(\tilde{z_1}(z_3),\tilde{z_2}(z_3),\tilde{z_3}(z_3))\in b\Gamma_{E(2;2;1,1)}$ for all $z_3\in \mathbb{T}$ with $|x_4|<1$, using the characterization of $b\Gamma_{E(2;2;1,1)}$, we have 
\begin{equation}\label{tildeq1}\tilde{z_1}(z_3)=\overline{\tilde{z_2}(z_3)}\tilde{z_3}(z_3), \tilde{z_2}(z_3)=\overline{\tilde{z_1}(z_3)}\tilde{z_3}(z_3)~\text{and}~ |\tilde{z_3}(z_3)|=1~~\text{for all}~z_3\in \mathbb{T}.\end{equation} 
Thus, from \eqref{tildeq1}, we deduce that 
\begin{equation}\label{distguish12}x_5=\bar{x}_2x_7,x_1+x_5\bar{x}_4=\bar{x}_2x_3+\bar{x}_6x_7, x_1\bar{x}_4=x_3\bar{x}_6~\text{and}~x_2+x_6\bar{x}_4=\bar{x}_1x_3+\bar{x}_5x_7,x_6=\bar{x}_1x_7,x_2\bar{x}_4=x_3\bar{x}_5.\end{equation}
Since $\tilde{\textbf{y}}^{(z_2)}\in b\Gamma_{E(2;2;1,1)}~\text{for all}~z_2\in \mathbb{T}$, we apply the characterization of $b\Gamma_{E(2;2;1,1)}$ to derive the following relationship:
\begin{equation}\label{tildeq11}\tilde{y_1}(z_2)=\overline{\tilde{y_2}(z_2)}\tilde{y_3}(z_2), \tilde{y_2}(z_2)=\overline{\tilde{y_1}(z_2)}\tilde{y_3}(z_2)~\text{and}~ |\tilde{y_3}(z_2)|=1~~\text{for all}~z_2\in \mathbb{T},\end{equation} 
which leads to the conclusion that
\begin{equation}\label{distguish123}x_3=\bar{x}_4x_7,x_1+x_3\bar{x}_2=\bar{x}_4x_5+\bar{x}_6x_7, x_1\bar{x}_2=x_5\bar{x}_6~\text{and}~x_4+x_6\bar{x}_2=\bar{x}_1x_5+\bar{x}_3x_7,x_6=\bar{x}_1x_7,x_4\bar{x}_2=x_5\bar{x}_3.\end{equation}
From \eqref{distguish12} and \eqref{distguish123}, we notice that 
\begin{equation}\label{x786}
x_6\bar{x}_2=\bar{x}_1x_7\bar{x}_2=\bar{x}_1x_5.
\end{equation}
Therefore, it can be deduced from \eqref{distguish123} that $x_3=\bar{x}_4x_7.$ According to the hypothesis, we have $|\tilde{z_3}(z_3)|=1$ for all $z_3 \in \mathbb{T}$. Furthermore, considering the relationship $x_3=\bar{x}_4x_7$, we note that $$1=|\frac{\bar{x}_4x_7-z_3x_7}{1-x_4z_3}|=|x_7||\frac{\bar{x}_4-z_3}{1-x_4z_3}|=|x_7| ~\text{for~all}~z_3\in \mathbb T.$$ 
 As $|x_7|=1$, from \eqref{distguish12} and \eqref{distguish123} we conclude that  $x_2=\bar{x}_5x_7, $ and $x_1=\bar{x}_6x_7.$ This implies that $\textbf{x}\in K$. 
\newline \textbf{Case $2:$} Suppose  that $(\tilde{z_1}(z_3),\tilde{z_2}(z_3),\tilde{z_3}(z_3))\in b\Gamma_{E(2;2;1,1)}$ for all $z_3\in \mathbb{D}$ with $|x_4|=1$. It follows from the characterization of  $ b\Gamma_{E(2;2;1,1)}$ that $|\tilde{z_3}(z_3)|\leq 1$ for all $z_3\in \mathbb{D}$.  Once more by using the characterization of tetrablock, it implies that $(x_3,x_4,x_7)\in \Gamma_{E(2;2;1,1)}$.
 
 As $|x_4|=1$ and $(x_3,x_4,x_7)\in \Gamma_{E(2;2;1,1)},$  we have $x_3x_4=x_7$ and $x_3=\bar{x}_4x_7$. Since $|\tilde{z_3}(z_3)|=1$ for all $z_3 \in \mathbb{D}$ and $x_3=\bar{x}_4x_7$, we observe that $$1=|\frac{\bar{x}_4x_7-z_3x_7}{1-x_4z_3}|=|x_7||\frac{\bar{x}_4-z_3}{1-x_4z_3}|=|x_7|.$$  Because $|x_7|=1$ and $x_3x_4=x_7$, we have $|x_4|=|x_3|=1$.  As $(x_1,x_2,x_3)$, $(x_1,x_4,x_5)$ and $(x_2,x_4,x_6)\in \Gamma_{E(2;2;1,1)}$ and $|x_4|=|x_3|=1,$ it yields that $$x_1x_4=x_5,~x_1=\bar{x}_4x_5,~x_2x_4=x_6,~x_2=\bar{x}_4x_6,~x_1=\bar{x}_2x_3,~x_2=\bar{x}_1x_3,$$ which gives 
\begin{align*}
	x_5=x_1x_4=x_1\bar{x}_3x_7=\bar{x}_2x_7, x_6=x_2x_4=x_2\bar{x}_3x_7=\bar{x}_1x_7.
\end{align*}
This shows that $x\in K$.
\newline \textbf{Case $3:$} Assume  that $(\tilde{y_1}(z_2),\tilde{y_2}(z_2),\tilde{y_3}(z_2))\in b\Gamma_{E(2;2;1,1)}$ for all $z_2\in \mathbb{D}$ with $|x_2|=1$. By employing an analogous argument to that in Case $2$, we conclude that $\textbf{x} \in K.$
\newline Similarly, we can prove $(1)\Leftrightarrow(3)$ and $(1)\Leftrightarrow(4)$. This completes the proof.
\end{proof}

\begin{thm}
$K$ is homeomorphic to $\bar{\mathbb D}^3\times \mathbb T.$

\end{thm}
\begin{proof}
The map $(x_4,x_5,x_6,x_7)\to (\bar{x}_6x_7,\bar{x}_2x_7,\bar{x}_4x_7,x_4,x_5,x_6,x_7)$ is homeomorphic and hence $K$ is homeomorphic to $\bar{\mathbb D}^3\times \mathbb T.$ This completes the proof.

\end{proof}

The polynomial map $\tilde{\pi}_{{E(3;2;1,2)}}:\mathcal M_{3\times 3}(\mathbb C)\to \mathbb C^5$ defined in \eqref{pi123} is of the following form
			$$\tilde{\pi}_{{E(3;2;1,2)}}(A):=\left( a_{11},\det \left(\begin{smallmatrix} a_{11} & a_{12}\\
					a_{21} & a_{22}
				\end{smallmatrix}\right)+\det \left(\begin{smallmatrix}
					a_{11} & a_{13}\\
					a_{31} & a_{33}
				\end{smallmatrix}\right),\operatorname{det}A, a_{22}+a_{33}, \det  \left(\begin{smallmatrix}
					a_{22} & a_{23}\\
					a_{32} & a_{33}\end{smallmatrix}\right)\right).$$			
Note that $\tilde{\pi}_{{E(3;2;1,2)}}(\mathbb B)\subseteq G_{E(3;2;1,2)}$ and $\tilde{\pi}_{{E(3;2;1,2)}}(\bar{\mathbb B})\subseteq \Gamma_{E(3;2;1,2)}.$
Let $$K_1=\{\tilde{\textbf{x}}=(x_1,x_2,x_3,y_1,y_2)\in  \Gamma_{E(3;2;1,2)} :x_1=\bar{y}_2x_3, x_2=\bar{y}_1x_3, |x_3|=1\}.$$

\begin{lem}\label{Bharalidomain}
Let $\tilde{x}=(x_1,x_2,x_3,y_1,y_2)\in \mathbb{C}^5$.  Then $\tilde{\textbf{x}}\in \Gamma_{E(3;2;1,2)}$ if and only if $$\left(p_1(z),p_2(z), p_3(z)\right)\in\Gamma_{E(2;2;1,1)}$$ for all  $z\in \mathbb D$, where $p_1(z)=\frac{2x_1-zx_2}{2-y_1z},p_2(z)=\frac{y_1-2zy_2}{2-y_1z},\mbox{and}\;\; p_3(z)=\frac{x_2-2zx_3}{2-y_1z}$. 
\end{lem}

\begin{proof}
For $(s,p)\in \Gamma_{E(2;1;2)}$ and $z\in \mathbb D,$ let $\Lambda(z,(s,p))=\frac{2zp-s}{2-zs}$. It follows from [\cite{ay1}, Theorem $1.1$] that $(s,p)\in \Gamma_{E(2;1;2)}$ if and only if $|s|\leq 2$ and for all $z\in \mathbb D,$ $|\Lambda(z,(s,p))|\leq 1$. We notice from Proposition \ref{bhchh}  that  a point $\tilde{\textbf{x}}=(x_1,x_2,x_3,y_1,y_2)$ belongs to $ \Gamma_{E(3;2;1,2)}$ if and only if $\Big(\frac{y_1 - zx_2}{1 - zx_1}, \frac{y_2 - zx_3}{1 - zx_1}\Big)$ belongs to $\Gamma_{E(2;1;2)}$ for all $z \in \mathbb D.$  For all $z,w\in \mathbb D$  and  $\Big(\frac{y_1 - zx_2}{1 - zx_1}, \frac{y_2 - zx_3}{1 - zx_1}\Big)\in \Gamma_{E(2;1;2)}$, we observe that
\begin{align}\label{psi1111}
-\Lambda\left(w, \Big(\frac{y_1 - zx_2}{1 - zx_1}, \frac{y_2 - zx_3}{1 - zx_1}\Big)\right) \nonumber&=\frac{\frac{y_1-zx_2}{1-zx_1}-2w\frac{y_2-zx_3}{1-zx_1}}{2-w\frac{y_1-zx_2}{1-zx_1}}\\\nonumber&=\frac{y_1-zx_2-2wy_2+2zwx_3}{2-2zx_1-wy_1+zwx_2}\\\nonumber &=\frac{\frac{y_1-2wy_2}{2-y_1w}-z\frac{x_2-2wx_3}{2-y_1w}}{1-z\frac{2x_1-wx_2}{2-y_1w}}\\\nonumber&=\frac{p_{1}(w)-zp_3(w)}{1-zp_2(w)}\\&=\Psi(z,(p_1(w),p_2(w),p_3(w))),
\end{align}
where $p_1(w)=\frac{y_1-2wy_2}{2-y_1w},p_2(w)=\frac{2x_1-wx_2}{2-y_1w}~{\rm{and}}~p_3(w)=\frac{x_2-2wx_3}{2-y_1w}$. It follows from [\cite{ay1}, Theorem $1.1$], [\cite{bha}, Theorem $3.5$], Proposition \ref{bhchh} and \eqref{psi1111} that $\tilde{\textbf{x}}\in \Gamma_{E(3;2;1,2)}$ if and only if \begin{equation}\label{psi890}1-zp_2(w)\neq 0 ~{\rm{for ~all~}} z,w\in \mathbb D~{\rm{and }}~\sup_{(z,w)\in \mathbb D^2}|\Psi(z,(p_1(w),p_2(w),p_3(w)))|\leq 1.\end{equation} 
By [\cite{Abouhajar}, Theorem 2.4] and \eqref{psi890}, we have $\tilde{\textbf{x}}\in \Gamma_{E(3;2;1,2)}$ if and only if $$\left(p_1(z),p_2(z), p_3(z)\right)\in\Gamma_{E(2;2;1,1)}$$ for all  $z\in \mathbb D$.
 
\end{proof}

\begin{thm}\label{distinguish1}
$\pi_{E(3;2;1,2)}(\mathcal U(3)) \subseteq K_1.$
\end{thm}
\begin{proof}
Let $\tilde{\textbf{x}}\in \pi_{E(3;2;1,2)}(\mathcal U(3)),$ that is, $\tilde{\textbf{x}}= \pi_{E(3;2;1,2)}(U_2)$ for some $U_2\in \mathcal U(3).$ Then by Theorem \eqref{matix ABCD}, we notice that $\tilde{\textbf{x}}\in \Gamma_{E(3;2;1,2)}.$ From Lemma \ref{Bharalidomain}, it follows that $$\left(p_1(z)=\frac{2x_1-zx_2}{2-y_1z},p_2(z)=\frac{y_1-2zy_2}{2-y_1z}, p_3(z)=\frac{x_2-2zx_3}{2-y_1z}\right)\in\Gamma_{E(2;2;1,1)}$$ for all  $z\in \mathbb D.$ By characterization of $\Gamma_{E(2;2;1,1)}$ we observe that $|p_3(z)|\leq 1$ for all $z\in \mathbb D$ which gives 
\begin{equation}\label{cgargama1}|\frac{x_2}{2}-\frac{\bar{y}_1}{2}x_3|+|\frac{y_1}{2}-\frac{\bar{x}_2}{2}x_3|\leq 1-|x_3|^2.
\end{equation}
As $U_2\in \mathcal U(3)$ and $|x_3|=|\det(U_2)|=1$, it implies from \eqref{cgargama1} that $x_2=\bar{y}_1x_3$ and $y_1=\bar{x}_2x_3.$  As $|y_1|\leq 2$ the proof involves the following two cases.

\noindent{\bf{Case $1$}:} Suppose $|y_1|=2.$ By Proposition \ref{bhchh}, we deduce that $(y_1,y_2)\in \Gamma_{E(2;1;2)}$. Thus we obtain $y_1=z_1+z_2$ and $y_2=z_1z_2$ for some $z_1,z_2\in \bar{\mathbb D}.$ Note that $2=|y_1|\leq |z_1|+|z_2|\leq 2$ implies that $|z_1|+|z_2|=2.$ Since $z_1,z_2\in \bar{\mathbb D}$, it yields that $|z_1|=|z_2|=1.$ This shows that $(y_1,y_2)\in b\Gamma_{E(2;1;2)} $ and $y_1=\bar{y}_1y_2,$ where $b\Gamma_{E(2;1;2)} $ is the distinguish boundary of $\Gamma_{E(2;1;2)} .$ As $|p_1(z)|\leq 1$  for every $z\in \mathbb D$, characterization of $\Gamma_{E(2;2;1,1)}$ allows us to obtain  $x_1=\frac{\bar{y}_1x_2}{4}, x_2=y_1x_1$. Thus, we can deduce that $x_1=\frac{\bar{y}_1x_2}{4}=\bar{y}_2x_3$ which suggests  that $\tilde{\textbf{x}}\in K_1.$

\noindent{\bf{Case $2$:}} 
Assume that $|y_1|<2.$ Then we have $2-y_1z\neq 0$ for all $z\in \bar{\mathbb D}.$ Since $x_2=\bar{y}_1x_3$  and $|x_3|=1$, implies that $|p_3(z)|=1$ for all $z\in \mathbb T.$ Let $z$ be a point of $\mathbb{T}$. Let $\{z_n\}$ be a sequence in $\mathbb{D}$ such that $\lim_{n\to \infty}z_n=z$. By the Lemma \ref{Bharalidomain}, it follows that $(p_1(z_n),p_2(z_n),p_3(z_n))\in \Gamma_{E(2;2;1,1)}$ for all $n$. Hence $(p_1(z),p_2(z),p_3(z))=\lim_{n\to \infty}(p_1(z_n),p_2(z_n),p_3(z_n))$ belongs to $\Gamma_{E(2;2;1,1)}$. Thus $(p_1(z),p_2(z),p_3(z))\in \Gamma_{E(2;2;1,1)}$ for all $z\in \mathbb{T}$. Since $|p_3(z)|=1$ for $z\in \mathbb{T}$, so by [\cite{Abouhajar}, Theorem 2.4], it follows that $p_1(z)=\overline{p_2(z)}p_3(z)$ for all $z\in \mathbb T$ and hence we conclude that
\begin{equation}\label{tildez11}x_1=\bar{y}_2x_3, x_2=\bar{y}_1x_3~{\rm{and}}~x_1\bar{y}_1=\bar{y}_2x_2.\end{equation}
From \eqref{tildez11}, we deduce that $x_1=\bar{y}_2x_3$ and hence we have $\tilde{\textbf{x}}\in K_1.$ This completes the proof.
\end{proof}

\begin{proof}[Another proof]
Let $\tilde{\textbf{x}}\in \pi_{E(3;2;1,2)}(\mathcal U(3)),$ that is, $\tilde{\textbf{x}}= \pi_{E(3;2;1,2)}(U_2)$ for some $U_2\in \mathcal U(3).$ By Theorem \eqref{matix ABCD}, we notice that $\tilde{\textbf{x}}\in \Gamma_{E(3;2;1,2)}.$ From Lemma \ref{Bharalidomain}, it follows that $$\left(p_1(z)=\frac{2x_1-zx_2}{2-y_1z},p_2(z)=\frac{y_1-2zy_2}{2-y_1z}, p_3(z)=\frac{x_2-2zx_3}{2-y_1z}\right)\in\Gamma_{E(2;2;1,1)}$$ for all  $z\in \mathbb D.$ By characterization of $\Gamma_{E(2;2;1,1)}$ we observe that $|p_3(z)|\leq 1$ for all $z\in \mathbb D$ which gives 
\begin{equation}\label{cgargama1}|\frac{x_2}{2}-\frac{\bar{y}_1}{2}x_3|+|\frac{y_1}{2}-\frac{\bar{x}_2}{2}x_3|\leq 1-|x_3|^2.
\end{equation}
As $U_2\in \mathcal U(3)$ and $|x_3|=|\det(U_2)|=1$, it implies from \eqref{cgargama1} that $x_2=\bar{y}_1x_3$ and $y_1=\bar{x}_2x_3.$ We deduce from Proposition \ref{bhchh}  that   $\Big(\frac{y_1 - zx_2}{1 - zx_1}, \frac{y_2 - zx_3}{1 - zx_1}\Big)$ belongs to $\Gamma_{E(2;1;2)}$ for all $z \in \mathbb D.$ Hence, it follows from the characterization of $\Gamma_{E(2;1;2)}$ that $|\frac{y_2 - zx_3}{1 - zx_1}|\leq 1$ for all $z \in \mathbb D$. Again, by the characterization of tetrablock we get $(y_2,x_1,x_3)\in\Gamma_{E(2;2;1,1)}$.  As $(y_2,x_1,x_3)\in\Gamma_{E(2;2;1,1)}$, we have 
\begin{equation}\label{cgargama12}|y_2-\bar{x}_1x_3|+|x_1-\bar{y}_2x_3|\leq 1-|x_3|^2.
\end{equation}
Since $|x_3|=1,$ it follows from \eqref{cgargama12} that $y_2=\bar{x}_1x_3$ and $x_1=\bar{y}_2x_3$. Thus, we conclude that $\tilde{\textbf{x}}\in K_1.$ This completes the proof.
\end{proof}

We prove the following theorem using the same argument as in Theorem \eqref{relation bw}. Consequently, we omit the proof.
\begin{thm}\label{relation bw1}
	Let $\tilde{\textbf{x}}\in \mathbb{C}^5$. Then the following conditions are equivalent

	\begin{enumerate}
		\item $\tilde{\textbf{x}}\in K_1=\{(x_1,x_2,x_3,y_1,y_2)\in \Gamma_{E(3;2;1,2)}:x_1=\bar{y}_2x_3, x_2=\bar{y}_1x_3, |x_3|=1\}.	$	
		\item 	$\begin{cases}
			(p_1(z),p_2(z),p_3(z))\in b\Gamma_{E(2;2;1,1)}~\text{for ~all}~z\in \mathbb{T}~\text{with}~|y_1|<2 \\
			(p_1(z),p_2(z),p_3(z))\in b\Gamma_{E(2;2;1,1)}~\text{for~all}~z\in \mathbb{D}~\text{with}~|y_1|=2.
		\end{cases}$
		\end{enumerate}
\end{thm}

\begin{thm}
$K_1$ is homeomorphic to $\bar{\mathbb D}\times \overline{D(0,2)}\times \mathbb T,$ where $$\overline{D(0,2)}=\{z\in \mathbb C: |z|\leq 2\}.$$

\end{thm}
\begin{proof}
The map $(y_1,y_2,x_3)\to (\bar{y}_2x_3,\bar{y}_1x_3,x_3,y_1,y_2)$ is homeomorphic and hence $K_1$ is homeomorphic to $\bar{\mathbb D}\times \overline{D(0,2)}\times \mathbb T.$ This completes the proof.

\end{proof}

We establish the connection between $K$ and $K_1$ in the following theorem.

\begin{thm}
Suppose $\textbf{x}=(x_1,x_2,x_3,x_4,x_5,x_6,x_7)\in \mathbb C^7.$ Then $\textbf{x}\in K$  if and only if $$(x_1,x_3+\eta x_5,\eta x_7,x_2+\eta x_4,\eta x_6)\in K_1 ~{\rm{for~ all }}~\eta \in \mathbb T.$$

\end{thm}
\begin{proof}
Suppose $\textbf{x}\in K$. Then we have $x_1=\bar{x}_6x_7, x_3=\bar{x}_4x_7,x_5=\bar{x}_2x_7 ~{\rm{and}}~|x_7|=1.$ This suggests that $(x_1,x_3+\eta x_5,\eta x_7,x_2+\eta x_4,\eta x_6)\in K_1 ~{\rm{for~ all }}~\eta \in \mathbb T.$ It is also simple to verify the converse implication. Therefore, we omit the proof.

\end{proof} 
According to the Theorem \ref{distinguish} and Theorem \ref{relation bw}, it is interesting to determine whether $K$ is equal to the distinguished boundary of $ \Gamma_{E(3;2;1,2)}$. Similarly, from Theorem \ref{distinguish1} and Theorem \ref{relation bw1}, one may ask whether $K_1$  is equal to the distinguished boundary of $\Gamma_{E(3;2;1,2)}$ . From above discussions, we propose the following conjecture:

\noindent{\textbf{Conjecture}:} The subset $K$ (respectively, $K_1$) of $\Gamma_{E(3;3;1,1,1)}$ (respectively, $\Gamma_{E(3;2;1,2)}$) is distinguished boundary for $\mathcal A(\Gamma_{E(3;3;1,1,1)})$ (respectively, $\mathcal A(\Gamma_{E(3;2;1,2)})$).

\section{Necessary condition of Schwarz Lemma for $\Gamma_{E(3;3;1,1,1)}$ and $\Gamma_{E(3;2;1,2)}$ }
The following theorem is the classical Nevanlinna-Pick problem, which  was solved in $1916$ by Pick \cite{AglerM, AglerMY}.
\begin{thm}\label{AMY}[Theorem $1.81$, \cite{AglerMY}]
Let $z_1,\ldots,z_n\in \mathbb D$ and $\lambda_1,\ldots,\lambda_n\in \mathbb C.$ Then there exists a function $\phi \in H^{\infty}(\mathbb D)$ such that $\|\phi\|_{\infty}\leq 1 $ and $\phi(z_i)=\lambda_i$ for $i=1,\ldots,n$ if and only if  the Pick matrix is non-negative definite, that is,
$$\Big(\Big(\frac{1-\bar{\lambda}_i\lambda_j}{1-\bar{z}_iz_j}\Big)\Big)_{i,j=1}^{n}\geq 0.$$ 

\end{thm}
The standard Schwarz lemma is a particular case of the Nevanlinna-Pick interpolation problem. In solving a two-point interpolation problem for analytic functions from the open unit disc $\mathbb D$ into itself, the standard Schwarz lemma provides a necessary and sufficient condition. It is defined as follows: 

\noindent{\textbf{The Standard Schwarz Lemma :}}
Let $\lambda_{0}\in \mathbb D\setminus \{0\}$ and $z_0\in \mathbb D.$ Then there exists an analytic function $f : \mathbb D\to \mathbb D$ such that $f(0) = 0 $ and $f(\lambda_0) = z_0$ if and only if $|z_0| \leq |\lambda_0|.$

In this subsection, we discuss the necessary conditions for the Schwarz lemma for $\Gamma_{E(3;3;1,1,1)}$ and $\Gamma_{E(3;2;1,2)}.$
\subsection{Necessary conditions for Schwarz Lemma for $\Gamma_{E(3;3;1,1,1)}$}

In order to provide the necessary conditions for Schwarz Lemma for $\Gamma_{E(3;3;1,1,1)}$, we prove the following lemma.

\begin{lem}\label{phiD}
Let $\phi:\mathbb D \to \Gamma_{E(3;3;1,1,1)}$ be a analytic function.  Suppose $\phi$ maps some point of $\mathbb D$ into $G_{E(3;3;1,1,1)}$. Then $\phi(\mathbb D)\subset G_{E(3;3;1,1,1)}.$
\end{lem}
\begin{proof}
Suppose that $\phi(\lambda)=(\phi_1(\lambda),\ldots,\phi_7(\lambda))$ for $\lambda\in \mathbb D$ and $\phi(\lambda_0)\in G_{E(3;3;1,1,1)}$ for some $\lambda_0\in \mathbb D.$ Then by characterization of $G_{E(3;3;1,1,1)},$  we observe that \begin{equation}\label{philam}\sup_{(z_1,z_2)\in \bar{\mathbb D}^2}|\Psi^{(1)}(z_1,z_2,\phi(\lambda_0))|<1~{\rm{ and}}~(\phi_2(\lambda_0),\phi_4(\lambda_0),\phi_6(\lambda_0))\in G_{E(2;2;1,1)}.
\end{equation}
By characterization of $\Gamma_{E(3;3,1,1,1)},$ we have an analytic map $$\tilde{\phi}:=(\phi_2,\phi_4,\phi_6):\mathbb D\to \Gamma_{E(2;2,1,1)}.$$ It follows from [Lemma $3.4$ \cite{Abouhajar}] that $\tilde{\phi}(\mathbb D)\subseteq G_{E(2;2;1,1)}.$ Fix $(z_1,z_2)\in \bar{\mathbb D}^2.$ The function $\lambda \to \Psi^{(1)}(z_1,z_2,\phi(\lambda))$ is well defined and analytic on $\mathbb D$; it maps $\mathbb D$ into $\bar{\mathbb D}$ and $\lambda_0$ into $\mathbb D$ and hence, it follows from the Schwarz-Pick lemma that it maps  $\mathbb D$ into $\mathbb D.$ Now fix $\lambda \in \mathbb D$, the map $\Psi^{(1)}(.,.,\phi(\lambda))$ maps $\bar{\mathbb D}^2$ to $\mathbb D$ and hence again by characterization of $G_{E(3;3;1,1,1)},$ $\phi(\mathbb D)\subseteq G_{E(3;3;1,1,1)}.$ This completes the proof.

\end{proof}

\begin{thm}\label{schwarz}
%
Let $\lambda_0 \in \mathbb D\setminus \{0\}$ and let $\textbf{x}=(x_1,\ldots,x_7)\in G_{E(3;3;1,1,1)}.$ Then, in the following, $(1)\Leftrightarrow (1^{\prime}),  (2^{\prime})\Leftrightarrow (2)\Leftrightarrow (5) \Leftrightarrow (6) \Leftrightarrow (7),$ $(3^{\prime})\Leftrightarrow (3) \Leftrightarrow (8) \Leftrightarrow (9) \Leftrightarrow (10)$, $(4^{\prime})\Leftrightarrow (4)\Leftrightarrow (11) \Leftrightarrow (12) \Leftrightarrow (13),$ $(1^{\prime})\implies (2^{\prime})$, $(1^{\prime})\implies (3^{\prime})$and $(1^{\prime})\implies (4^{\prime}):$
\begin{enumerate}
\item[$(1)$] There exists an analytic function $\phi:=(\phi_1,\ldots,\phi_7):\mathbb D\to \Gamma_{E(3;3;1,1,1)}$ such that $\phi(0)=(0,0,0,0,0,0,0)$ and $\phi(\lambda_0)=\textbf{x}.$
\item[ $(1^{\prime})$] There exists an analytic function $\phi:\mathbb D\to G_{E(3;3;1,1,1)}$ such that $\phi(0)=(0,0,0,0,0,0,0)$ and $\phi(\lambda_0)=\textbf{x}.$

\item[$(2)$] For every $z_1\in \bar{\mathbb D},$ there exists an analytic function $\phi^{(z_1)}:\mathbb D\to \Gamma_{E(2;2;1,1)}$ such that $\phi^{(z_1)}(0)=(0,0,0)$ and $\phi^{(z_1)}(\lambda_0)=(\tilde{x}_1(z_1),\tilde{x}_2(z_1),\tilde{x}_3(z_1)).$ 

\item[$(2^{\prime})$] For every  $z_1\in \bar{\mathbb D},$ there exists an analytic function $\phi^{(z_1)}:\mathbb D\to G_{E(2;2;1,1)}$ such that $\phi^{(z_1)}(0)=(0,0,0)$ and $\phi^{(z_1)}(\lambda_0)=(\tilde{x}_1(z_1),\tilde{x}_2(z_1),\tilde{x}_3(z_1)).$ 

\item[$(3)$] For every $z_2\in \bar{\mathbb D},$ there exists an analytic function $\phi_1^{(z_2)}:\mathbb D\to \Gamma_{E(2;2;1,1)}$ such that $\phi_1^{(z_2)}(0)=(0,0,0)$ and $\phi_1^{(z_2)}(\lambda_0)=(\tilde{y}_1(z_2),\tilde{y}_2(z_2),\tilde{y}_3(z_2)).$ 

\item[$(3^{\prime})$] For every $z_2\in \bar{\mathbb D},$ there exists an analytic function $\phi_1^{(z_2)}:\mathbb D\to G_{E(2;2;1,1)}$ such that $\phi_1^{(z_2)}(0)=(0,0,0)$ and $\phi_1^{(z_2)}(\lambda_0)=(\tilde{y}_1(z_2),\tilde{y}_2(z_2),\tilde{y}_3(z_2))$.

\item[$(4)$] For every $z_3\in \bar{\mathbb D},$ there exists an analytic function $\phi_2^{(z_3)}:\mathbb D\to \Gamma_{E(2;2;1,1)}$ such that $\phi_2^{(z_3)}(0)=(0,0,0)$ and $\phi_2^{(z_3)}(\lambda_0)=(\tilde{z}_1(z_3),\tilde{z}_2(z_3),\tilde{z}_3(z_3)).$ 

\item[$(4^{\prime})$] For every $z_3\in \bar{\mathbb D},$
 there exists an analytic function $\phi_2^{(z_3)}:\mathbb D\to G_{E(2;2;1,1)}$ such that $\phi_2^{(z_3)}(0)=(0,0,0)$ and $\phi_2^{(z_3)}(\lambda_0)=(\tilde{z}_1(z_3),\tilde{z}_2(z_3),\tilde{z}_3(z_3))$.
\item[$(5)$] $\max\{G_1,G_2\}\leq |\lambda_0|,$
where $G_1=\sup_{z_1\in \bar{\mathbb D}}\frac{|\tilde{x}_1(z_1)-\overline{\tilde{x}_2(z_1)}\tilde{x}_3(z_1)|+|\tilde{x}_1(z_1)\tilde{x}_2(z_1)-\tilde{x}_3(z_1)|}{1-|\tilde{x}_2(z_1)|^2},$ \\$G_2=\sup_{z_1\in \bar{\mathbb D}}\frac{|\tilde{x}_2(z_1)-\overline{\tilde{x}_1(z_1)}\tilde{x}_3(z_1)|+|\tilde{x}_1(z_1)\tilde{x}_2(z_1)-\tilde{x}_3(z_1)|}{1-|\tilde{x}_1(z_1)|^2}.$

\item [$(6)$] Either $$\big|\tilde{x}_2(z_1)\big|\leq \big|\tilde{x}_1(z_1)\big|~{\rm{for ~all}}~z_1\in \bar{\mathbb{D}}~{\rm {and}}~ G_1\leq |\lambda_0|,$$
or 
$$\big|\tilde{x}_1(z_1)\big|\leq \big|\tilde{x}_2(z_1)\big|~{\rm{for ~all}}~z_1\in \bar{\mathbb{D}}~{\rm {and}}~ G_2\leq |\lambda_0|.$$

\item [$(7)$] For every $z_1\in \bar{\mathbb D},$ there exists a $2 \times 2$ matrix valued function $F^{(z_1)}$ defined on $\mathbb{D}$ is in the Schur class such that
$$F^{(z_1)}(0)=\left(\begin{smallmatrix} 0 &\star\\0 & 0\end{smallmatrix}\right)~{\rm{and}}~ F^{(z_1)}(\lambda_0)=\left(\begin{smallmatrix} a_{11}^{(z_1)} &a_{12}^{(z_1)} \\a_{21}^{(z_1)}  & a_{22}^{(z_1)} \end{smallmatrix}\right),$$
where $\tilde{x}_1(z_1)=a_{11}^{(z_1)} ,\tilde{x}_2(z_1)=a_{22}^{(z_1)}, \tilde{x}_3(z_1)=\det(F^{(z_1)}(\lambda_0)).$

\item[$(8)$] $\max\{H_1,H_2\}\leq |\lambda_0|,$
where $H_1=\sup_{z_2\in \bar{\mathbb D}}\frac{|\tilde{y}_1(z_2)-\overline{\tilde{y}_2(z_2)}\tilde{y}_3(z_2)|+|\tilde{y}_1(z_2)\tilde{y}_2(z_2)-\tilde{y}_3(z_2)|}{1-|\tilde{y}_2(z_2)|^2},$ \\$H_2=\sup_{z_2\in \bar{\mathbb D}}\frac{|\tilde{y}_2(z_2)-\overline{\tilde{y}_1(z_2)}\tilde{y}_3(z_2)|+|\tilde{y}_1(z_2)\tilde{y}_2(z_2)-\tilde{y}_3(z_2)|}{1-|\tilde{y}_1(z_2)|^2}.$

\item [$(9)$] Either $$\big|\tilde{y}_2(z_2)\big|\leq \big|\tilde{y}_1(z_2)\big|~{\rm{for ~all}}~z_2\in \bar{\mathbb{D}}~{\rm {and}}~ H_1\leq |\lambda_0|,$$
or 
$$\big|\tilde{y}_1(z_2)\big|\leq \big|\tilde{y}_2(z_2)\big|~{\rm{for ~all}}~z_2\in \bar{\mathbb{D}}~{\rm {and}}~ H_2\leq |\lambda_0|.$$

\item [$(10)$] For every $z_2\in \bar{\mathbb D},$ there exists a $2 \times 2$ matrix valued  function $F_{1}^{(z_2)}$ defined on $\mathbb{D}$ is in the Schur class such that
$$F_{1}^{(z_2)}(0)=\left(\begin{smallmatrix} 0 &\star\\0 & 0\end{smallmatrix}\right)~{\rm{and}}~ F_{1}^{(z_2)}(\lambda_0)=\left(\begin{smallmatrix} b_{11}^{(z_2)} &b_{12}^{(z_2)} \\b_{21}^{(z_2)}  & b_{22}^{(z_2)} \end{smallmatrix}\right),$$
where $\tilde{y}_1(z_2)=b_{11}^{(z_2)} ,\tilde{y}_2(z_2)=b_{22}^{(z_2)}, \tilde{y}_3(z_2)=\det(F_{1}^{(z_2)}(\lambda_0)).$

\item[$(11)$] $\max\{I_1,I_2\}\leq |\lambda_0|,$
where $I_1=\sup_{z_3\in \bar{\mathbb D}}\frac{|\tilde{z}_1(z_3)-\overline{\tilde{z}_2(z_3)}\tilde{z}_3(z_3)|+|\tilde{z}_1(z_3)\tilde{z}_2(z_3)-\tilde{z}_3(z_3)|}{1-|\tilde{z}_2(z_3)|^2},$ \\$I_2=\sup_{z_3\in \bar{\mathbb D}}\frac{|\tilde{z}_2(z_3)-\overline{\tilde{z}_1(z_3)}\tilde{z}_3(z_3)|+|\tilde{z}_1(z_3)\tilde{z}_2(z_3)-\tilde{z}_3(z_3)|}{1-|\tilde{z}_1(z_3)|^2}.$

\item [$(12)$] Either $$\big|\tilde{z}_2(z_3)\big|\leq \big|\tilde{z}_1(z_3)\big|~{\rm{for ~all}}~z_3\in \bar{\mathbb{D}}~{\rm {and}}~ I_1\leq |\lambda_0|,$$
or 
$$\big|\tilde{z}_1(z_3)\big|\leq \big|\tilde{z}_2(z_3)\big|~{\rm{for ~all}}~z_3\in \bar{\mathbb{D}}~{\rm {and}}~ I_2\leq |\lambda_0|.$$

\item [$(13)$] For every $z_3\in \bar{\mathbb D},$ there exists a $2 \times 2$  matrix valued function $F_{2}^{(z_3)}$ defined on $\mathbb{D}$ is in the Schur class such that
$$F_{2}^{(z_3)}(0)=\left(\begin{smallmatrix} 0 &\star\\0 & 0\end{smallmatrix}\right)~{\rm{and}}~ F_{2}^{(z_3)}(\lambda_0)=\left(\begin{smallmatrix} c_{11}^{(z_3)} &c_{12}^{(z_3)} \\c_{21}^{(z_3)}  & c_{22}^{(z_3)} \end{smallmatrix}\right),$$
where $\tilde{z}_1(z_3)=c_{11}^{(z_3)} ,\tilde{z}_2(z_3)=c_{22}^{(z_3)}, \tilde{z}_3(z_3)=\det(F_{2}^{(z_3)}(\lambda_0)).$

\end{enumerate}

\end{thm}

\begin{proof}
$(1)\Leftrightarrow (1^{\prime})$ follows from the Lemma \ref{phiD}. We now prove that $(1^{\prime})\implies (2^{\prime}).$ For fixed $z_1\in \bar{\mathbb{D}},$ define $\psi^{(z_1)}:G_{E(3;3;1,1,1)}\to G_{E(2;2;1,1)}$ by $$\psi^{(z_1)}(y_1,\ldots,y_7)=\left(\frac{y_2-z_1y_3}{1-y_1z_1},\frac{y_4-z_1y_5}{1-y_1z_1},\frac{y_6-z_1y_7}{1-y_1z_1}\right).$$  Set $\phi^{(z_1)}:=\psi ^{(z_1)}\circ \phi.$ Then we have $\phi^{(z_1)}:\mathbb D\to G_{E(2;2;1,1)}$ and $\phi^{(z_1)}(0)=(0,0,0)$  and $\phi^{(z_1)}(\lambda_0)=(\tilde{x}_1(z_1),\tilde{x}_2(z_1),\tilde{x}_3(z_1))$.  This shows that $(1^{\prime})$ implies $(2^{\prime}).$ Similarly, $(1^{\prime}) \Rightarrow (3^{\prime})$ and $(1^{\prime}) \Rightarrow (4^{\prime})$.

From [Lemma $3.4$, \cite{Abouhajar}], we observe that $(2)\Leftrightarrow (2^{\prime})$. By using the [Theorem $1.2$, \cite{Abouhajar}], we conclude that $(2)\Leftrightarrow (5) \Leftrightarrow (6) \Leftrightarrow (7).$ Using the similar arguments as above, we prove that $(1)\implies (3)\Leftrightarrow (3^{\prime})\Leftrightarrow (8) \Leftrightarrow (9) \Leftrightarrow (10)$ and $(1)\implies (4)\Leftrightarrow (4^{\prime})\Leftrightarrow (11) \Leftrightarrow (12) \Leftrightarrow (13).$

\end{proof}

\begin{thm}
Let $\lambda_0 \in \mathbb D\setminus \{0\}$ and let $\textbf{x}=(x_1,\ldots,x_7)\in G_{E(3;3;1,1,1)}.$ Suppose that there exists a $3\times 3$ matrix valued hololmorphic function $F=((F_{ij}))_{i,j=1}^{3}$ in $ \mathcal S_1(\mathbb C^3,\mathbb C^3)$ such that $F(0)=\left(\begin{smallmatrix} 0 & \star &\star \\0 & 0 & \star \\ 0 & 0 &0\end{smallmatrix}\right)$ and $F(\lambda_0)=A=((a_{ij}))_{i,j=1}^{3},$ where \small{$$x_1=a_{11}, x_2=a_{22}, x_3=\det \left(\begin{smallmatrix} a_{11} & a_{12}\\
					a_{21} & a_{22}
				\end{smallmatrix}\right), x_4=a_{33}, x_5=\det \left(\begin{smallmatrix}
					a_{11} & a_{13}\\
					a_{31} & a_{33}
				\end{smallmatrix}\right), x_6=\det  \left(\begin{smallmatrix}
					a_{22} & a_{23}\\
					a_{32} & a_{33}\end{smallmatrix}\right) ~{\rm{and}}~x_7=\det A.$$} 
Then there exists an analytic function $\phi:\mathbb D\to G_{E(3;3;1,1,1)}$ such that $\phi(0)=(0,0,0,0,0,0,0)$ and $\phi(\lambda_0)=\textbf{x}.$

\end{thm}

\begin{proof}
Suppose that $F$ satisfies the above conditions. Then, by Theorem \ref{matix barAAAA}, the function $\phi$ is defined by $$\phi=\left(F_{11}, F_{22}, \det \left(\begin{smallmatrix} F_{11} & F_{12}\\
F_{21} & F_{22}
\end{smallmatrix}\right), F_{33}, \det \left(\begin{smallmatrix}
F_{11} & F_{13}\\
F_{31} & F_{33}
\end{smallmatrix}\right), \det \left(\begin{smallmatrix}
F_{22} & F_{23}\\
F_{32} & F_{33}\end{smallmatrix}\right),\det F\right),$$ is an analytic function from $\mathbb D $ to $\Gamma_{E(3;3;1,1,1)}$ satisfies $\phi(0)=(0,0,0,0,0,0,0)$ and $\phi(\lambda_0)=\textbf{x}$. By Lemma \ref{phiD}, we have $\phi(\mathbb D)\subset G_{E(3;3;1,1,1)}. $ This completes the proof.
\end{proof} 
We now prove the necessary conditions for Schwarz's lemma for $\Gamma_{E(3;2;1,2)}$. 
\begin{lem}\label{phiD1}
Let $\psi:\mathbb D \to \Gamma_{E(3;2;1,2)}$ be a analytic function.  Suppose $\psi$ maps a point of $\mathbb D$ into $G_{E(3;2;1,2)}$. Then $\psi(\mathbb D)\subset G_{E(3;2;1,2)}.$
\end{lem}
\begin{proof}
Let $\psi(\lambda)=(\psi_1(\lambda),\psi_2(\lambda),\psi_3(\lambda),\psi_4(\lambda),\psi_5(\lambda))$ for $\lambda\in \mathbb D$ and $\psi(\lambda_0)\in G_{E(3;2;1,2)}$ for some $\lambda_0\in \mathbb D.$ Then by characterization of $G_{E(3;2;1,2)},$  we have \begin{equation}\label{philam}\sup_{z\in \bar{\mathbb D}}|\Psi_3(z,\psi(\lambda_0))|<1~{\rm{ and}}~(\psi_4(\lambda_0),\psi_5(\lambda_0))\in G_{E(2;1;2)}.
\end{equation} 
It follows from the characterization of $\Gamma_{E(3;2,1,2)}$ that  $\tilde{\psi}:=(\psi_4,\psi_5):\mathbb D\to \Gamma_{E(2;1,2)}$ is a analytic function. However,  from proof of Lemma \ref{analytic retr 1} we deduce  that  $(\frac{\psi_4}{2},\frac{\psi_4}{2},\psi_5)\in \Gamma_{E(2;2;1,1)}$ if and only if $(\psi_4,\psi_5)\in \Gamma_{E(2;1;2)}.$  Let $\theta_1$ be the map from $\Gamma_{E(2;1;2)}$ to $\Gamma_{E(2;2;1,1)} $ defined by
	$$\theta_1(s,p)=(\frac{s}{2},\frac{s}{2}, p).$$ Set $\tilde{\psi}_1=\theta_1\circ \tilde{\psi}.$ Then $\tilde{\psi}_1:\mathbb D\to \Gamma_{E(2;2,1,1)}$ is a analytic function. It indicates from [Lemma $3.4$ \cite{Abouhajar}] that $\tilde{\psi}_1(\mathbb D)\subseteq G_{E(2;2;1,1)}$ and hence  $\tilde{\psi}(\mathbb D)\subseteq G_{E(2;1;2)}.$ Fix $z\in \bar{\mathbb D}.$ The function $\lambda \to \Psi_3(z,\psi(\lambda))$ is well defined and analytic on $\mathbb D$, maps $\mathbb D$ into $\bar{\mathbb D}$ and $\lambda_0$ into $\mathbb D$, hence it implies from the Schwarz-Pick lemma that it maps all of $\mathbb D$ into $\mathbb D.$ Now fix $\lambda \in \mathbb D$, the map $\Psi_3(z,\psi(\lambda))$ maps $\bar{\mathbb D}$ to $\mathbb D$ and hence again by characterization of $G_{E(3;2;1,2)},$ $\psi(\mathbb D)\subset G_{E(3;2;1,2)}.$ This completes the proof.

\end{proof}
We state the necessary conditions for Schwarz's lemma for $G_{E(3;2;1,2)}$ and its proof follows by using the same argument as in Theorem \ref{schwarz}. Consequently, we omit the proof.
\begin{thm}
Let $\lambda_0 \in \mathbb D\setminus \{0\}$ and let $\tilde{\textbf{x}}=(x_1,x_2,x_3,y_1,y_2)\in G_{E(3;2;1,2)}.$ Then, in the following, $(1)\Leftrightarrow (1^{\prime}), (2)\Leftrightarrow (2^{\prime})\Leftrightarrow (3) \Leftrightarrow (4) \Leftrightarrow (5)$ and $ (1)\implies (2)$:
\begin{enumerate}
\item[$(1)$] There exists an analytic function $\psi:=(\psi_1,\ldots,\psi_5):\mathbb D\to \Gamma_{E(3;2;1,2)}$ such that $\psi(0)=(0,0,0,0,0)$ and $\psi(\lambda_0)=\tilde{\textbf{x}}.$
\item[ $(1^{\prime})$] There exists an analytic function $\psi:\mathbb D\to G_{E(3;2;1,2)}$ such that $\psi(0)=(0,0,0,0,0)$ and $\psi(\lambda_0)=\tilde{\textbf{x}}.$

\item[$(2)$] For every $z\in \bar{\mathbb D},$ there exists an analytic function $\psi^{(z)}:\mathbb D\to \Gamma_{E(2;2;1,1)}$ such that $\psi^{(z)}(0)=(0,0,0)$ and $\psi^{(z)}(\lambda_0)=(p_1(z),p_2(z),p_3(z)),$ where $$p_1(z)=\frac{2x_1-zx_2}{2-y_1z},p_2(z)=\frac{y_1-2zy_2}{2-y_1z}~{\rm{and}}~p_3(z)=\frac{x_2-2zx_3}{2-y_1z}.$$

\item[$(2^{\prime})$] For every $z\in \bar{\mathbb D},$ there exists an analytic function $\psi^{(z)}:\mathbb D\to G_{E(2;2;1,1)}$ such that $\psi^{(z)}(0)=(0,0,0)$ and $\psi^{(z)}(\lambda_0)=(p_1(z),p_2(z),p_3(z))$.

%
%
%

\item[$(3)$] $\max\{\tilde{G}_1,\tilde{G}_2\}\leq |\lambda_0|,$
where $\tilde{G}_1=\sup_{z\in \bar{\mathbb D}}\frac{|p_1(z)-\overline{p_2(z)}p_3(z)|+|p_1(z)p_2(z)-p_3(z)|}{1-|p_2(z)|^2},$ \\$\tilde{G}_2=\sup_{z\in \bar{\mathbb D}}\frac{|p_2(z)-\overline{p_1(z)}p_3(z_1)|+|p_1(z)p_2(z)-p_3(z)|}{1-|p_1(z)|^2}.$

\item [$(4)$] Either $$\big|p_2(z)\big|\leq \big|p_1(z)\big|~{\rm{for~all}}~z\in \bar{\mathbb D}~{\rm {and}}~ \tilde{G}_1\leq |\lambda_0|$$
or 
$$\big|p_1(z)\big|\leq \big|p_2(z)\big|~{\rm{for~all}}~z\in \bar{\mathbb D}~{\rm {and}}~ \tilde{G}_2\leq |\lambda_0|$$

\item [$(5)$] For every $z\in \bar{\mathbb D},$ there exists a $2 \times 2$ matrix valued function $\tilde{F}^{(z)}$ defined on $\mathbb{D}$ is in the Schur class such that
$$\tilde{F}^{(z)}(0)=\left(\begin{smallmatrix} 0 &\star\\0 & 0\end{smallmatrix}\right)~{\rm{and}}~ F^{(z)}(\lambda_0)=\left(\begin{smallmatrix} d_{11}^{(z)} &d_{12}^{(z)} \\d_{21}^{(z)}  & d_{22}^{(z)} \end{smallmatrix}\right),$$
where $d_{11}^{(z)} =p_1(z),d_{22}^{(z)} =p_2(z), p_3(z)=\det(\tilde{F}^{(z)}(\lambda_0)).$ \end{enumerate}
\end{thm}

\begin{thm}
Let $\lambda_0 \in \mathbb D\setminus \{0\}$ and let $\tilde{\textbf{x}}=(x_1,x_2,x_3,y_1,y_2)\in G_{E(3;2;1,2)}.$ Suppose that there exists a $3\times 3$ matrix valued function $F=((F_{ij}))_{i,j=1}^{3}$ in $\mathcal S_1(\mathbb C^3,\mathbb C^3)$ with $F(0)=\left(\begin{smallmatrix} 0 & \star &\star \\0 & 0 & \star \\ 0 & 0 &0\end{smallmatrix}\right)$ and $F(\lambda_0)=B,$ where \small{$$x_1=b_{11},  x_2=\det \left(\begin{smallmatrix} b_{11} & b_{12}\\
					b_{21} & b_{22}
				\end{smallmatrix}\right)+\det \left(\begin{smallmatrix}
					b_{11} & b_{13}\\
					b_{31} & b_{33}
				\end{smallmatrix}\right), x_3=\det B, y_1=b_{22}+b_{33},~{\rm{and}}~y_2=\det  \left(\begin{smallmatrix}
					b_{22} & b_{23}\\
					b_{32} & b_{33}\end{smallmatrix}\right) .$$} 
Then there exists an analytic function $\psi:\mathbb D\to G_{E(3;2;1,2)}$ such that $\psi(0)=(0,0,0,0,0)$ and $\psi(\lambda_0)=\tilde{\textbf{x}}.$

\end{thm}

\textsl{Acknowledgements:}
The second-named author is supported by the research project of SERB with ANRF File Number: CRG/2022/003058 and the third named author thankfully acknowledges the financial support provided by Mathematical Research Impact Centric Support (MATRICS) grant, File no: MTR/2020/000493, by the Science and Engineering Research Board (SERB), Department of Science and Tech-nology (DST), Government of India. 
\vskip-1cm


\end{document}